\documentclass[10pt]{article}

\usepackage{amsmath,amsthm,amsfonts,amssymb,bm,bbm,enumerate, graphicx, mathtools,color}
\usepackage{tcolorbox}
\usepackage[english]{babel}
\usepackage[utf8]{inputenc}
\usepackage{fancyhdr}
\usepackage{caption}
\usepackage{subcaption}
\usepackage[toc,page]{appendix}
\usepackage{csquotes,multirow,yfonts}

\newcommand{\pr}[1]{\mathbb{P}\!\left(#1\right)}
\newcommand{\E}[1]{\mathbb{E}\!\left[#1\right]}
\newcommand{\estart}[2]{\mathbb{E}_{#2}\!\left[#1\right]}
\newcommand{\prstart}[2]{\mathbb{P}_{#2}\!\left(#1\right)}
\newcommand{\prnstart}[3]{\mathbb{P}^{#3}_{#2}\!\left(#1\right)}
\newcommand{\prcond}[3]{\mathbb{P}_{#3}\!\left(#1\;\middle\vert\;#2\right)}
\newcommand{\econd}[2]{\mathbb{E}\!\left[#1\;\middle\vert\;#2\right]}

\newcommand{\prb}[1]{\mathbf{P}\!\left(#1\right)}
\newcommand{\Eb}[1]{\mathbf{E}\!\left[#1\right]}

\newcommand{\prcondb}[3]{\mathbf{P}_{#3}\!\left(#1\;\middle\vert\;#2\right)}
\newcommand{\econdb}[2]{\mathbf{E}\!\left[#1\;\middle\vert\;#2\right]}

\newtheorem{theorem}{Theorem}[section]
\newtheorem{lem}[theorem]{Lemma}
\newtheorem{prop}[theorem]{Proposition}

\newtheorem{defn}[theorem]{Definition}

\newtheorem{rmk}[theorem]{Remark}

\newtheorem{assn}[theorem]{Assumption}

\parskip \medskipamount
\def\N{\mathbb{N}}

\def\C{\mathbb{C}}
\def\bP{\mathbb{P}}
\def\S{\mathbb{S}}

\def\bP{\mathbb{P}}

\def\X{X^{\text{exc}}}
\renewcommand{\epsilon}{\varepsilon}

\def\T{\mathcal{T}}

\def\C{\mathcal{C}}

\newcommand{\bPb}{\mathbf{P}}

\newcommand{\Loop}{\textsf{Loop}}

\newcommand{\diam}{\textsf{Diam}}
\newcommand{\diamR}{\textsf{Diam}_{\mathrm{res}}}

\renewcommand{\Xi}{X^{\infty}}

\newcommand{\Height}{\textsf{Height}}

\newcommand{\Tai}{T_{\infty}}
\newcommand{\p}{\textsf{par}}
\newcommand{\Ynu}{\tilde{Y}_i}

\newcommand{\Ynui}{\tilde{Y}_{j_i}}

\newcommand{\Znv}{Z_j}
\newcommand{\Znu}{Z_i}

\newcommand{\car}{\mathsf{c_{\alpha}^{R}}}

\def\ER{Erd\"os-R\'enyi }
\def\V{V^{\text{dec}}}
\def\BB{B^{\text{dec}}}
\def\BT{B^{\text{dec}}} 

\newcommand{\Sr}{\mathsf{Spine}_r}

\def\S2r{\text{Spine}_{\frac{1}{2}r}}

\def\sER{s^{ER}}

\newcommand{\sad}{\mathsf{s_{\alpha}^d}}
\newcommand{\tad}{\mathsf{t^d_{\alpha}}}
\newcommand{\ttad}{\mathsf{T^d_{\alpha}}}
\def\sav{\mathsf{s_{\alpha}^v}}

\def\fv{\mathsf{f_{\alpha}^v}}
\def\Sc{S^{(c)}}

\def\Tc{T^{(c)}}

\def\Nfr{N^f_r}
\def\Nf2r{N^f_{\frac{r}{2}}}

\def\HER{H^{ER}}
\def\1a{a_1}
\def\2a{a_2}
\def\3a{a_3}
\def\4a{a_4}

\def\N{\mathbb{N}}

\def\C{\mathbb{C}}
\def\bP{\mathbb{P}}
\def\S{\mathbb{S}}

\def\X{X^{\text{exc}}}

\renewcommand{\epsilon}{\varepsilon}

\def\T{\mathcal{T}}

\def\Loop'{\textsf{Loop'}}
\def\C{\mathcal{C}}

\def\diam{\textsf{Diam}}

\def\Sr2{S^{(\frac{1}{2}r)}_{\sigma}}

\def\Height{\textsf{Height}}

\def\T{\mathcal{T}}

\def\Loop'{\textsf{Loop'}}
\def\Loop{\textsf{Loop}}
\def\C{\mathcal{C}}

\def\diam{\textsf{Diam}}

\def\Sr2{S^{(\frac{1}{2}r)}_{\sigma}}

\def\Ti{T_{\infty}}

\def\Ref{R_{\text{eff}}}

\def\Height{\textsf{Height}}
\def\Heightdec{\textsf{Height}^{\text{dec}}}
\def\HER{H^{\text{dec}}}

\def\ER{Erd\"os-R\'enyi }
\def\TERa{\T_{\infty}^{\text{dec}}}
\def\TER{\T^{\text{dec}}}
\def\TERh{(\T^{[h,2h]})^{\text{dec}}}
\def\tTER{\tilde{\T}^{\text{dec}}}

\def\Sr{\text{Spine}_r}
\def\S2r{\text{Spine}_{\frac{1}{2}r}}

\def\dER{d^{\text{dec}}}
\def\dERa{\mathsf{d^{\text{vol}}}}
\def\rER{\rho^{\text{dec}}}
\def\1a{a_1}
\def\2a{a_2}
\def\3a{a_3}
\def\4a{a_4}

\def\rhoERA{\rho^{\text{dec}}}
\def\sER{s^{\text{dec}}}

\def\dg{d^{\text{dec}}}

\def\ea{w_{\alpha}}
\def\Da{\mathsf{d^{\text{dis}}}}
\def\ka{s_{\alpha}}
\def\ba{\beta_6}
\def\bat{\beta_7}
\def\batt{\beta_8}

\def\saR{\mathsf{s_{\alpha}^R}}
\def\taR{\mathsf{t^R_{\alpha}}}

\def\Sc{S^{(k)}}

\def\Tc{T^{(k)}}

\def\Nfr{N^f_r}
\def\Nf2r{N^f_{\frac{r}{2}}}
\def\Dn{D^{\nu}_n}

\def\esim{\overset{e}{\sim}}

\def\Trin{T^{\Delta}_n}

\def\ddis{\mathsf{d^{\text{dis}}}}
\def\dspec{\mathsf{d^{\text{spec}}}}

\allowdisplaybreaks
\begin{document}
\title{Random walks on decorated Galton--Watson trees}
\author{Eleanor Archer\thanks{Tel Aviv University. Email: {eleanora@mail.tau.ac.il}. This work was supported by a JSPS short term predoctoral fellowship, and EPSRC grant number EP/N509796/1.}}
\maketitle

\begin{abstract}
In this article, we study a simple random walk on a decorated Galton--Watson tree, obtained from a Galton-Watson tree by replacing each vertex of degree $n$ with an independent copy of a graph $G_n$ and gluing the inserted graphs along the tree structure. We assume that there exist constants $d,R \geq 1, v< \infty$ such that the diameter, effective resistance across and volume of $G_n$ respectively grow like $n^{\frac{1}{d}}, n^{\frac{1}{R}}, n^v$ as $n \to \infty$. We also assume that the underlying Galton--Watson tree is critical with offspring tails $\xi(x)$ decaying like $cx^{-\alpha-1}$ as $x \to \infty$ for some constant $c$ and some $\alpha \in (1,2)$. We establish the fractal dimension, spectral dimension, walk dimension and simple random walk displacement exponent for the resulting metric space as functions of $\alpha, d, R$ and $v$, along with bounds on the fluctuations of these quantities.
\end{abstract}



\section{Introduction}

The purpose of this paper is to study random walks on a general class of so-called decorated Galton--Watson trees, informally obtained from classical Galton--Watson trees by replacing each vertex of degree $n$ with a graph of boundary size $n$, and gluing these graphs along the underlying tree structure as indicated in Figure \ref{fig:dec tree pic} below. In the planar case, we can alternatively construct these graphs from the looptrees of \cite{RSLTCurKort} by ``filling in'' the loops. Such decorated trees naturally describe the structure of a variety of critical statistical mechanics models on random planar maps; examples include critical percolation clusters \cite{CurKortUIPTPerc, RichierIICUIHPT}, Fortuin-Kasteleyn models \cite{BerestyckiLaslierRayFKRPM}, and quadrangulations with skewness \cite{BaurRichQuadSkew}. More generally, this structure applies to critical Boltzmann maps in the dense phase \cite{RichierMapBoundaryLimit}, which are conjectured to describe a wide range of statistical physics models on Brownian surfaces.

We are interested in the regime where the underlying Galton--Watson tree has a critical offspring distribution $\xi$ satisfying
\begin{equation}\label{eqn:offspring tails}
    \xi(k) \sim ck^{-\alpha-1}
\end{equation}
as $k \rightarrow \infty$, for some $\alpha \in (1,2)$. The case $\alpha = 2$ is less interesting since in this case the metric space structure of the decorated tree does not differ substantially from that of the original undecorated tree (similarly for any finite variance offspring distribution). Analogous results will also hold on incorporating a slowly-varying function into the offspring tails, but for ease of notation we do not prove this here.

\begin{figure}[h]
\begin{subfigure}{0.5\textwidth}
\includegraphics[width=7cm]{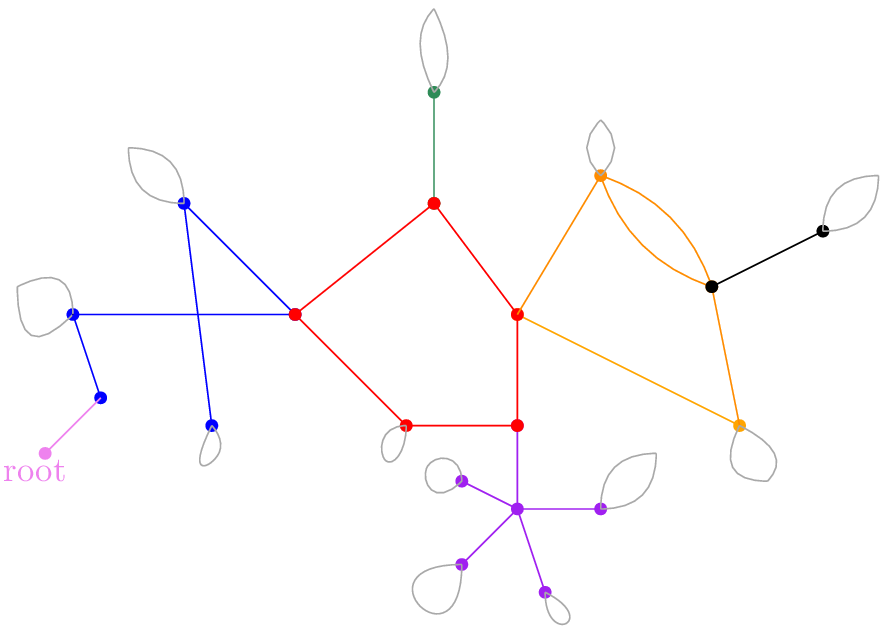}
\centering
\end{subfigure}\hfill
\begin{subfigure}{.5\textwidth}
\includegraphics[width=6cm]{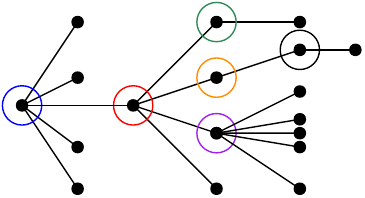}
\end{subfigure}\hfill
\caption{An example of a rooted decorated tree, and its underlying Galton--Watson tree.}\label{fig:dec tree pic}
\end{figure}

In order to understand the long-term behaviour of the random walk, we will work with a critical Galton--Watson tree conditioned to survive, also known as Kesten's tree. Under the above assumptions on the offspring distribution, we denote this tree $\Tai$ and its root $\rho$. Given a family of connected, potentially random graphs $(G_n)_{n \geq 1}$, where each $G_n$ is conditioned to have $n$ boundary vertices, we construct our decorated tree from $\Tai$ by replacing each vertex $v \in \Tai$ with an independent copy of $G_{\deg v}$. We then uniquely identify each boundary vertex of $G_n$ with an edge incident to $v$, in such a way as to respect any planarity restrictions, and then if $v \sim v'$ in $\Tai$ we glue their corresponding graphs at the two vertices identified with the edge $(v,v')$. We use the following rooting convention: if $v$ is the root of $\TERa$, we replace $v$ with an independent copy of $G_{\deg v + 1}$, rather than $G_{\deg v}$, and we add an extra edge emanating from the ``spare'' boundary vertex of $G_{\deg v + 1}$ to a new extra vertex, and designate this extra vertex to be the root vertex of $\TERa$, denoted $\rER$. The graph $G_n$ may be random (e.g. an \ER graph on $n$ vertices) or deterministic (e.g. the complete graph on $n$ vertices); in the random case we sample independently for each vertex, conditionally on the boundary size.  We call the resulting structure $\TERa$. We further assume that the inserted graphs are each endowed with a measure and a metric, and extend these to a measure $\V$ and a metric $\dg$ on $\TERa$ by superposing the measures and adding distances along tree branches in the natural way. We will make the construction precise in Section \ref{sctn:model def}.

Our aim is to establish the exponents governing the behaviour of a simple random walk $(X_n)_{n \geq 0}$ on $\TERa$ in terms of $\alpha$ and the relevant exponents for the inserted graphs; in particular, to identify the following two quantities (and show that the limits exist):
\begin{enumerate}[(i)]
\item The spectral dimension, $d_S = -2 \lim_{n \rightarrow \infty} \frac{\log (p_{2n}(x,x))}{\log n}$,
\item The displacement exponent, $d_{\text{dis}} = \lim_{n \rightarrow \infty} \frac{\log (\sup_{k \leq n} \dg(X_0, X_k))}{\log n}$.
\end{enumerate}
Here $p_n(x,y) = \frac{\prstart{X_n = y}{x}}{\deg y}$ is the transition density of the simple random walk; it is easy to see that for recurrent graphs, the limits do not depend on the choice of $x$ or $X_0$. The quantity $d_w = \frac{1}{d_{\text{dis}}}$ is also known as the walk dimension, and can be naturally compared to the fractal dimension of $\TERa$, denoted $d_f$, given by
\[
d_f = \lim_{r \rightarrow \infty} \frac{\log \left(\V (\BB(x,r))\right)}{\log r},
\]
where $\BB(x,r)$ denotes the open ball of radius $r$ centred at $x$, defined with respect to the metric $\dg$. We will see that, as is commonly the case for sufficiently homogeneous graphs, the relation 
\[
d_S = \frac{2d_f}{d_w}
\]
holds on $\TERa$. This will hold as a consequence of verifying the conditions of \cite{KumMisumiHKStronglyRecurrent}.

It is well-known that these simple random walk exponents are determined by two key properties of the ambient graph: its effective resistance and volume growth. Most of the paper is therefore devoted to a study of these two properties of $\TERa$, and as a result, we also determine $d_f$ as defined above, which is of independent interest. 

Given $n \in \N$, we define the following notation.
\begin{itemize}
\item $G_n$ denotes a copy of the inserted graph conditioned on having $n$ boundary vertices,
\item $(U^{(1)}_n, U^{(2)}_n)$ denotes a uniform pair of distinct points on the boundary of $G_n$,
\item $d_n$ is a given metric on $G_n$ (e.g. the graph metric), and $d^U_n = d_n(U^{(1)}_n, U^{(2)}_n)$,
\item $\Ref$ denotes effective resistance on $G_n$ when each edge has conductance $1$, and $R^U_n = \Ref(U^{(1)}_n, U^{(2)}_n)$,
\item $\diam (G_n) = \sup_{x, y \in G_n} d_n(x,y)$ and $\diam_{\text{res}} (G_n) = \sup_{x, y \in G_n} \Ref(x,y)$,
\item 
 $V_n$ is a pre-defined measure on $G_n$ (e.g. $V_n(x) = \deg x$).
\end{itemize}

For example, $G_n$ could be a uniformly chosen plane tree with $n$ leaves, $d_n$ could be the graph metric and $V_n$ could be the counting measure on its vertices. Then effective resistance is equal to graph distance, so that $d^U_n$ and $R^U_n$ both correspond to the graph distance between two uniformly chosen distinct leaves.

In what follows, we assume that the decorated tree $\TERa$ is defined on the probability space $(\mathbf{\Omega}, \mathbf{\mathcal{F}}, \bPb)$, and let $\pr{\cdot}$ denote the (quenched) law of a simple random walk on $\TERa$, started from the root. We will make the following assumption on the quantities defined above.

\begin{assn}\label{assn:sec}
\ \\
\textbf{(D) Metric growth.} There exist $d \geq 1$ and $\epsilon > 0$ such that
\begin{align*}
\inf_{n \geq 1} \prb{d_n^U (G_n) \geq n^{\frac{1}{d}}} >0, \hspace{1cm} \sup_{n \geq 1} \prb{\diam (G_n) \geq \lambda n^{\frac{1}{d}}} = {O(\lambda^{-\frac{\alpha + \epsilon}{\alpha - 1}})} \text{ as } \lambda \to \infty.
\end{align*}
\textbf{(R) Resistance growth.} There exist $R \geq 1$ and $\epsilon >0$ such that
\begin{align*}
\inf_{n \geq 1} \prb{R_n^U (G_n) \geq n^{\frac{1}{R}}} >0, \hspace{1cm} \sup_{n \geq 1} \prb{\diamR (G_n) \geq \lambda n^{\frac{1}{R}}} ={O(\lambda^{-\frac{\alpha + \epsilon}{\alpha - 1}})} \text{ as } \lambda \to \infty.
\end{align*}
\textbf{(V) Volume growth.} There exist $v \geq 1$ and $\epsilon > 0$  such that
\begin{align*}
\inf_{n \geq 1} \prb{V_n (G_n) \geq n^{v}} >0, \hspace{1cm} \sup_{n \geq 1} \prb{V_n (G_n) \geq \lambda n^{v}} = O(\lambda^{-\frac{\alpha + \epsilon}{v}}) \text{ as } \lambda \to \infty.
\end{align*}
\end{assn}

\begin{rmk}
\begin{enumerate}
\item Assumption \ref{assn:sec}(D) ensures that, if $T_{[n,2n]}$ is a Galton--Watson tree conditioned to have progeny in $[n,2n]$, with decorated version $\TER_{[n,2n]}$, then the sequence 
$\left( n^{-\left[ (1 - \frac{1}{\alpha}) \vee \frac{1}{\alpha d}\right]} \diam (\TER_{[n,2n]}) \right)_{n \geq 1}$ is a tight sequence (with an extra scaling factor of $\log n$ when $d(\alpha - 1)=1$). We do not expect this to be optimal in the regime where $d(\alpha - 1) < 1$, in which case tail decay of $o(\lambda^{-\alpha d})$ is probably sufficient, as this is enough to ensure that the maximal diameter of any individual graph in $\TER_n$ is of order $n^{\frac{1}{\alpha d}}$, with high probability.

In the regime $d(\alpha - 1)>1$ one certainly needs tail decay that is $o(\lambda^{-\frac{\alpha}{\alpha - 1}})$, again to control the maximal diameter in the tree. We could also remove the $\epsilon$ from Assumption \ref{assn:sec}(D) by employing a similar strategy to that carried out in \cite[Section 3]{MarzoukStableSnake} to deal with larger values of $n^{-\frac{1}{d}} \diam (G_n)$, however we have chosen to keep the $\epsilon$ as this makes our proofs shorter.
\item Assumption \ref{assn:sec}(V) is essentially what is required to ensure that inserted graphs of high volume coincide with high degree vertices in $\Tai$. (We could technically weaken it very slightly to an integrability condition). If this is not satisfied, then the typical order of the total volume of a finite decorated Galton-Watson tree $\TER_n$ additionally depends on the precise tail decay of the probabilities in $(V)$, and not just on $v$.
\end{enumerate}
\end{rmk}

Under Assumption \ref{assn:sec}, we define the following key exponents.

\begin{tcolorbox}[colback=white]
\begin{center}
$\sad = d(\alpha - 1), \ \saR = R(\alpha - 1), \ \sav= \frac{\alpha - 1}{v}$ and $\fv = \frac{\alpha}{v}$.
\end{center}
\end{tcolorbox}\label{box:decorated exponents paper}

We will see in Section \ref{sctn:finite tree bounds dec} that, roughly speaking, for a typical vertex $v$ on the backbone of $\Tai$ and a typical vertex $u$ not on the backbone of $\Tai$, with corresponding graphs $G(v)$ and $G(u)$ respectively, there exist constants $c, C \in (0, \infty)$ such that for all large enough $x$ we have $cx^{-\sad} \leq \prb{\diam (G(v)) \geq x} \leq Cx^{-(\sad \wedge 1)}$ and $cx^{-\fv} \leq \prb{\V (G(u)) \geq x} \leq Cx^{-(\fv \wedge 1)}$. Similar results hold for the other exponents.

We define the decorated volume exponent as
\begin{equation}\label{eqn:volume exponent def}
\dERa = {\frac{\alpha(\sad \wedge 1)}{(\alpha - 1)(\fv \wedge 1)}}.
\end{equation}
We also let $\BT(\rER, r)$ denote a ball of radius $r$ around the root of $\TERa$ with respect to the metric $\dg$. The main volume growth theorem is as follows.

\begin{theorem}[Volume growth/fractal dimension]\label{thm:dec vol growth main}
Under Assumption \ref{assn:sec} and \eqref{eqn:offspring tails}, $d_f (\TERa) = \dERa$, $\bPb$-almost surely. Moreover, define the constant 
$$b_1 = \begin{cases} 0 &\text{ if } \sad, \fv \neq 1, \\
-\dERa &\text{ if } \sad=1, \fv \neq 1, \\
1 &\text{ if } \sad \neq 1, \fv = 1, \\
1-\dERa &\text{ if } \sad=\fv = 1.
\end{cases}$$
Then:
\begin{enumerate}[(i)]
\item There exist $c_1<\infty, a_1>0$ depending only on $d,\alpha,v$ such that for all $r, \lambda>1$,
$$\prb{\V(\BT(\rER, r)) \notin [r^{\dERa} (\log r)^{b_1}\lambda^{-1}, r^{\dERa}(\log r)^{b_1} \lambda]} \leq c_1 \lambda^{-a_1}.$$
\item There exists a finite, explicit constant $\beta_1 = \beta_1(d,\alpha,v)$ such that, $\bPb$-almost surely,
\begin{align*}
r^{\dERa} (\log r)^{-\beta_1} \leq \V(\BT(\rER, r)) \leq r^{\dERa} (\log r)^{\beta_1}
\end{align*}
for all sufficiently large $r$.
\end{enumerate}
\end{theorem}

These logarithmic fluctuations in part $(ii)$ are not optimal in many cases, and in particular can often be improved to log-logarithmic when inserting deterministic graphs (although not always, $\Tai$ itself being an example where the upper fluctuations are genuinely logarithmic). However, although we make some comments on how the arguments can potentially be fine-tuned at appropriate parts of the proof, our emphasis here is on determining the correct leading terms for the volume growth, rather than the optimal fluctuations.

\begin{figure}[h]
\begin{subfigure}{.5\textwidth}
\includegraphics[height=4.9cm]{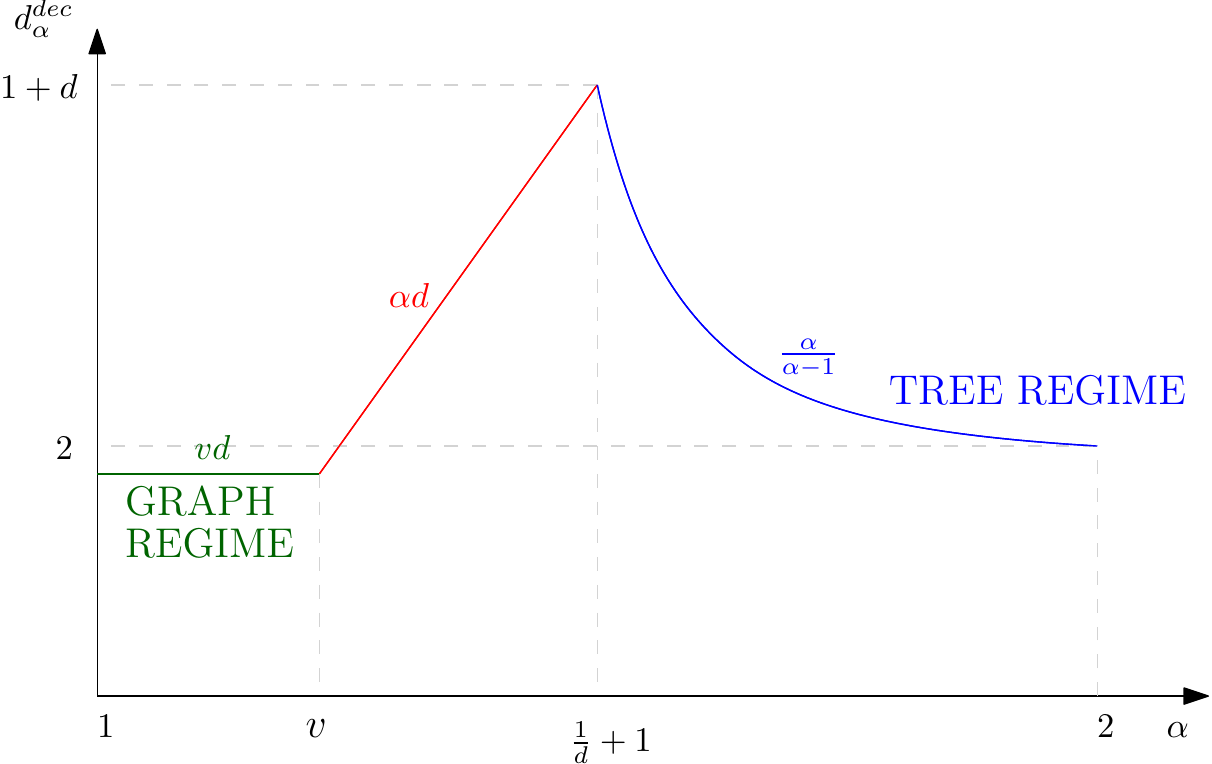}
\centering
\subcaption{If $\frac{1}{d}+1 \geq v$.}
\end{subfigure}
\begin{subfigure}{.5\textwidth}
\includegraphics[height=4.9cm]{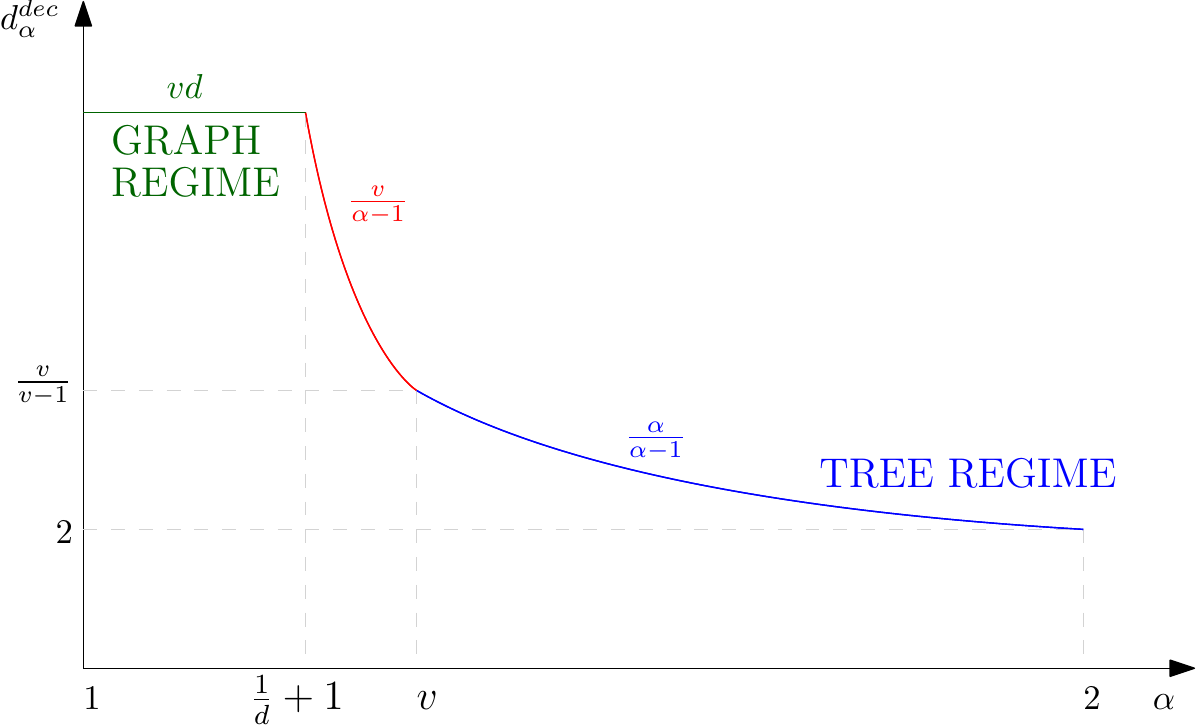}
\centering
\subcaption{If $\frac{1}{d}+1 \leq v$.}
\end{subfigure}
\caption{Different phases of the decorated volume exponent. Note that we do not necessarily see all three phases for one given model.}\label{fig:dec vol exp graphs}
\end{figure}


The factor $\frac{\alpha}{\alpha - 1}$ appearing in \eqref{eqn:volume exponent def} is the fractal dimension of the underlying tree $\Tai$. The factors of $\sad \wedge 1$ and $\fv \wedge 1$ appearing in \eqref{eqn:volume exponent def} reflect the way that distances and volumes add up differently along branches and in subtrees in the decorated version $\TERa$.

We have graphed the volume exponent in Figure \ref{fig:dec vol exp graphs}, viewing $v$ and $d$ as fixed and varying $\alpha$. There are two cases for the graph, depending on which of $\sad$ and $\fv$ exceeds one ``first''. In both cases, there are up to three regimes. The case where both of these exponents exceed $1$ can be thought of as the ``tree regime'': in this case the relevant tails on the inserted graphs are not heavy enough to impact the volume exponent, so we see the same exponent appearing as for an undecorated tree. The case where both of the exponents are less than $1$ can be thought of as the ``graph regime'', and we lose the dependence on $\alpha$. This reflects the fact that as the offspring tails get heavier (i.e. as $\alpha \downarrow 1$), it is easier for a finite critical Galton--Watson tree to be large by having one vertex of macroscopic degree (cf \cite[Proposition 3.6]{RSLTCurKort}), so that we essentially just see one macroscopic copy of the inserted graph in a finite decorated tree. In the case of a decorated Kesten's tree, we essentially just see a one-dimensional sequence of graphs glued along the backbone of $\Tai$. As $\alpha \uparrow 2$, however, the vertex degrees become more balanced, and the contribution from any one single vertex is less significant, so we regain some tree structure and eventually recover it entirely once the distance and volumes across typical inserted graphs have finite expectation.

We can also fix $\alpha$ and consider different sequences of graphs with differing values of $v$ and $d$. As can be seen in \eqref{eqn:volume exponent def}, $\dERa$ increases linearly with $d$ (respectively $v$) up to the point at which the expected diameter (respectively volume) of a graph inserted at a typical vertex becomes finite, after which point we lose the dependence on $d$ (respectively $v$).

%
We can use similar considerations to those discussed above to either establish the volume growth exponents with respect to the resistance metric, or otherwise add up resistance contributions along the backbone and along paths in subtrees to compare resistance to the graph distance. Again, there are different regimes depending on whether resistance across a typical backbone vertex has finite expectation or not: as a result, we will also see a factor of $\saR \wedge 1$ in the exponents below. We can then combine the resistance and volume estimates using results of \cite{KumMisumiHKStronglyRecurrent} to identify the random walk exponents.

Although the volume results hold quite generally, for the purpose of understanding a simple random walk on $\TERa$ we are specifically interested in the measure on $\TERa$ given by $V (x) = \deg x$ for all $x \in \TERa$. (In the construction outlined above, this is obtained by superposing the measures given by $V_n (x) = \deg x$ for $x \in G_n$).

We also set 
\begin{align}\label{eqn:spec dis exponent def}
\begin{split}
\dspec &= \frac{2\alpha (\saR \wedge 1)}{\alpha (\saR \wedge 1) + (\alpha - 1)(\fv \wedge 1)}, \\
\ddis &= \frac{(\saR \wedge 1)(\alpha -1)(\fv \wedge 1)}{(\alpha -1)(\fv \wedge 1)(\sad \wedge 1) + \alpha(\saR \wedge 1)(\sad \wedge 1)}.
\end{split}
\end{align}

The exponents may look complicated, but similarly to $\dERa$ they can be obtained from the corresponding exponents for $\Tai$ by adding extra terms which reflect how distance, volume and resistance add up along branches of $\TERa$, and in smaller subtrees. We will see below that they respectively give the quenched spectral dimension of $\TERa$, and the quenched displacement exponent. On $\Tai$, these quantities are respectively $\frac{2\alpha}{2\alpha - 1}$ and $\frac{\alpha - 1}{2\alpha - 1}$.

In what follows, we let $(X_n)_{n \geq 0}$ denote a simple random walk on $\TERa$ started at $\rER$, $\tau_r$ its exit time from a ball of radius $r$ with respect to the decorated metric $\dg$, and $p_n(x,y) = \frac{\prstart{X_n=y}{x}{}}{\deg y}$ its transition density. Our first result shows that $(X_n)_{n \geq 0}$ is $\bPb$-almost surely recurrent; as for discrete critical Galton--Watson trees, this essentially follows from the fact that the underlying tree $\Tai$ has a unique path to infinity.

\begin{theorem}\label{thm:recurrence}
$\bPb$-almost surely under Assumption \ref{assn:sec} and \eqref{eqn:offspring tails}, $(X_n)_{n \geq 0}$ is recurrent.
\end{theorem}

We establish the following quenched results on the random walk exponents using \cite[Proposition 1.3]{KumMisumiHKStronglyRecurrent}.

\begin{theorem}[Quenched random walk exponents]\label{thm:dec main quenched RW}
Under Assumption \ref{assn:sec} and \eqref{eqn:offspring tails}, $\bPb$-almost surely,
\begin{align*}
d_S (\TERa) &= \dspec, \hspace{2cm}
d_{\text{dis}} (\TERa) = \ddis.
\end{align*}

Moreover, there exist finite constants $\beta_2, \ldots, \beta_5$, depending only on $\alpha, d, R$ and $v$, such that, $\bPb$-almost surely, the following statements hold.
\begin{enumerate}[a)]
\item There exists $N< \infty$ such that $$n^{\frac{-\dspec}{2}} (\log n)^{-\beta_2} \leq p_{2n}(\rhoERA, \rhoERA) \leq n^{\frac{-\dspec}{2}} (\log n)^{\beta_2}$$ for all $n \geq N$.
\item Let $\tau_r$ denote the exit time of a ball of radius $r$ defined with respect to the metric $\dg$. Then there exists $R_1 < \infty$ such that for all $r \geq R_1$, $$r^{\frac{1}{\ddis}} (\log r)^{-\beta_3} \leq \E{\tau_r} \leq r^{\frac{1}{\ddis}} (\log r)^{\beta_3}.$$ Moreover, there $\bP$-almost surely exists $R_2< \infty$ such that, for all $r \geq R_2$, $$r^{\frac{1}{\ddis}} (\log r)^{-\beta_4} \leq \tau_r \leq r^{\frac{1}{\ddis}} (\log r)^{\beta_4}.$$
\item $\bP$-almost surely, there exists $N< \infty$ such that $$n^{\ddis} (\log n)^{-\beta_5} \leq \sup_{k \leq n} \dER (\rhoERA, X_k) \leq n^{\ddis} (\log n)^{\beta_5}$$ for all $n \geq N$.
\item Let $W_n = \{X_0, X_1, \ldots, X_n\}$ and $S_n = \V (W_n)$. Then, $\bP$-almost surely, $$\lim_{n \to \infty} \frac{\log S_n}{\log n} = \frac{\dspec}{2}.$$
%
%
\end{enumerate}
\end{theorem}

The annealed results follow similarly from \cite[Proposition 1.4]{KumMisumiHKStronglyRecurrent}.

\begin{theorem}[Annealed random walk results]\label{thm:dec main ann RW}
There exist $\ba, \bat, \batt < \infty$, deterministic functions of $\alpha, d, R, v$, such that under Assumption \ref{assn:sec} and \eqref{eqn:offspring tails}, we have that:
\begin{enumerate}[a)]
\item There exist constants $c_1 >0, c_2 < \infty, \gamma_1 > 0$ such that for all $r \geq 1$, $c_1 r^{\frac{1}{\ddis}}(\log r)^{\ba} \leq \Eb{\estart{\tau_r}{\rhoERA}{}}$, and $\Eb{\left(\estart{\tau_r}{\rhoERA}{}\right)^{\gamma_1}} \leq c_2 r^{\frac{\gamma_1}{\ddis}} (\log r)^{\gamma_1 \ba}$.
\item There exist constants $c_3 > 0, c_4 < \infty, \gamma_2 > 0$ such that $c_3 n^{\frac{-\dspec}{2}} (\log n)^{\bat} \leq \Eb{p_{2n}(\rhoERA, \rhoERA)}$ for all $n \geq 1$, and $\Eb{\left(p_{2n}(\rhoERA, \rhoERA)\right)^{\gamma_2}} \leq c_3 n^{\frac{-\gamma_2 \dspec}{2}} (\log n)^{\gamma_2\bat}$.
\item There exist constants $c_5 > 0, c_6 < \infty, \gamma_3 > 0$ such that $c_5n^{\ddis} (\log n)^{\batt} \leq \Eb{ \estart{\dER (\rhoERA, X_n)}{\rhoERA}{}}$ for all $n \geq 1$, and $\Eb{ \left(\estart{\dER (\rhoERA, X_n)}{\rhoERA}{}\right)^{\gamma_3}} \leq c_6n^{\gamma_3 \ddis} (\log r)^{\gamma_3 \batt}$.
\end{enumerate}
If $\saR, \sad, \fv \neq 1$, then $\ba = \bat = \batt = 0$. Otherwise, we will give the precise values in Section \ref{sctn:RW exponents}.
\end{theorem}


\begin{rmk}
In general it is not possible to get annealed upper bounds without taking lower powers, since these quantities are infinite on the underlying tree $\Tai$. This is because the expected volume of a unit ball is infinite in this case, as established by Croydon and Kumagai \cite{CroyKumRWGWTreeInfiniteVar}. This does not mean that it is always infinite in the decorated case (e.g. see \cite[Theorem 1.2]{BjornStef} for the corresponding result for discrete looptrees), but one must insert graphs that sufficiently ``spread out'' different branches of the tree. 
\end{rmk}

As mentioned earlier, the decorated tree structure considered in this paper arises naturally in the context of critical statistical physics models on random planar maps, such as a critical percolation cluster on the UIPT \cite{CurKortUIPTPerc}. Such a critical percolation cluster is believed to rescale to the $\frac{7}{6}$-stable map \cite[Section 5.4]{BerCurMierPerconTriang}, which is also believed to have a topological structure somewhat resembling that of a Sierpinski gasket. In fact these connections hold more generally for stable maps in the dense phase, which are believed to describe a range of statistical physics models on random planar maps, have gasket-type structures, and were shown to have a decorated tree structure by Richier \cite[Theorem 1.2]{RichierMapBoundaryLimit}. In Section \ref{sctn:examples} we verify that the Sierpinski triangle does indeed satisfy Assumption \ref{assn:sec}, indicating that this viewpoint is relevant for understanding a random walk on the incipient infinite cluster. We also consider some other examples of graphs to insert, such as critical \ER graphs and dissections of polygons, and in each case verify Assumption \ref{assn:sec} and compute the exponents considered above.

For example, for the Sierpinski triangle it turns out that we can take $d = 1$, $R = \frac{\log 2}{\log 5 - \log 3}, v=\frac{\log 3}{\log 2}$ and plug into the theorems above.

Since scaling limits of critical trees and looptrees are now well-understood \cite{AldousCRTIII, LeG2005randomtreesandapplications, DuqContourLimit, RSLTCurKort}, it is natural to wonder if one can construct general scaling limits for decorated Galton--Watson trees. In the regime where $d(\alpha - 1) < 1$, it will be clear from Proposition \ref{prop:decorated height bound without log term prog} that large graphs persist in the scaling limit and in this regime it has been shown in a recent work of S\'enizergues, Stef\`ansson and Stufler \cite{SenStefStufDecStableTrees} that the Gromov--Hausdorff--Prokhorov scaling limit can be constructed by gluing scaling limits of the inserted graphs along an appropriate tree structure. In the case $d(\alpha - 1) > 1$, Proposition \ref{prop:decorated height bound without log term} indicates that the Gromov--Hausdorff--Prokhorov scaling limit is plausibly a stable tree, endowed with a measure which may be supported on either the hubs or the leaves of the tree (or both), depending on the value of the parameter $v$. Moreover, we anticipate that in this regime, the techniques of \cite{DavidResForms} will apply to show that the simple random walk on $\TERa$ rescales to a limiting diffusion on the stable tree with this speed measure. Interestingly, if the hubs play a role the limiting measure will be singular with respect to the canonical mass measure on the stable tree, which is supported on the leaves, and the limiting diffusion will experience trapping at hubs, which contrasts with Brownian motion and classically studied FIN diffusions on stable trees, which spend almost all their time in the leaves.

In the regime $\sad < 1$, instead of decorating trees, we also expect that one would obtain the same results on decorating a Boltzmann dissection of the polygon $P_n$ with boundary length $n$, in the case where the face weights of the dissection enjoy the same asymptotics as the offspring distribution $\xi$ considered in this paper. Just as we can imagine constructing $\TERa$ by filling in the loops of looptrees, we can imagine decorating a dissected polygon by inserting into each face $f$ a graph with boundary length $\deg f$. The metric space structure of such a dissection is very close to that of a looptree (see Figure \ref{fig:BoltzmannDiss} and \cite[Proposition 4.5]{RSLTCurKort}) and it was shown in \cite[Corollary 1.3]{RSLTCurKort} that these looptrees and dissections have the same scaling limit. In the regime $\sad <1$, when the metric space scaling factors are sufficiently large, we anticipate that we should recover the same exponents for this model of decorated dissected polygons. However, we have not proved this in this article.

We conclude the introduction by commenting briefly on the case where $\xi([k, \infty)) = O(k^{-2})$ as $k \rightarrow \infty$ (note that this may be finite or infinite variance). It was shown in \cite[Theorem 1.2]{KortRichBoundaryRPMLooptrees} and \cite[Theorem 13]{CurHaasKortCRTDiss} that the metric space scaling limit of any discrete looptree with offspring distribution in the domain of attraction of a Gaussian law is the Brownian CRT, meaning that the loops do not persist in the scaling limit. Assuming that $R,d \geq1$ this would therefore also be the case for our decorated tree model (since then distances are stochastically no bigger than those in looptrees). However, if there is polynomial tail decay in the offspring distribution it is always possible to choose the volume exponent $v$ large enough that larger volumes persist in the scaling limit; in the same spirit, if $\alpha = 2$ and we were to repeat the arguments of this paper we expect that we would obtain a volume growth exponent
\begin{equation*}
d_2^{\text{dec}} = {\frac{2}{f_2^v \wedge 1}},
\end{equation*}
with $\fv$ defined as above. We have not pursued this line of enquiry in this paper. In the infinite variance case when the offspring tails decay exactly like $cx^{-2}$, the arguments of this paper apply almost directly (one simple tweak is required) and we can simply take $\alpha = 2$ in \eqref{eqn:volume exponent def} and \eqref{eqn:spec dis exponent def} to give the correct exponents, but there are further subcases to consider for the logarithmic corrections for the precise gauge functions.

The paper is organised as follows. In Section \ref{sctn:model def} we give some technical background, and formally define a simple random walk on a decorated Galton--Watson tree. In Section \ref{sctn:finite tree bounds dec} we study the volume and distance asymptotics for finite decorated Galton--Watson trees. In particular, we show that if $T_{[n,2n]}$ is a Galton--Watson tree conditioned to have progeny in $[n.2n]$ vertices, and $\TER_{[n,2n]}$ its decorated version, then the rescaled diameter and volume of $\TER_{[n,2n]}$ are tight sequences under Assumption \ref{assn:sec}. In Section \ref{sctn:volume bounds decorated} we apply these results to obtain the volume results of Theorem \ref{thm:dec vol growth main}. In Section \ref{sctn:dec res bounds} we study resistance on $\TERa$, and in Section \ref{sctn:RW exponents} we bring everything together to prove Theorems \ref{thm:recurrence}, \ref{thm:dec main quenched RW} and \ref{thm:dec main ann RW}. We conclude in Section \ref{sctn:examples} with some examples.

We will work under Assumption \ref{assn:sec} for the entirety of the paper, and always assume that our trees have offspring distribution satisfying \eqref{eqn:offspring tails}. However, from Section \ref{sctn:model def} onwards we will not write it explicitly in the statements.

\textbf{Notation.} 
Here we briefly summarise the main notation used throughout the paper.
\begin{center}
\begin{tabular}{ |c|c|c|c| } 
\hline
$d_S, d_{\text{dis}}, d_f, d_w$ & spectral dimension, displacement exponent, fractal dimension, walk dimension (general)\\
$\Ti$ & critical Galton-Watson tree conditioned to survive, offspring distribution satisfying \eqref{eqn:offspring tails} \\
$(G_n)_{n \geq 1}$ & graph sequence used for decoration \\
$\TERa$ & decorated version of Kesten's tree $\Ti$ \\
$\TER$ & decorated version of GW tree $T$\\
$\dg$ & metric on $\TERa$ \\
$\rER$ & root of $\TERa$ \\
$\V$ & volume measure on $\TERa$ \\
$\BB (v,r)$ & closed ball of radius $r$ around $v \in \TERa$, defined wrt $\dg$ \\
$(U^{(1)}_n, U^{(2)}_n)$ & a uniform pair of distinct points on the boundary of $G_n$ \\
$d_n$ & pre-defined metric on $G_n$ \\
$d^U_n$ & $d_n(U^{(1)}_n, U^{(2)}_n)$ \\
$\Ref$ & effective resistance on $G_n$ when each edge has conductance $1$ \\
$R^U_n$ & $\Ref(U^{(1)}_n, U^{(2)}_n)$ \\
$\diam (G_n)$ & $\sup_{x, y \in G_n} d_n(x,y)$ \\
$\diamR (G_n)$ & $\sup_{x, y \in G_n} \Ref(x,y)$ \\
$V_n$ & a pre-defined measure on $G_n$ \\
$\ell_n$ & a labelling of boundary vertices of $G_n$; $\ell_n: \partial G_n \to \{0, 1, \ldots, n-1\}$ \\
$\bPb$ & law of $\TERa$ \\
$\bP$ & law of SRW on $\TERa$ \\
$\sad$ &$d(\alpha - 1)$ \\
$\saR$ &$R(\alpha - 1)$ \\
$\sav$ &$\frac{\alpha - 1}{v}$ \\
$\fv$ &$\frac{\alpha}{v}$ \\
$\dERa$ &$\frac{\alpha(\sad \wedge 1)}{(\alpha - 1)(\fv \wedge 1)}$ \\
$\dspec$ &$\frac{2\alpha (\saR \wedge 1)}{\alpha (\saR \wedge 1) + (\alpha - 1)(\fv \wedge 1)}$ \\
$\ddis$ & $\frac{(\saR \wedge 1)(\alpha -1)(\fv \wedge 1)}{(\alpha -1)(\fv \wedge 1)(\sad \wedge 1) + \alpha(\saR \wedge 1)(\sad \wedge 1)}$ \\
$X=(X_n)_{n \geq 0}$ & SRW on $\TERa$, started at $\rER$ \\
$\tau_r$ & exit time of $X$ from $\BB(\rER, r)$ \\
$\tau^R_r$ & exit time of $X$ from resistance ball of radius $r$ \\
$T$ & unconditioned critical Galton--Watson tree  with offspring distribution satisfying \eqref{eqn:offspring tails} \\
$d$ & graph distance on $T$ \\
$T_n$ & critical Galton-Watson tree with offspring distribution satisfying \eqref{eqn:offspring tails}, with $n$ vertices \\ 
$W^T / W^{(n)}$ &Lukasiewicz path of $T/T_n$ \\
$\TER_n$ & decorated version of GW tree $T_n$ \\
$|u|$ & label of $u$ in the lexicographical ordering of $T$ \\
$\p (u$) & parent of vertex $u$ \\
$\chi_u$ & relative position of $u$ amongst the offspring of $\p(u)$ \\
$|u-v|$ & the {lexicographical} distance between $u$ and $v$ \\
$M_Y$ & Given $T$ and $(Y_i)_{0 \leq 1 \leq |T|-1}$, $M_Y = \arg \max \left\{i \leq |T|-1: \sum_{j \preceq i} Y_j\right\}$ \\
$\Height_Y$ & $\sum_{j \preceq M_Y} Y_j$ \\
$H_Y$ & $d(\rho, u_{M_Y})$ \\
\hline
\end{tabular}

\begin{tabular}{ |c|c|c|c| } 
\hline
$T_{[n,2n]}$ & a Galton--Watson tree conditioned on having total progeny in the interval $[n, 2n]$ \\
$\TER_{[n, 2n]}$ & a decorated version of $T_{[n,2n]}$ \\
$T^{[h,2h]}$ & a Galton--Watson tree conditioned on having height in the interval $[h, 2h]$ \\
$\TERh$ & a decorated version of $T^{[h,2h]}$ \\
$W$ & Lukasiewicz path for forest \\
$H, H^h$ & $H$ the height function for forest, and $H^h_m = H_m \mathbbm{1}\{\arg \min_{[0, n_{h, \lambda}]} X \geq m\}$ \\
$\tau_h$ & $\inf\{i \geq 1: \Height (T^{(i)}) \in [h,2h]\}$ \\
$\Sr$, $\partial \Sr$ & $\bigcup_{j=1}^{\infty} G({s_j}) \cap \BT(\rho, r)$ and $\bigcup_{j=1}^{\infty} \partial G({s_j}) \cap \BT(\rho, r)$ respectively \\
$\Nfr$ & $|\partial \Sr|$: number of subtrees on backbone of $\TERa$ within decorated distance $r$ of $\rER$ \\
$n_{r,\epsilon,\lambda}$ & $r^{\sad \wedge 1} \lambda^{\epsilon}$ when $\sad \neq 1$, and $r \lambda^{\epsilon}(\log r)^{-1}$ when $\sad =1$ \\
$\Heightdec (\TER)$ & $\sup_{x \in \TER} \dER (\rER, x)$ \\
$\Heightdec_{\text{res}} (\TER)$ & $\sup_{x \in \TER} \Ref (\rER, x)$ \\
\hline
\end{tabular}
\end{center}

\textbf{Constants. }Throughout, $c, c', C, C'$ etc. will denote constants bounded from above and below, but values may change on each appearance. In general these constants can depend on the quantities appearing in Assumption \ref{assn:sec} and the constant $c$ appearing in \eqref{eqn:offspring tails}, but we will not always make this explicit. In the case where constants may depend on other parameters, this will be indicated using subscripts.

\textbf{Acknowledgements.} I would like to thank David Croydon for helpful discussions and Takashi Kumagai for hosting me at RIMS in 2019, during which time this work was initiated. I would also like to thank Delphin S\'enizergues for interesting conversations about a concurrent project and for his comments on the proof of Proposition \ref{prop:decorated height bound without log term}, and Nicolas Curien for putting us in touch and for helpful questions. Finally I would like to thank the referee for a detailed reading and many helpful comments.

%

\section{Definition and background}\label{sctn:model def}
\subsection{A technical lemma}

First we give a short lemma that we will apply several times in the paper to compose several levels of randomness. For example, we can apply it to the pairs $(X,Y) = (|T|, \sup_{v \in T} \deg v)$ or $(X,Y) = (\deg v_i, \V (G(v_i)))$, where $T$ is a Galton--Watson tree and where $v_i$ is a labelled vertex of $T$.

\begin{lem}\label{lem:stable composition tail}
Let $\beta > 0$, let $X$ be a non-negative random variable with $\pr{X>x} \sim cx^{-\beta}$ as $x \rightarrow \infty$, and let $Y$ be a second random variable, dependent on $X$.
\begin{enumerate}[(i)]
\item Suppose that there exist $m,z>0$ and $C < \infty$ such that for all $n \geq 1$,
$$\prcond{Y \geq \lambda n^z}{X = n}{} \leq C\lambda^{-m} \text{ for all } \lambda \geq 1.$$
Then, there exists $C_{m,z,\beta} < \infty$ such that, if $m \neq \frac{\beta}{z}$:
\[
\pr{Y \geq k} \leq C_{m,z,\beta}k^{-\left(\frac{\beta}{z}\wedge m\right)} \text{ for all } k \geq 1.
\]
Suppose that there exist $z, c'>0$ and $N<\infty$ such that for all $n \geq N$, $\prcond{Y \geq c'n^z}{X \geq n}{} \geq c'.$
Then there exists $c''>0$ such that $k \geq c'N^z$,
$$ \pr{Y \geq k} \geq c''k^{-\frac{\beta}{z}}.$$
\item Suppose that there exist $m>0$ and $C < \infty$ such that for all $n \geq 1$,
$$\prcond{Y \geq \lambda n \log n}{X = n}{} \leq C\lambda^{-m} \text{ for all } \lambda > 1.$$
Then, there exists $C_{m,z,\beta} < \infty$ such that, if $m \neq \beta$:
\[
\pr{Y \geq k \log k} \leq C_{m,z,\beta}k^{-\left(\beta\wedge m\right)} \text{ for all } k \geq 1.
\]
Suppose that there exist $c'>0$ and $N < \infty$ such that for all $n \geq N$,
$\prcond{Y \geq c'n \log n}{X \geq n}{} \geq c'.$
Then there exists $c''>0$ such that for all $k \geq c'N$, $$\pr{Y \geq k \log k} \geq c''k^{-\beta}.$$
\end{enumerate}
\end{lem}
\begin{proof}
This is just a computation. (In this paper we mainly apply the result when $X$ is a positive integer, but clearly the same proof works for continuous random variables). Note that it is sufficient to prove the result only for $k$ sufficiently large, since we can then increase or decrease the constants $C'$ and $c''$ if necessary so that the claim holds for all $k \geq 1$. 
\begin{enumerate}[(i)]
\item If $m > \frac{\beta}{z}$, we have for all $k \geq 1$ that
\begin{align*}
\pr{Y \geq k} \leq \pr{X \geq k^{\frac{1}{z}}} + \sum_{x=1}^{k^{\frac{1}{z}}} \pr{X = x} \prcond{Y \geq k x^{-z} X^{z}}{X = x}{} &\leq 2ck^{-\frac{\beta}{z}} + 2\sum_{x=1}^{k^{\frac{1}{z}}}\pr{X = x} C(kx^{-z})^{-m}.
\end{align*}
We can then apply summation by parts to deduce that there exists $C_{m,z,\beta}<\infty$ such that
\begin{align*}
\sum_{x=1}^{k^{\frac{1}{z}}}\pr{X = x} (kx^{-z})^{-m} &= k^{-m} \left( \pr{X \geq 1} -k^{m} \pr{X \geq k^{1/z} + 1} + \sum_{x=2}^{k^{1/z}} \pr{X \geq x} \left(x^{mz} - (x-1)^{mz}\right) \right)
\\
&\leq C_{m,z,\beta} k^{-m} \cdot k^{m-\frac{\beta}{z}} = C_{m,z,\beta}k^{-\frac{\beta}{z}}.
\end{align*}
The proof is the same when $\prcond{Y \geq \lambda n^z}{X = n}{} \leq C\lambda^{-m}$ where $m < \frac{\beta}{z}$, in which case the bracketed term in the summation by parts is of constant order and we get overall tail decy of order $k^{-m}$.

The lower bound is even simpler: 
\begin{align*}
\pr{Y \geq k} \geq \pr{X \geq (c')^{-\frac{1}{z}}k^{\frac{1}{z}}} \prcond{Y \geq k}{X \geq (c')^{-\frac{1}{z}}k^{\frac{1}{z}}}{} \geq c''k^{\frac{-\beta}{z}}.
\end{align*}

\item This follows by exactly the same proof as $(i)$ with $z=1$, again on decomposing according to whether $X\geq k$ or $X<k$.

%
\end{enumerate}
\end{proof}

\subsection{Galton--Watson trees}\label{sctn:GW tree defs}
Before defining the full decorated tree model, we give a brief background on Galton--Watson trees, starting with the Ulam-Harris labelling notation for discrete trees and following the formalism of \cite{Neveu}. Firstly, we introduce the set
\[
\mathcal{U}=\cup_{n=0}^{\infty} {\N}^n.
\]
By convention, ${\N}^0=\{ \emptyset \}$. If $u=(u_1, \ldots, u_n)$ and $v=(v_1, \ldots, v_m) \in \mathcal{U}$, we let $uv= (u_1, \ldots, u_n, v_1, \ldots, v_m)$ denote the concatenation of $u$ and $v$.

\begin{defn} A plane tree $T$ is a finite subset of $\mathcal{U}$ such that
\begin{enumerate}[(i)]
\item $\emptyset \in T$,
\item If $v \in T$ and $v=uj$ for some $j \in \N$, then $u \in T$,
\item For every $u \in T$, there exists a number $k_u(T) \geq 0$ such that $uj \in T$ if and only if $1 \leq j \leq k_u(T)$.
\end{enumerate}
\end{defn}

We let $\mathbb{T}$ denote the set of all plane trees. A plane tree therefore comes pre-equipped with a \textbf{lexicographical ordering} on its vertices, also known as the \textbf{depth-first} search order. If $T_n$ is a plane tree with $n$ vertices, we will often list its vertices as $u_0, u_1, \ldots, u_{n-1}$ in lexicographical order. If $u$ is a vertex of $T$, we let $|u|$ denote its label in the lexicographical ordering of $T$.

A plane tree $T$ can be coded by a walk excursion $(W^{T}_m)_{0 \leq m \leq n}$ called the \textbf{Lukasiewicz path}; this is defined by setting $W^{T}_0 = 0$, then for $m \geq 1$ listing the vertices $u_0, u_1, \ldots, u_{|T|-1}$ in lexicographical order and setting $W^T_{m+1} = W^T_m + k_{u_m}(T)-1$. It is not too hard to see that $W^{T}_m \geq 0$ for all $0 \leq m < |T|-1$, and $W^T_{|T|-1} = -1$.

A plane tree can be alternatively coded by its \textbf{height function} $(H^{T}_m)_{0 \leq m \leq n}$: this is constructed by again listing the vertices in lexicographical order, and then setting $H^T_m$ to be the height of $u_m$, or in other words its distance from the root of $T$. It is straightforward to see that the height function and the Lukasiewicz path are related by 
\begin{equation}\label{eqn:H from W background}
H^T_m = |\{\ell \leq m: W^T_{\ell} = \inf_{\ell \leq p \leq m}W^T_p\}|.
\end{equation}

If $u \neq u_0$ is a vertex of $T$, we also define $\p (u)$ to be the \textbf{parent} of $u$; that is, if $u = (u_1, \ldots, u_{n-1}, u_n)$, then $\p (u) = (u_1, \ldots, u_{n-1})$. We also let $\chi_u$ denote its relative position amongst the offspring of $\p(u)$, and say that $v$ is an \textbf{ancestor} of $u$ and write $v \preceq u$ if $v = (u_1, \ldots, u_m)$ for some $m \leq n$. For $i, j < |T|$, we write $i \preceq j$ if $u_i \preceq u_j$ in the lexicographical ordering.

Given a finite plane tree $T$, we let $\rho$ denote its root. This is the vertex corresponding to $\emptyset$ or $u_0$ in the lexicographical ordering. We also let $d$ denote graph distance on $T$, and if $v_h(T) = \arg \max_{v \in T} d(\rho, v)$, we call the path from $\rho$ to $v_h(T)$ the \textbf{Williams' spine}. If $v_h(T)$ is not unique, we take the leftmost path. If $u, v \in T$, we let $[[u,v]]$ denote the path from $u$ to $v$, inclusive of endpoints, and $((u,v]] = [[u,v]]\setminus \{u\}$.

In this paper we will work with Galton--Watson trees, meaning that $k_u$ is random and i.i.d. for each $u \in T$. (We assume that the reader is familiar with Galton--Watson trees, but if not see \cite[Section 1.2]{LeG2005randomtreesandapplications} for a brief introduction). When $T$ is a $\xi$-Galton--Watson tree, the Lukasiewicz path is therefore a random walk excursion with jump distribution $\xi - 1$ conditioned on $W^{T}_m \geq 0$ for all $0 \leq m < |T|-1$, and $W^T_{|T|-1} = -1$.

\subsection{Bounds for undecorated Galton--Watson trees}\label{sctn:GW tree undec}
In this section, we let $T$ denote a critical (i.e. with mean $1$) Galton--Watson tree with offspring distribution $\xi$, where $\xi$ is aperiodic and satisfies \eqref{eqn:offspring tails}. It is straightforward to show that $T$ is almost surely finite. Let $p_n = \xi(n)$ and recall that $\rho$ denotes the root of $T$. We also let $T_n$ denote a copy of $T$ conditioned on having $n$ vertices (since $\xi$ is aperiodic, this is well-defined for all sufficiently large $n$).

In this section we collect some standard results on Galton-Watson trees.

\subsubsection{Progeny bounds}

\begin{lem}\label{lem:prog tail bound}
\begin{enumerate}[(i)]
\item There exist constants $c', c'' \in (0, \infty)$, such that, as $k \rightarrow \infty$,
\[
\prb{|T| \geq k} \sim c'k^{\frac{-1}{\alpha}}, \qquad \prb{|T| = k} \sim c''k^{\frac{-(\alpha + 1)}{\alpha}}.
\]
\item There exists a constant $q>0$ such that $1 - \Eb{e^{-\theta |T|}} \sim q \theta^{\frac{1}{\alpha}}$ as $\theta \rightarrow 0$.
\end{enumerate}
\end{lem}
\begin{proof}
\begin{enumerate}[(i)]
\item This is shown in \cite[Proposition A.3(i)]{CurKortUIPTPerc}.
\item This is a direct consequence of part $(i)$ and Lemma \ref{lem:Tauberian at 0 app}, with $\beta = \frac{1}{\alpha}$.
\end{enumerate}
\end{proof}

\subsubsection{Height bounds}

At various points, we will also need to perform spinal decompositions along various choices of spine in $T$. For this, the following results will be useful. If $T$ has root $\rho$, we let 
\[
\Height(T) = \sup_{v \in T} d(\rho, v).
\]

\begin{lem}\label{lem:tree height bound slack}
There exists $c' \in (0, \infty)$ such that, as $n \rightarrow \infty$, 
\[
\prb{\Height(T) \geq n} \sim c'n^{-\frac{1}{\alpha - 1}}.
\]
\end{lem}
\begin{proof}
This follows directly from \cite[Lemma 2]{SlackInfVar} (note that the parameter $\alpha$ in \cite{SlackInfVar} corresponds to the parameter $\alpha -1$ in our notation).
\end{proof}

\subsubsection{Spinal decompositions}\label{sctn:spinal decomp}
At several points in this paper we will perform a spinal decomposition of a finite Galton--Watson tree along the branch maximising a certain functional (for example, maximising $\sum_u \diam (G(u))$). Performing this kind of maximisation often biases the laws of the degrees of the vertices on this spine. In this subsection we prove a general result showing that the vertex degrees along this kind of spine can always be stochastically dominated by an independent sequence of random variables with the size-biased distribution.

Accordingly, let $T$ be as above, with vertices $\{u_0, \ldots, u_{|T|-1}\}$ listed in lexicographical ordering, and let $Y=(Y_i)_{0 \leq i \leq |T|-1}$ be a sequence of random variables such that given $T$, the law of $Y_i$ depends only on $\deg u_i$. In our applications $Y_i$ will be non-negative for each $i$, but we do not assume this a priori. Let
\begin{equation}\label{eqn:M H def}
M_Y(T) = \arg \max \left\{i \leq |T|-1: \sum_{j \preceq i} Y_j\right\}, \hspace{1cm} \Height_Y(T) = \sum_{j \preceq M_Y} Y_j, \hspace{1cm} H_Y(T) = d(\rho, u_{M_Y}),
\end{equation}
breaking ties if necessary by taking the minimal such $i$. In the case that $Y_i \equiv 1$ for all $i$, we write $M$ and $H$ in place of the first and third quantities above. Given the sequence $(Y_i)_{0 \leq i \leq |T|-1}$, we then define a sequence of \textbf{spinal vertices} $s_0, s_1, \ldots, s_{H_Y}$ to be the ancestors of $u_{M_Y}$, listed in lexicographical order.

We first give the following (elementary) result.

\begin{lem}\label{lem:height tree decreasing prob}
The function $h: \N \rightarrow [0,1]$, $x \mapsto \prb{H_Y(T)=x}$ is non-increasing in $x$.
\end{lem}
\begin{proof}
Take some $n \geq 1$. Conditionally on the event $\{\deg \rho =n\}$, for $i \in \{1, \ldots, n\}$ let $T^{(i)}$ denote the subtree emanating from the $i^{th}$ child of $\rho$. For $x \geq 0$, we then write:
\begin{align*}
\prb{H_Y(T) = x+1} &= \sum_{n = 0}^{\infty} \prb{\deg \rho =n} \prcondb{H_Y(T) = x+1}{{\deg \rho =n}}{} \\
&\leq \sum_{n = 0}^{\infty} p_n \sum_{i=1}^n \prb{H_Y (T^{(i)}) = x} 
\leq \sum_{n = 0}^{\infty} np_n \prb{H_Y(T) = x} 
 = \prb{H_Y(T) = x},
\end{align*}
where the final line follows since $\xi$ is critical.
\end{proof}

The next proposition follows by the same logic used to prove a similar result in \cite[Section 2]{GeigerKerstingGWHeight}, which covers the special case in which $Y_i \equiv 1$ for all $0 \leq i \leq |T|-1$. In particular, by starting at $s_{H_Y}$ and recursively working backwards towards the root, it shows that the offspring distribution of the vertices along any such spine can be stochastically dominated by independent size-biased random variables.

\begin{prop}\label{prop:spinal offspring UB general}
Take the notation as above. There exists $c'< \infty$ such that, for  any $h \in \N \cup \{0\}$, any $i \leq h$, any set $A \subset \N^{h-i}$ and all $K \geq 1$,
\[
\prcondb{\deg s_i \geq K}{H_Y(T) = h, (\deg s_j)_{i < j \leq h} \in A}{} \leq c'K^{-(\alpha - 1)}.
\]
\end{prop}
\begin{proof}
On the event $\deg s_i = k+1$ for some $k \geq 1$, let $T^{(1)}, T^{(2)}, \ldots, T^{(k)}$ denote the subtrees rooted at each of the offspring of $s_i$, listed in lexicographical order. Given $1 \leq j \leq k$, also define
\begin{align*}
B_{j,k} &= \{\Height_Y (T^{(j)}) > \sup_{1 \leq i < j} \Height_Y (T^{(i)}), \ \Height_Y (T^{(j)}) \geq \sup_{j < i \leq k} \Height_Y (T^{(i)})\}.
\end{align*}
Then there exists $c'<\infty$ such that for any $K \geq 1$,
\begin{align*}
\prcondb{\deg s_i > K}{H_Y(T) = h, (\deg s_j)_{i < j \leq h} \in A}{} &\leq \sum_{k \geq K} p_k \sum_{j=1}^k \prcondb{B_{j,k}}{H_Y(T) = h, (\deg s_j)_{i < j \leq h} \in A}{} \\
&\leq \sum_{k \geq K} kp_k \leq c'K^{-(\alpha - 1)}. \qedhere
\end{align*}
\end{proof}

\begin{rmk}\label{rmk:spinal decomp}
This implies that the vertex degrees of the spinal vertices can be stochastically dominated by independent random variables satisfying the tail bound of Proposition \ref{prop:spinal offspring UB general}. (This does not imply that the vertex degrees are themselves independent of each other). Moreover, the same proof applies along the spine to a uniformly chosen vertex, or to the vertex $u_j$ (lexicographically ordered) for any fixed $j < |T|$, or the spine to the leftmost vertex in generation $j$, for any fixed $j \leq \Height (T)$.
\end{rmk}

For technical reasons later on we will also need a slightly more general version of Proposition \ref{prop:spinal offspring UB general}. Set $H = \Height (T)$, and let $H_Y^{\wedge}$ denote the spinal index of the most recent common ancestor of $u_M$ and $u_{M^Y}$, where $M$ and $M_Y$ are as in \eqref{eqn:M H def}. Later on we will need to condition on both $H$ and $H_Y$, for which we will need the following upper bound on the spinal offspring distribution. The proof is the same as that of Proposition \ref{prop:spinal offspring UB general} on incorporating the extra height condition into the event $B_{j,k}$.

\begin{prop}\label{prop:spinal offspring UB general conditioned}
Take the notation as above. There exists $c'< \infty$ such that, for  any $h_1 \geq h_2 \geq h_3 \in \N \cup \{0\}$, any $h_3 < i \leq h_2$, any set $A \subset \N^{h_2-i}$ and any $K \geq 1$,
\[
\prcondb{\deg s_i \geq K}{H=h_1, H_Y(T) = h_2, H_Y^{\wedge} = h_3, (\deg s_j)_{i < j \leq h_2} \in A}{} \leq c'K^{-(\alpha - 1)}.
\]
\end{prop}

Now restrict to the case $Y_i \equiv 1$, and recall that $H=H_Y(T)$ in this case. It is well-known \cite[remark below Proposition 2.2]{GeigerKerstingGWHeight} that the offspring distribution of spinal vertices is asymptotically size-biased as $H \rightarrow \infty$, but we will need the following precise formulation of this result.

\begin{prop}\label{prop:spinal offspring LB}
Let $c$ be as in \eqref{eqn:offspring tails}. There exists $c'<\infty$ such that for every $C, \epsilon>0$ there exists another constant $c_{\epsilon, C, c} \in (0, \infty)$ such that for all $n > c' \epsilon^{-1}$, all $k \leq Cn^{\frac{1}{\alpha - 1}}$ and any set $A \subset \N^{n}$:
\begin{enumerate}
\item $\prcondb{\deg s_n \geq k}{H \geq (1+\epsilon)n, (\deg s_i)_{i<n} \in A}{} \geq c_{\epsilon, C, c} k^{-(\alpha - 1)}.$
\item For any $m \geq (1+\epsilon)n$, $\prcondb{\deg s_n \geq k}{H = m, (\deg s_i)_{i<n} \in A}{} \geq c_{\epsilon, C, c} k^{-(\alpha - 1)}.$
\end{enumerate}
\end{prop}
\begin{proof}
We prove the first point and the second point will follow as a byproduct of the proof. Given some $\tilde{k}\geq 1$, we first prove a corresponding result for $\prcondb{\deg s_n = \tilde{k}}{H \geq (1+\epsilon)n}{}$, and then obtain the result by summing over $\tilde{k} \geq k$. 

Take the notation as in the proof of Proposition \ref{prop:spinal offspring UB general}. The key observation is then as follows (this was first observed in \cite[Section 2]{GeigerKerstingGWHeight}). Given that $\deg s_n = k+1$, $(\deg s_i)_{i < n} \in A$ and $H = m$ (where $m > n$), the subtrees $(T^{(i)})_{i \leq k}$ are independent of $\deg s_n$, and conditionally on $j$ being the smallest index such that $T^{(j)}$ maximises $\Height (T^{(i)})$, the law of $T^{(j)}$ is conditioned to have $\Height (T^{(j)}) = m-n$, and the subtrees $(T^{(i)})_{i \leq k}$ are jointly conditioned to satisfy $\Height (T^{(i)}) < m-n$ for all $i < j$, and $\Height (T^{(i)}) \leq m-n$ for all $i > j$. Moreover, given all of this, the random variables $(T^{(i)})_{i \neq j}$ are independent of each other.

In particular, for $1 \leq n < m < \infty$ and $1 \leq j \leq k < \infty$, letting
\[
P_{j,k,m,n} = p_k \prb{H < m - n}^{j-1} \prb{H < m+1-n}^{k-j} \frac{\prb{H=m-n}}{\prb{H=m}},
\]
it therefore follows from \cite[Lemma 2.1]{GeigerKerstingGWHeight} and then Lemma \ref{lem:height tree decreasing prob} that for any $m \geq (1+\epsilon)n$ and any $1 \leq k \leq Cn^{\frac{1}{\alpha - 1}}$,
\begin{align*}
\prcondb{\deg s_n = k+1}{H = m, (\deg s_i)_{i < n} \in A}{} \geq \sum_{j=1}^k P_{j,k,m,n} &\geq k p_k \prb{H < \epsilon n}^{k-1} =k p_k \prb{H < \epsilon n}^{k-1}.
\end{align*}
Then, by Lemma \ref{lem:tree height bound slack}, we know that we can choose $c'<\infty$ such that $\prb{H \geq \epsilon n} \leq c'(\epsilon n)^{-\frac{1}{\alpha - 1}}$ for all $n \geq \epsilon^{-1}$. Therefore, if $k \leq Cn^{\frac{1}{\alpha - 1}}$ then
\[
\prb{H < \epsilon n}^{k-1} \geq (1 - c'(\epsilon n)^{-\frac{1}{\alpha - 1}})^{Cn^{\frac{1}{\alpha - 1}}} \geq e^{-2Cc'\epsilon^{-\frac{1}{\alpha - 1}}}
\]
for all $n \geq 2(c')^{\alpha-1}\epsilon^{-1}$, so that $\prcondb{\deg s_n = k+1}{H \geq (1+\epsilon)n, (\deg s_i)_{i < n} \in A}{} \geq e^{-3Cc'\epsilon^{-\frac{1}{\alpha - 1}}}kp_k$.

To prove the result as stated, we can then choose $c_{\epsilon, C, c}>0$ so that
\begin{align*}
&\prcondb{\deg s_n \geq k+1}{H =m, (\deg s_i)_{i < n} \in A}{} \\
&\qquad \geq \sum_{k \leq \tilde{k} \leq 2k} \prcondb{\deg s_n = \tilde{k}+1}{H =m, (\deg s_i)_{i < n} \in A}{} \geq \sum_{k \leq \tilde{k} \leq 2k} e^{-3Cc'\epsilon^{-\frac{1}{\alpha - 1}}}kp_k \geq c_{\epsilon,C,c} k^{-(\alpha - 1)}.
\end{align*}
\end{proof}

\subsubsection{Vertex degrees}
We will also need the following result on the degree of a typical vertex.

\begin{lem}\label{lem:GW typical degree prob bound}
Let $n \in \N$, and let $T_n$ be a Galton--Watson tree with offspring distribution $\xi$ satisfying \eqref{eqn:offspring tails} but conditioned to have $n$ vertices. Let $U_n$ be uniform on $\{0, \dots, n-1 \}$ and let $u_{U_n}$ be the corresponding vertex in the lexicographical ordering of $T_n$. Then there exists a constant $c'<\infty$ such that for all $n \geq 1$ and all $k \geq 1$,
\[
\prb{\deg(u_{U_n}) \geq k} \leq c' k^{-\alpha}.
\]
\end{lem}
\begin{proof}
Fix $n \in \N$, and recall from Section \ref{sctn:GW tree defs} that the vertex degrees of the vertices of $T_n$ correspond to (two more than) each of the jump sizes of the Lukasiewicz path $W^{(n)}$, which is conditioned to first hit $-1$ at time $n$. Since ${U_n}$ is uniform amongst $\{0, \ldots, n-1\}$, it follows from the (discrete) Vervaat transform (e.g. see \cite[Proposition 10]{KortSubexp}) that the ${U_n}^{th}$ cyclic shift of $W^{(n)}$, i.e. the random walk $\tilde{W}^{(n)}$ given by
\[
\tilde{W}^{(n)}(t) = \begin{cases} {W}^{(n)}({U_n}+t) - {W}^{(n)}({U_n}) &\text{ if } {U_n}+t \leq n \\
{W}^{(n)}({U_n}+t-n) - {W}^{(n)}({U_n}) &\text{ if } {U_n}+t > n
\end{cases}
\]
is a random walk bridge from $0$ to $-1$, by which we mean that, if $W$ is a simple random walk with $\prb{W_{m+1}-W_m=x-1} = \xi(x)$ for all $m, x \in \N$, then $\left(\tilde{W}^{(n)}\right)_{m=0}^n$ has the law of $\left(W\right)_{m=0}^n$ but conditioned on $W_n = -1$. In particular, it follows from the Markov property that $\tilde{W}^{(n)}$ has a density with respect to the unconditioned walk $W$, and moreover $\deg u_{U_n}$ is equal to $\tilde{W}^{(n)}(1)+2$ (or equal to $\tilde{W}^{(n)}(1)+1$ if ${U_n}=0$). We deduce that for any $k \geq 1$,
\begin{align}\label{eqn:Verv trans abs cont rel}
\prb{ \deg u_{U_n} \geq k} \leq \prb{ \tilde{W}^{(n)}(1) \geq k-2} = \Eb{ \mathbbm{1}\{W(1) \geq k-2\} \frac{p_{n-1}(-W(1)-1)}{p_n(-1)}}.
\end{align}
Moreover, by the local limit theorem on p.236 of \cite{GnedenkoKolmogorovLimitStable}, we have that 
\begin{align}\label{eqn:Gned LLT}
a_n \prb{W(n) = k} - \tilde{p}_{\alpha}\Big(\frac{k}{a_n}\Big) \rightarrow 0
\end{align}
uniformly in $k$, where $\tilde{p}_{\alpha}$ is the density of a centred stable random variable $Y_{\alpha}$, and $\tilde{p}_{\alpha}(0) = |\Gamma(\frac{-1}{\alpha})|^{-1}>0$. Moreover, under the assumption of \eqref{eqn:offspring tails} on the tails of $\xi$, $a_n$ is of order $n^{\frac{1}{\alpha}}$. It follows from this that the ratio $\frac{p_{n-1}(-W(1)-1)}{p_n(-1)}$ is bounded above uniformly in $k$ and $n$, which gives the result.
\end{proof}

\begin{lem}\label{lem:max degree tail}
Let $T$ be a Galton--Watson tree with a critical aperiodic offspring distribution $\xi$ satisfying \eqref{eqn:offspring tails}. Then there exist constants $c', C' \in (0, \infty)$ such that for all $k \geq 1$,
\[
c'k^{-1} \leq \prb{\sup_{v \in T} (\deg v) \geq k} \leq C'k^{-1}
\]
\end{lem}
\begin{proof}
We first take some $n \geq 1$ and look at what happens on the event $\{|T|=n\}$ so that we can apply Lemma \ref{lem:stable composition tail} to the pair $(X, Y)=(|T|, \ \sup_{v \in T} \deg v)$. For an upper bound, note that it follows from a union bound and Lemma \ref{lem:GW typical degree prob bound} that for all $n, \lambda \geq 1$,
\begin{align}\label{eqn:sup deg LB}
\prb{\sup_{v \in T_n} \deg v \geq n^{\frac{1}{\alpha}} \lambda} &\leq \sum_{i=0}^{n-1} \prb{ \deg u_i \geq n^{\frac{1}{\alpha}}\lambda} = n\prb{\deg u_{U_n} \geq n^{\frac{1}{\alpha}} \lambda} \leq c' \lambda^{-\alpha}.
\end{align}

For the lower bound, we can apply the Vervaat transform as in the previous proof. Note that it follows from the Markov property just as in \eqref{eqn:Verv trans abs cont rel} that if $A_n$ is an event depending precisely on the degrees of the vertices $(u_{U+i})_{i \leq \lfloor \frac{n}{2}\rfloor }$, and therefore precisely on $(\tilde{W}^{(n)}(i))_{i \leq \lfloor \frac{n}{2}\rfloor }$, then
\begin{align*}
\prb{ A_n(\tilde{W}^{(n)})} = \Eb{ \mathbbm{1}\{A_n (W)\} \frac{p_{\lceil \frac{n}{2}\rceil}(-W(\lfloor \frac{n}{2}\rfloor)-1)}{p_{n}(-1)}}.
\end{align*} 

By our assumptions on the offspring distribution from \eqref{eqn:offspring tails} we can take $c>0$ so that $\xi ([n^{\frac{1}{\alpha}}, \infty)) \geq cn^{-1}$ for all $n \geq 1$. By \cite[Section VIII, Proposition 4]{BertoinLevy} and \eqref{eqn:Gned LLT}, we can also choose $A< \infty$ so that 
\[
\prb{\left|W\left(\left\lfloor \frac{n}{2}\right\rfloor \right) + 1\right| > An^{\frac{1}{\alpha}}} \leq \frac{c}{5}.
\]
On the event $\{|W\left(\lfloor \frac{n}{2}\rfloor \right) + 1| \leq An^{\frac{1}{\alpha}}\}$, note that $\frac{p_{\lceil \frac{n}{2}\rceil}(-W(\lfloor \frac{n}{2}\rfloor)-1)}{p_{n}(-1)}$ is uniformly bounded below by a positive constant (again by \eqref{eqn:Gned LLT}). Therefore, if
\[
A_n(\tilde{W}^{(n)}) = \left\{\sup_{i \leq \lfloor\frac{n}{2}\rfloor} \deg u_i \geq n^{\frac{1}{\alpha}} \text{ and } \left|\tilde{W}^{(n)}\left(\left\lfloor \frac{n}{2}\right\rfloor\right) + 1\right| \leq An^{\frac{1}{\alpha}}\right\},
\]
then there exists $\tilde{c}>0$ such that (re-running the argument with a smaller $c$ if necessary so that $1- e^{-c/3} \geq \frac{c}{4}$):
\begin{align*}
\prb{\sup_{v \in T_n} \deg v \geq n^{\frac{1}{\alpha}}} \geq \prb{A_n(\tilde{W}^{(n)})} \geq \tilde{c}\prb{A_n(W)} &\geq\tilde{c} \left[ 1 - (1-cn^{-1})^{\frac{n-1}{2}} - \frac{c}{5}\right] \geq \tilde{c} \left[1- e^{-c/3} - \frac{c}{5}\right] \geq \frac{c\tilde{c}}{20}. 
\end{align*}
Combining with \eqref{eqn:sup deg LB} and since $\prb{|T|\geq n} \sim c'n^{\frac{-1}{\alpha}}$ for some $c' \in (0, \infty)$ by Lemma \ref{lem:prog tail bound}, we can now apply Lemma \ref{lem:stable composition tail}$(i)$ with $(X, Y)=(|T|, \sup_{v \in T} \deg v)$, $\beta= z = \frac{1}{\alpha}$ and $m=\alpha$ to deduce the result.
%
\end{proof}

\subsection{Infinite critical trees}\label{sctn:infinite trees bckgrnd}
In this section we introduce Kesten's tree $T_{\infty}$ for a given critical offspring distribution $\xi$.

\begin{defn}\label{def:Kesten's tree}\cite{KestenIICtree}.
Let $\xi$ be a critical offspring distribution, and define its size biased version $\xi^*$ by
\[
\xi^*(n) = {n \xi (n)}
\]
for all $n \geq 0$. The \textbf{Kesten's tree} $T_{\infty}$ associated to the probability distribution $\xi$ is a two-type Galton--Watson tree distributed as follows:
\begin{itemize}
\item Individuals are either normal or special.
\item The root of $T_{\infty}$ is special.
\item A normal individual produces only normal individuals according to $\xi$.
\item A special individual produces individuals according to the size-biased distribution $\xi^*$. Of these, one of them is chosen uniformly at random to be special, and the rest are normal.
\end{itemize}
\end{defn}
Almost surely, the special vertices form a unique infinite backbone of $T_{\infty}$. Note that this is one-ended. Aldous in \cite{AldousFringeSinTree} coined the term \textit{sin-trees} for such trees, since they have a single infinite spine. Although a critical Galton--Watson tree is almost surely finite, Kesten \cite[Lemma 1.14]{KestenIICtree} showed that $T_{\infty}$ arises as the local limit of a critical Galton--Watson tree with offspring distribution $\xi$ as its height goes to infinity.

%

\subsection{Decorated tree definition}\label{sctn:dec tree def}
In this section we give the construction of our decorated tree $\TERa$. Formally, we let $\Tai$ denote Kesten's tree with a critical (aperiodic) offspring distribution $\xi$ satisfying \eqref{eqn:offspring tails}.

In line with the literature, we also let $p_k = \xi(k)$. By standard theory (e.g. \cite[Chapter VIII]{BGTRegVariation}), if $(X_i)_{i=1}^{\infty}$ are i.i.d. distributed according to $\xi$, it follows that $a_n^{-1} \sum_{i=1}^{\infty} X_i \rightarrow Z_{\alpha}$, where $Z_{\alpha}$ is a non-negative $\alpha$-stable random variable and $a_n = (c|\Gamma (-\alpha)|n)^{\frac{1}{\alpha}}$. ($Z_{\alpha}$ is different to $Y_{\alpha}$ in the proof of Lemma \ref{lem:GW typical degree prob bound}, which may be negative).

To construct the decorated model, we will suppose that $((G_n, d_n, V_n, \ell_n))_{n \geq 1}$ is a sequence of random graphs with some pre-specified distribution, such that for all $n \geq 1$, $G_n$ is almost surely connected, $d_n$ is a metric on $G_n$, $V_n$ is a measure on the vertices of $G_n$, and $G_n$ has $n$ pre-specified ``boundary'' vertices (for example, a tree with $n$ leaves, or a dissection of the $n$-gon). Moreover, we assume that the boundary vertices $\partial G_n$ come pre-equipped with a labelling $\ell_n: \partial G_n \to \{0, 1, \ldots, n-1\}$ (in the case of a random planar map, this would represent the clockwise ordering of the boundary vertices when $G_n$ is embedded in the plane; for a graph such as the complete graph where the ordering is not important, $\ell_n$ can be a uniform bijection from $\partial G_n$ to $\{0, 1, \ldots, n-1\}$).

Finally we will assume for the construction that $\Ti$ has been \textbf{planted} in the following way: we add a single vertex to $\Ti$, which we call the \textbf{seed} and denote $\p (\rho)$, and attach it to the root of $\Ti$. We call the edge joining the seed to the root the \textbf{root edge}. We let $\rho$ denote the root vertex of $\Ti$.

To construct $\TERa$, we work on the probability space $(\mathbf{\Omega}, \mathbf{\mathcal{F}}, \bPb)$ and do the following:
\begin{enumerate}
\item Sample a pair $\left(\Tai, ((G(u), d(u), V(u), \ell(u))\right)_{u \in \Tai \setminus \p(\rho)})$, where $\Tai$ is Kesten's tree as in Definition \ref{def:Kesten's tree} and has been planted, and for each $u \in \Ti$, $G(u)$ is independently sampled according to the law of $(G_{\deg u}, d_{\deg u}, V_{\deg u}, \ell_{\deg u})$.
\item For each $u \in \Tai$, independently choose $U_u \sim \textsf{Uniform} \{0, \ldots \deg u -1\}$. Define a bijection $\ell_U(u): \partial G(u) \to \{0, 1, \ldots, \deg u -1\}$ by $\ell_U(u)(v) \mapsto (\ell(v)+U_u) \mod (\deg u)$, and then an injection $b(u): \partial G(u) \to E(\Tai)$ such that $b(u)(v) = (u, \p (u))$ if $ \ell_U(u)(v) = 0$, and $b(u)(v) = (u, u(\ell_U(u)(v)))$ otherwise (recall that for $j \geq 1$, $uj$ is the $j^{th}$ child of $u$, ordered from left to right).
\item Given a vertex $x \in \bigcup_{u \in \Tai} G(u)$ and a vertex $v \in \Ti$, we say that $v = \hat{V}_T(x)$ if $x \in G(v)$. We then define an equivalence relation $\esim$ on the vertices of $\bigcup_{u \in \Tai} \partial G(u)$ by saying that $x \esim y$ if and only if $b(\hat{V}_T(x))(x) = b(\hat{V}_T(y))(y)$, or in other words that they are both in bijection with the same edge of $\Tai$.
\item We then set
\[
\TERa = \bigcup_{v \in \Tai} G(v) / \esim.
\]
\item We define a metric $\dg$ on $\TERa$ as follows. If $x, y \in \TERa$ and $d(x,y)=n$ for some $n \geq 0$, we let $[[\hat{V}_T(x), \hat{V}_T(y)]] = v_0, v_1, \ldots, v_n$ denote the path of (underlying tree) vertices between $\hat{V}_T(x)$ and $\hat{V}_T(y)$ in $\Tai$. We then define $\dg$ by setting
\begin{equation}\label{eqn:decorated metric def}
\begin{split}
    \dg (x,y) = &d_{G(v_0)}\left(x, b(v_0)^{-1}((v_0,v_1))\right) + \sum_{1 \leq i \leq n-1} d_{G({v_i})}\left(b(v_i)^{-1}((v_{i-1},v_i)), b(v_i)^{-1}((v_i, v_{i+1}))\right) \\
&+ d_{G(v_n)}\left(b(v_n)^{-1}((v_{n-1},v_n)), y\right),
\end{split}
\end{equation}
where $d_{G({v_i})} = d(v_i)$ represents the metric inherited from $G({v_i})$.
\item We define a measure $\V = \sum_{u \in \Tai} V(u)$.
\item Finally, we root $\TERa$ as follows. Recall from step 2 that there is a unique vertex $x \in \partial G(\rho)$ such that $b(\rho)(x) = (\rho, \p(\rho))$. Denote this vertex $x^*$. We add a new vertex $\rER$ to $\TERa$, join it to $x^*$ by a single edge of length one, so that $\dg (\rER, x^*)=1$. Also set $\V (\rER)=1$.
\end{enumerate}

\begin{rmk}
The reason for this rooting convention is that it will be technically convenient to bound distances and volumes away from $0$ to prove some of the bounds required to directly apply the theorems of \cite{KumMisumiHKStronglyRecurrent}. The results of this paper will still be true with other sensible rooting conventions, but the assumptions of \cite{KumMisumiHKStronglyRecurrent} would not hold exactly as stated in their paper.
\end{rmk}

If $v \in \Tai$, let $T_v$ be the subtree of $\Tai$ rooted at $v$. This also induces a subgraph of $\TERa$ consisting of the graphs $\bigcup_{u \in T_v} G(u)$. We denote this subgraph by $\TERa(v)$. 

To study a simple random walk on $\TERa$, we will use the measure $\V$ such that $\V(x) = \deg x$ for all $x$ in $\TERa$ (and therefore we also take $V_n(x) = \deg x$ on $G_n$). The assumptions on the measure $V_n$ in Assumption \ref{assn:sec} then correspond to assumptions on the number of edges in $G_n$.

Throughout this paper, if $T$ is a finite tree, its \textbf{decorated version} describes the metric measure space obtained by replacing $\Tai$ with $T$ in the above construction.

\subsection{Simple random walk on $\TERa$}
A simple random walk $(X_n)_{n \geq 0}$ is a discrete-time Markov chain on $\TERa$ with $X_0 = \rER$ and 
\[
\prcond{X_{n+1} = y}{X_n=x}{} = \frac{1}{\deg x} \mathbbm{1}\{y \sim x\}
\]
for all $n \geq 1$. We define the (consequently symmetric) transition density by
\[
p_n(x,y) = \frac{\prstart{X_n=y}{x}{}}{\deg y}.
\]

In this paper, we will be interested in two metrics on $\TERa$: the metric $\dg$ constructed above, and the effective resistance metric. We let $\tau_r$ denote the exit time of $(X_n)_{n \geq 0}$ from a ball of radius $r$ of $\TERa$ defined with respect to the metric $\dg$.

For two vertices $a, b \in \TERa$, the effective resistance between them is defined as 
\[
R(a,b) =\left( \inf \{\mathcal{E}(f,f): \mathcal{E}(f,f) < \infty, f(a)=1, f(b)=0\}\right)^{-1},
\]
where 
\[
\mathcal{E}(f,f) = \frac{1}{2}\sum_{\substack{x, y \in \TERa: \\ x \sim y}} (f(x)-f(y))^2.
\]
This is a rather abstract definition but really corresponds to the classical notion of electrical resistance when we consider $G$ to be an electrical network in which all edges have conductance $1$. In particular, $R$ is a metric (e.g. see \cite{tetali1991random}) and satisfies the parallel and series laws for resistance.

We therefore also let $\tau^R_r$ denote the exit time of $(X_n)_{n \geq 0}$ from a ball of radius $r$ defined with respect to the effective resistance metric.

We will assume that $\Tai$ and the set $\left((G(v), d(v), V(v), \ell (v))\right)_{v \in \Tai}$ are defined on the probability space $(\mathbf{\Omega}, \mathbf{\mathcal{F}}, \bPb)$. We denote the law of a random walk on $\TERa$ by $\bP$: this law is therefore also a random variable on the probability space $(\mathbf{\Omega}, \mathbf{\mathcal{F}}, \bPb)$.

\section{Exponents for finite decorated Galton--Watson trees}\label{sctn:finite tree bounds dec}
In this section we prove some results on volume and distance on finite decorated Galton--Watson trees.
Since $\Tai$ is constructed to have one infinite backbone to which many finite fringe Galton--Watson trees are grafted, these estimates will be crucial when we prove the volume and resistance bounds for $\TERa$ in Sections \ref{sctn:volume bounds decorated} and \ref{sctn:dec res bounds}.


The strategy will be to use the results of Section \ref{sctn:GW tree undec} together with Assumption \ref{assn:sec} and Lemma \ref{lem:stable composition tail} to compute the exponents for decorated trees.

For technical reasons in this section we will need to consider the following structures for $n, h \geq 1$.
\begin{itemize}
\item $T_{[n,2n]}$, which is a Galton--Watson tree conditioned on having total progeny in the interval $[n, 2n]$.
\item $\TER_{[n, 2n]}$, which is a decorated version of $T_{[n,2n]}$.
\item $T^{[h,2h]}$, which is a Galton--Watson tree conditioned on having height in the interval $[h, 2h]$.
\item $\TERh$, which is a decorated version of $T^{[h,2h]}$. 
\end{itemize} 

\subsubsection{Heuristics}
Since the backbone vertices of $\Ti$ have size-biased tails by Definition \ref{def:Kesten's tree}, it follows from Lemma \ref{lem:stable composition tail}$(i)$ (with $(X,Y) = (\deg s_n,\diam (G(s_n)) )$,  $z=\frac{1}{d}$, $\beta = \alpha - 1$ and $m = \frac{\alpha}{\alpha - 1}$) that, if $s_n$ is the $n^{th}$ backbone vertex of $\Tai$, then there exist constants $c', C' \in (0, \infty)$ such that
\[
c'x^{-\sad} \leq \prb{\diam (G(s_n)) \geq x} \leq C'x^{-(\sad \wedge \frac{\alpha}{\alpha - 1})}.
\]
It is a standard fact that a sum of i.i.d. random variables $\sum_{i=1}^n X_i$ where $\pr{X_1 > x}$ is of order $x^{-\beta}$ is typically of order $n^{\frac{1}{\beta \wedge 1}}$ (some precise results are proven in the appendix). Therefore, we expect that
\[
\sum_{u \preceq s_n} \diam (G(u)) \approx n^{\frac{1}{\sad \wedge 1}}.
\]
This is why there is a factor of $\sad \wedge 1$ appearing in \eqref{eqn:volume exponent def} and \eqref{eqn:spec dis exponent def}. Similar results respectively hold for $\saR$ and $\fv$ for sums with terms of the form $\diamR (G(u))$ or $V (G(u))$. In this section we give quantitative versions of these kinds of statements.

\subsection{Decorated height bounds}
\subsubsection{Lower bounds}
If $\TER$ is the decorated version of a finite Galton--Watson tree $T$, we define the \textbf{decorated height} of $\TER$ by
\[
\Heightdec (\TER) = \sup_{x \in \TER} \dg(\rhoERA, x).
\]

We start with a lower bound on the decorated height.

\begin{prop}\label{prop:TERa height prob LB}
Under Assumption \ref{assn:sec} and with offspring distribution satisfying \eqref{eqn:offspring tails}, if $\sad \neq 1$, then there exists a constant $c' \in (0, \infty)$ such that for all $n\geq 1$,
\[
\prb{\Heightdec( \TER) \geq n} \geq c'n^{\frac{-(\sad \wedge 1)}{\alpha - 1}}.
\]
If $\sad = 1$, then there exists a constant $c' \in (0, \infty)$ such that for all $n\geq 1$,
\[
\prb{\Heightdec( \TER) \geq n \log n} \geq c'n^{\frac{-1}{\alpha - 1}}
\]
\end{prop}
%

Before we give the proof of the lower bound, we will need the following result on distances across the graphs corresponding to spinal vertices. We work in the setup of Section \ref{sctn:spinal decomp} in the case where $Y_i \equiv 1$ for all $i \geq 1$. We simply write $H$ in place of $H_Y$ and as in Section \ref{sctn:spinal decomp} we let $s_0, s_1, \ldots s_H$ denote the ordered set of spinal vertices along the path corresponding to the height of the tree (taking the leftmost spine in case of ties).

Recall that, for each $i \geq 1$, $d^U(s_i)$ denotes the distance between the two boundary vertices of $G(s_i)$ that correspond to edges joining $s_i$ to the neighbouring spinal vertices $s_{i-1}$ and $s_{i+1}$ (so these form a uniform pair of distinct vertices in $\partial G(s_i)$).

\begin{lem}\label{lem:ER spine offspring dist}
Let $(s_i)_{i \leq H}$ for some $H \geq 0$ or $(s_i)_{i < \infty}$ denote the ordered set of spinal vertices of a finite Galton-Watson tree as in Section \ref{sctn:spinal decomp}, \textit{or} the ordered set of backbone (special) vertices of Kesten's tree as defined in Definition \ref{def:Kesten's tree}.
 Then there exist $C', c', \in (0, \infty)$ such that for any $h, i \geq 1$ and any $A \in \N^{h-i}, B \in \N^i$:
\begin{enumerate}[(i)]
\item For all $k \geq 1$,
\begin{align*}
\prcondb{\diam (G(s_i)) \geq k}{H=h, (\deg s_j)_{i < j \leq h} \in A}{} \leq C'k^{-(\sad \wedge \frac{\alpha}{\alpha - 1})}.
\end{align*}
\item For all $1 \leq k \leq c'i^{\frac{1}{\sad}}$ and $h' \geq 2i$:
\begin{align*}
\prcondb{d^U(s_i) \geq k}{H = h', (\deg s_j)_{j < i} \in B}{} \geq c'k^{-\sad}.
\end{align*}
\end{enumerate}
(On Kesten's tree these hold without the conditioning on $H$ and both statements hold for all $k \geq 1$).
\end{lem}
\begin{proof}
This follows from Lemma \ref{lem:stable composition tail}$(i)$ applied to the pair $(\deg s_i, d^U(G(s_i)))$. By Proposition \ref{prop:spinal offspring UB general} (for $(i)$) and Proposition \ref{prop:spinal offspring LB} (for $(ii)$) we can take $\beta = \alpha - 1$ and by Assumption \ref{assn:sec}(D) we can take $z = \frac{1}{d}$ and $m=\frac{\alpha}{\alpha - 1}$.
\end{proof}


\begin{proof}[Proof of Proposition \ref{prop:TERa height prob LB}]
We start with the case $\sad \neq 1$. Note that $\Heightdec (\TER)$ stochastically dominates $\sum_{i=1}^{\frac{H}{2}} d^U(s_i)$, so we will bound this latter quantity.

To do this, first note that by Lemma \ref{lem:tree height bound slack}, there exists $c'>0$ such that $\prb{H \geq n} \geq c'n^{\frac{-1}{\alpha - 1}}$ for all $n \geq 1$. Then, if $\sad < 1$, we have by Lemma \ref{lem:ER spine offspring dist}$(ii)$ that there exists a (deterministic) constant $\tilde{c}>0$ such that for all $n \geq 1$,
\[
\prcondb{\sum_{i=1}^{\frac{H}{2}} d^U(s_i) \geq \tilde{c}n^{\frac{1}{\sad \wedge 1}}}{H = n}{} \geq \prcondb{\exists i \leq \frac{H}{2}: d^U(s_i) \geq \tilde{c}n^{\frac{1}{\sad}}}{H = n}{} \geq  \frac{1}{2}.
\]
If $\sad > 1$ then the same result holds by the Law of Large Numbers. This proves the result by Lemma \ref{lem:stable composition tail}$(i)$ applied to the pair $(X,Y) = (H, \sum_{i=1}^{\frac{H}{2}} d^U(s_i))$. By Lemma \ref{lem:tree height bound slack} we can take $\beta = \frac{1}{\alpha -1}$ and by the calculations above we can take $z = \frac{1}{\sad \wedge 1}$.

If $\sad = 1$, recall that by standard results on asymmetric distributions in the domain of attraction of a Cauchy distribution (e.g. see \cite[Equation (1.3)]{BergerCauchy}) that if $(X_i)_{i=1}^{\infty}$ are non-negative i.i.d. with $\pr{X_1 > x} \sim cx^{-1}$ for some $c>0$, then there exists $C \in (0, \infty)$ such that $n^{-1}\left( \sum_{i=1}^n X_i - C n\log n\right)$ has a non-trivial scaling limit, so that there exists another $c'>0$ such that
\[
\prcondb{\sum_{i=1}^{\frac{H}{2}} d^U(s_i) \geq c'n \log n}{H = n}{} \geq  \frac{1}{2}.
\]
This proves the result by Lemma \ref{lem:stable composition tail}$(ii)$ applied to the pair $(X,Y) = (H, \sum_{i=1}^{\frac{H}{2}} d^U(s_i))$. By Lemma \ref{lem:tree height bound slack} we can take $\beta = \frac{1}{\alpha -1}$.
\end{proof}

\subsubsection{Upper bounds}\label{sctn:spinal decomp dec}
We now turn to proving an upper bound for the decorated height of a finite Galton--Watson tree.

Proving Proposition \ref{prop:decorated height bound without log term} is the main technical challenge of the paper, since it requires us to control distances along \emph{all} branches of the tree simultaneously.

\begin{prop}\label{prop:decorated height bound without log term prog}
Assume that Assumption \ref{assn:sec} and \eqref{eqn:offspring tails} hold, take any $n \geq 1$ and and let $\TER_{[n, 2n]}$ be as above (i.e. the decorated version of a Galton--Watson tree with between $n$ and $2n$ vertices). Set
\[
\tad=\frac{1}{3(2\alpha - (\sad \wedge 1))}.
\]
Take $\epsilon>0$ as in Assumption \ref{assn:sec}(D). Then there exists a constant $c_{\epsilon} < \infty$ such that for all $\lambda \geq 1$:
\begin{itemize}
\item If $\sad < 1, \prb{\Heightdec (\TER_{[n, 2n]}) \geq n^{\frac{1}{\alpha d}} \lambda} \leq c_{\epsilon}\lambda^{-\tad}$.
\item If instead $\sad = 1, \prb{\Heightdec (\TER_{[n, 2n]}) \geq \lambda n^{ 1-\frac{1}{\alpha}} \log n} \leq c_{\epsilon}\lambda^{-\tad}$.
\item If $\sad > 1, \prb{\Heightdec (\TER_{[n, 2n]}) \geq n^{1-\frac{1}{\alpha}} \lambda} \leq c_{\epsilon}\lambda^{-\tad}$.
\end{itemize}
\end{prop}

The precise value of $t^d_{\alpha}$ used here is neither optimal nor important.

\begin{rmk}
\begin{enumerate}
\item We strongly believe that the result should also be true on replacing $\TER_{[n, 2n]}$ with $\TER_n$ in this proposition, but this would be more difficult to prove as the requirement that $|T|=n$ is a much stricter conditioning.
\item This tail bound is not good enough to plug into Lemma \ref{lem:stable composition tail} and get an overall tail bound for the unconditioned tree complementing Proposition \ref{prop:TERa height prob LB}, but is nevertheless sufficient for our purposes, and in particular shows that the rescaled decorated height of $\TER_{[n,2n]}$ forms a tight sequence. 
\item Note that, since $\Height (T_{[n,2n]})$ is of order $n^{1-\frac{1}{\alpha}}$ with high probability, the proof of Proposition \ref{prop:TERa height prob LB} gives us complementary lower bounds, in particular explicit constants $\ttad>0, c'<\infty$ such that for all $\lambda \geq 1$:
\begin{align*}
\prb{\Heightdec (\TER_n) \leq \lambda^{-1} n^{(\frac{1}{\alpha d} \vee \frac{\alpha - 1}{\alpha})}} \leq c'\lambda^{-\ttad} \text{ when } \sad \neq 1, \\
\prb{\Heightdec (\TER_n) \leq \lambda^{-1} n^{ 1-\frac{1}{\alpha}} \log n} \leq c'\lambda^{-\ttad} \text{ when } \sad = 1.
\end{align*}

\end{enumerate}
\end{rmk}

\textbf{Outline of proof}

In fact it will be more convenient to prove an analogous result for Galton--Watson trees conditioned on their height, rather than their total progeny. This is because this allows us to apply the bounds in Propositions \ref{prop:spinal offspring UB general} and \ref{prop:spinal offspring UB general conditioned} regarding the tails of vertex degrees on the decorated spine.

\begin{prop}\label{prop:decorated height bound without log term}
Assume that Assumption \ref{assn:sec} holds, take any $h \geq 1$ and and let $\TERh$ be as above (i.e. the decorated version of a Galton--Watson tree with height between $h$ and $2h$). Recall that
\[
\tad=\frac{1}{3(2\alpha - (\sad \wedge 1))}.
\]
Take $\epsilon>0$ as in Assumption \ref{assn:sec}(D). Then there exists a constant $c_{\epsilon} < \infty$ such that for all $\lambda \geq 1$:
\begin{itemize}
\item If $\sad \neq 1, \prb{\Heightdec (\TERh) \geq h^{\frac{1}{\sad \wedge 1}} \lambda} \leq c_{\epsilon}\lambda^{-4\tad/3}$.
\item If instead $\sad = 1, \prb{\Heightdec (\TERh) \geq \lambda h \log h} \leq c_{\epsilon}\lambda^{-4\tad/3}$.
\end{itemize}
\end{prop}

As we see below, this directly implies the result of Proposition \ref{prop:decorated height bound without log term prog}.

\begin{proof}[Proof of Proposition \ref{prop:decorated height bound without log term prog} given Proposition \ref{prop:decorated height bound without log term}]
Given $n \geq 1$, set $h_n = n^{1-\frac{1}{\alpha}}$. Given also $\lambda \geq 1$, define 
\begin{align*}
f_n(\lambda) &= \sup \{x > 0: \prb{\Height (T_{[n,2n]}) \leq xh_n} \leq \lambda^{-1}\}, \\
g_n(\lambda) &= \inf \{x > 0: \prb{\Height (T_{[n,2n]}) \geq xh_n} \leq \lambda^{-1}\}.
\end{align*}
By \cite[Equation above (3)]{KortSubexp} we have that there exists $\lambda_{0}<\infty$ such that $f_n(\lambda) \geq (\log \lambda)^{-\frac{2(\alpha - 1)}{\alpha}}$ for all $\lambda \geq \lambda_{0}$, and by the convergence of \cite[Equation (1)]{KortSubexp} we have that there exists $N<\infty$ such that $\inf_{n \geq N} g_n(\lambda) \to \infty$ as $\lambda \to \infty$. It therefore follows from Lemmas \ref{lem:prog tail bound} and \ref{lem:tree height bound slack} that there exists $c>0$ such that, for an unconditioned Galton--Watson tree $T$ and all $\lambda \geq \lambda_{0}$,
\begin{align*}
\prcondb{|T| \in [n, 2n]}{\Height (T) \in [f_n(\lambda)h_n, g_n(\lambda)h_n]}{} &= \frac{\prb{\Height (T_{[n,2n]}) \in [f_n(\lambda)h_n, g_n(\lambda)h_n]} \prb{|T| \in [n,2n]}}{\prb{\Height (T) \in [f_n(\lambda)h_n, g_n(\lambda)h_n]}} \\
&\geq cf_n(\lambda)^{\frac{1}{\alpha - 1}} \geq c(\log \lambda)^{-\frac{2}{\alpha}}.
\end{align*}
When $\sad \neq 1$, we can therefore write that 
\begin{align*}
&\prb{\Heightdec (\TER_{[n, 2n]}) \geq n^{(\frac{1}{\alpha d}) \vee (1-\frac{1}{\alpha})} \lambda} \\
&\leq \prb{\Height (T_{[n, 2n]}) \notin [f_n(\lambda)h_n, g_n(\lambda)h_n]} + \prcondb{\Heightdec (\TER_{[n, 2n]}) \geq n^{(\frac{1}{\alpha d}) \vee (1-\frac{1}{\alpha})} \lambda}{\Height (T_{[n, 2n]}) \in [f_n(\lambda)h_n, g_n(\lambda)h_n]}{} \\
&\leq 2\lambda^{-1} + \prcondb{\Heightdec (\TER) \geq n^{(\frac{1}{\alpha d}) \vee (1-\frac{1}{\alpha})} \lambda}{\Height (T) \in [f_n(\lambda)h_n, g_n(\lambda)h_n], |T| \in [n,2n]}{} \\
&\leq 2\lambda^{-1} + c^{-1}(\log \lambda)^{\frac{2}{\alpha}}\prcondb{\Heightdec (\TER) \geq n^{(\frac{1}{\alpha d}) \vee (1-\frac{1}{\alpha})} \lambda}{\Height (T) \in [f_n(\lambda)h_n, g_n(\lambda)h_n]}{}.
\end{align*}
Note that it follows from \cite[Theorem 2]{KortSubexp} that we can increase $\lambda_{0}$ a bit if necessary so that $g_n(\lambda) \leq (\log \lambda)^{\frac{2}{\alpha}}$ for all $\lambda \geq \lambda_{0}$ as well. Therefore, by conditioning more precisely on the height we obtain from Proposition \ref{prop:decorated height bound without log term} that this latter probability is upper bounded by $c_{\epsilon}\lambda^{-4\tad/3} g_n(\lambda)^{\frac{1}{\sad \wedge 1}}$, so substituting back we deduce that, for all $\lambda \geq \lambda_0$,
\begin{align*}
\prb{\Heightdec (\TER_{[n, 2n]}) \geq n^{(\frac{1}{\alpha d}) \vee (1-\frac{1}{\alpha})} \lambda} &\leq (2+c^{-1})(\log \lambda)^{\frac{2}{\alpha}(1+\frac{1}{\sad \wedge 1})}\lambda^{-4\tad/3},
\end{align*}
which implies the result (we can change the multiplicative constant so that it in fact holds for all $\lambda \geq 1$). The proof is the same in the case $\sad \neq 1$.
\end{proof}

We first give a non-rigorous explanation of the proof of Proposition \ref{prop:decorated height bound without log term}. We take the notation as in Proposition \ref{prop:spinal offspring UB general conditioned}, and consider the framework of Section \ref{sctn:spinal decomp} with $Y_i = \diam (G(u_i))$. We let $H_h = \Height (T^{[h,2h]})$, let $s_0, \ldots, s_{H_h}$ denote the spinal vertices on the Williams' spine, and let $\sER_0, \ldots, \sER_{H_Y}$ denote the spinal vertices associated with the sequence $(Y_i)_{i < |T^{[h,2h]}|}$. We call this latter collection of spinal vertices the \textbf{decorated spine}. For the rest of the proof we set $\HER_h = H_Y$ and $(\HER_h)^{\wedge} = H_Y^{\wedge}$ with this specific choice of $(Y_i)_{i < |T^{[h,2h]}|}$ (recall from Section \ref{sctn:spinal decomp} that this means that $(\HER_h)^{\wedge} = \sup \{i \geq 0: s_i = \sER_i\}$). The strategy is to instead bound 
\[
1+\sum_{i=0}^{H_h} \diam (G(s_i)) + \sum_{i=(\HER_h)^{\wedge}}^{\HER_h} \diam (G(\sER_i)).
\]
(The additional $+1$ added on the LHS is because of our rooting convention, which is useful for technical reasons in some other proofs). The first sum can be controlled straightforwardly using similar arguments to those used to prove Proposition \ref{prop:TERa height prob LB}: in fact it follows directly from Assumption \ref{assn:sec}(D) and Lemma \ref{lem:stable composition tail} applied to the pair $(\deg s_i, \diam (G(s_i)))$ that we can stochastically dominate the sequence $(\diam ( G(s_i)))_{i \leq H_h}$ by an i.i.d. sequence satisfying upper bounds complementary to those in Lemma \ref{lem:ER spine offspring dist}. Therefore, it follows from Lemma \ref{lem:stable sum tail prob app} that the probability that the first sum above exceeds $\lambda h^{\frac{1}{\sad \wedge 1}}$ (with the extra $\log h$ term when $\sad=1$) is bounded by a term of the desired form. We therefore focus on the second sum for the rest of the proof.

For the second sum, we decompose by writing
\begin{equation}\label{eqn:Yi def}
Y_i = \diam (G(u_i)) = \overline{Y}_i (\deg u_i)^{\frac{1}{d}}
\end{equation}
for each $i \in \{0, \ldots, |T^{[h,2h]}|-1 \}$. We consider the behaviour of $\left( \overline{Y}_i \right)_{i < |T^{[h,2h]}|}$ and $\left( (\deg u_i)^{\frac{1}{d}} \right)_{i < |T^{[h,2h]}|}$ separately. 
We note the following.
\begin{enumerate}[(I)]
\item Under Assumption \ref{assn:sec}(D), $\left( \overline{Y}_i \right)_{i < |T^{[h,2h]}|}$ can be stochastically dominated by a sequence of i.i.d. random variables $\left( \tilde{Y}_i \right)_{i < |T^{[h,2h]}|}$ satisfying the assumptions of \cite[Theorem 1.2]{MarzoukStableSnake}. Modulo some technicalities, this theorem says that the snake process $(S_i)_{i = 0}^{|T^{[h,2h]}|-1}$ defined by $S_i = \sum_{j: j \preceq i} \tilde{Y}_j$ therefore converges under rescaling to an appropriately defined continuum snake process on a stable tree. We will not use the full power of this statement; however, this essentially means that the random variables $\left( \tilde{Y}_i \right)_{i < |T^{[h,2h]}|}$ can be tightly controlled and we will use ideas from \cite{MarzoukStableSnake} to simultaneously control their values along all branches of the underlying tree $T^{[h,2h]}$.
\item By Proposition \ref{prop:spinal offspring UB general}, a size-biased upper bound on the tails of $\deg \sER_j$ holds independently of the values of $\left(\deg \sER_i\right)_{i > j}$ and $\left(\tilde{Y}_{i}\right)_{i > j}$ for each $(\HER_h)^{\wedge} \leq j \leq \HER_h$.
\end{enumerate}
 Then, since $\HER_h \leq \Height (T^{[h,2h]}) \leq 2h$, we deduce from Lemma \ref{lem:stable sum tail prob app} that (if $\sad \neq 1$, otherwise we include an extra $\log h$ term):
\begin{align}\label{eqn:degree and marzouk results}
\begin{split}
\lim_{\lambda \to \infty} \limsup_{h \to \infty} \prb{\sum_{i=(\HER_h)^{\wedge}}^{\HER_h} (\deg \sER_i)^{\frac{1}{d}} \geq \lambda h^{\frac{1}{\sad \wedge 1}}} &= 0 \\
\lim_{\lambda \to \infty} \limsup_{h \to \infty} \prb{\sum_{i=(\HER_h)^{\wedge}}^{\HER_h} \tilde{Y}_i \geq \lambda h} &= 0.
\end{split}
\end{align}

It is reasonable to expect in light of Lemma \ref{lem:stable composition tail} that, when considering the sum $\sum_{i=(\HER_h)^{\wedge}}^{\HER_h} \tilde{Y}_i (\deg \sER_i)^{\frac{1}{d}}$, the heaviest tails above will dominate, i.e. that 
\begin{align*}
\lim_{\lambda \to \infty} \limsup_{h \to \infty} \prb{\sum_{i=(\HER_h)^{\wedge}}^{\HER_h} (\deg u_i)^{\frac{1}{d}} \tilde{Y}_i \geq \lambda h^{\frac{1}{\sad \wedge 1}}} &= 0.
\end{align*}
If $\sad = 1$ we similarly expect
\begin{align*}
\lim_{\lambda \to \infty} \limsup_{h \to \infty} \prb{\sum_{i=(\HER_h)^{\wedge}}^{\HER_h} (\deg u_i)^{\frac{1}{d}} \tilde{Y}_i \geq \lambda h\log h} &= 0.
\end{align*}

We formalise this logic in the next subsections. In what follows, we let $\left( \tilde{Y}_i \right)_{i < \infty}$ be a sequence of i.i.d. random variables satisfying the same upper tail bound in Assumption \ref{assn:sec}(D), and such that for all $h$ and all possible values of $|T^{[h,2h]}|$, $\left( \tilde{Y}_i \right)_{i < |T^{[h,2h]}|}$ stochastically dominates the sequence $\left( \overline{Y}_i \right)_{i < |T^{[h,2h]}|}$ appearing in \eqref{eqn:Yi def}. The existence of such a sequence $\left( \tilde{Y}_i \right)_{i < \infty}$ follows directly from Assumption \ref{assn:sec}.

\textbf{The Galton--Watson forest}

Before we can implement the strategy outlined above, we need a way to sample $T^{[h,2h]}$. To do this, rather than considering a single tree we instead consider a \textit{forest} of Galton--Watson trees with offspring distribution $\xi$ satisfying \eqref{eqn:offspring tails}; that is, an i.i.d. sequence $T^{(0)}, T^{(1)}, \ldots$ such that for each $i \geq 1$, $T^{(i)}$ is a Galton--Watson trees with offspring distribution $\xi$ satisfying \eqref{eqn:offspring tails}. As described in \cite[Section 0.2]{LeGDuqMono}, we can consider the \textit{height function} of this forest by concatenating the height functions of each of the individual trees. Moreover, by the only lemma in \cite[Section 0.2]{LeGDuqMono}, if $(W_m)_{m \geq 0}$ is a discrete-time random walk with jump distribution $\xi-1$, then $W$ plays the role of the Lukasiewicz path of this height function; in other words, if we define the function $(H_m)_{m \geq 0}$ by
\begin{equation}\label{eqn:H from W}
H_m = |\{\ell \leq m: W_{\ell} = \inf_{\ell \leq p \leq m}W_p\}|,
\end{equation}
then $(H_m)_{m \geq 0}$ has the law of this concatenated height function (cf \eqref{eqn:H from W background}). Moreover, each subtree $T^{(i)}$ then corresponds to an excursion between successive new infima of $W$. For the rest of this subsection, we therefore assume wlog that $H$ is defined from $W$ in this way.

It then follows by construction that the first tree appearing in this sequence with height in $[h,2h]$ has the law of $T^{[h,2h]}$. Our strategy will then be to run step (I) above on a large forest, before combining with (II) on this specific tree. More precisely, we proceed as follows:
\begin{enumerate}
\item We run the random walk $W$ until time $n_{h, \lambda} := h^{\frac{\alpha}{\alpha - 1}}(\log \lambda)^a$, where $a=\frac{\alpha}{2(2\alpha - (\sad \wedge 1))}$. This is long enough to ensure that a subtree with law $T^{[h,2h]}$ is very likely to appear in the sequence of trees coded by $W$.
\item We prove that $H$ defined by \eqref{eqn:H from W} is well-behaved on the interval $[0,n_{h, \lambda}]$ (Lemma \ref{lem:marzouk height holder result}), and in particular satisfies the assumptions of Kolmogorov's continuity criterion with a particular H\"older exponent (this essentially follows from the result of \cite{MarzoukStableSnake}).
\item This in turn implies that the snake process with increments given by the variables $(\tilde{Y}_i)_{i < n_{h, \lambda}}$ satisfies Kolmogorov's condition with another H\"older exponent, and is therefore also well-behaved on the whole interval $[0, n_{h,\lambda}]$ (Lemmas \ref{lem:tilde Y tail decay} and \ref{lem:snake holder result}). This in particular includes the interval corresponding to $T^{[h,2h]}$.
\item This good behaviour of the snake process is then sufficient to combine with the result of Proposition \ref{prop:spinal offspring UB general conditioned} and formalise the argument suggested in the previous subsection.
\end{enumerate}

Given this, the main inputs to the proofs are precise versions of the statements in \eqref{eqn:degree and marzouk results} with quantitative tail decay. In the case of the degrees this follows directly from Proposition \ref{prop:spinal offspring UB general conditioned}; in the case of the random variables $(\tilde{Y}_i)_{i < n}$ we will use a series of lemmas.

Throughout the lemmas there will be three parameters $a, k$ and $q$. In end we will take 
\begin{equation}\label{eqn:kq defs}
a=\frac{\alpha}{2(2\alpha - (\sad \wedge 1))}, \qquad q=\frac{\alpha - 1}{2\alpha - (\sad \wedge 1)}, \qquad k=1,
\end{equation}
and it may be helpful to keep this in mind throughout. Although these parameters will eventually just be fixed constants, we state some of the lemmas slightly more generally to allow some flexibility in the choice of $k$ and $q$ in case this is useful in future. We keep track of their knock-on effect on other constants through the use of subscripts throughout the lemmas.

\textbf{Proof of Proposition \ref{prop:decorated height bound without log term}}

\begin{lem}\label{lem:tilde Y tail decay}
Take $\epsilon > 0$ as in Assumption \ref{assn:sec}$(D)$, and $a$ as in \eqref{eqn:kq defs}. There exist $c'<\infty$ and $N_{\epsilon} < \infty$ such that for all $q>0$ there exists $C_{q, \epsilon} < \infty$ such that for all $h \geq N_{\epsilon}$ and all $\lambda \geq 1$,
\[
\prb{\sum_{i=(\HER_h)^{\wedge}}^{\HER_h} \tilde{Y}_i \geq h \lambda^{q}} \leq C_{q, \epsilon}\lambda^{-\left(\frac{q\alpha}{\alpha - 1}-a\right)} + c'(\log \lambda)^2\lambda^{-\frac{a}{\alpha}}.
\]
\end{lem}

For now we prove the proposition assuming Lemma \ref{lem:tilde Y tail decay}. We will prove Lemma \ref{lem:tilde Y tail decay} in the subsequent subsection.

\begin{proof}[Proof of Proposition \ref{prop:decorated height bound without log term}, assuming Lemma \ref{lem:tilde Y tail decay}]

Take some $h \geq 1$ and all other notation as in Lemma \ref{lem:tilde Y tail decay}. As outlined above, it is sufficient to prove the tail bound for the quantity
\[
\sum_{i=(\HER_h)^{\wedge}}^{\HER_h} \diam (G(\sER_i)).
\]
Using the decomposition in \eqref{eqn:Yi def} and Assumption \ref{assn:sec}(D) and taking $(\tilde{Y}_i)_{i < \infty}$ as above, we have that this quantity is stochastically dominated by 
\[
\sum_{i=(\HER_h)^{\wedge}}^{\HER_h} \tilde{Y}_i (\deg u_i)^{\frac{1}{d}}.
\]
We therefore work with this latter quantity throughout the proof. 

Conditionally on $\HER_h$, we also define $k_i = |\sER_i|$ to be the lexicographical index of the $i^{th}$ vertex on the decorated spine of $T^{[h,2h]}$. (For this proposition we are just considering the lexicographical ordering on $T^{[h,2h]}$ as a single tree, and not the forest mentioned in the previous subsection).

We write the full proof in the case $\sad < 1$. For notational convenience, for any $\lambda \geq 1$ we define the event
\begin{align*}
F_{h} = \left\{ \sup_{u \in T^{[h,2h]}} (\deg u)^{\frac{1}{d}} < h^{\frac{1}{(\alpha-1) d}} \lambda \right\}.
\end{align*}
We start by working conditionally on $\HER_h$ and $\left(\tilde{Y}_{k_i}\right)_{i \leq \HER_h}$. Since $\sad < 1$, note from Proposition \ref{prop:spinal offspring UB general conditioned} that there exists $c'<\infty$ such that for each $i \leq \HER_h$ and all $\lambda \geq 1$,
\begin{align*}
\econdb{(\deg \sER_i)^{\frac{1}{d}} \mathbbm{1}\left\{ F_{h}\right\}}{ \HER_h, \left(\tilde{Y}_{k_i}\right)_{i \leq \HER_h}}{} 
&\leq c'h^{\frac{1}{(\alpha-1) d}(1-\sad)} \lambda^{1-\sad}.
\end{align*}
Therefore, applying Markov's inequality and then applying the above upper bound on the expectation we deduce that for any $\lambda \geq 1$,
\begin{align}\label{eqn:Markov conditional comp}
\begin{split}
&\prcondb{ \sum_{i \leq \HER_h} \tilde{Y}_{k_i}(\deg \sER_i)^{\frac{1}{d}} \geq h^{\frac{1}{(\alpha-1) d}} \lambda \text{ and } F_{h} \text{ and} \sum_{i \leq \HER_h} \tilde{Y}_{k_i} < h \lambda^{q}}{ \HER_h, \left(\tilde{Y}_{k_i}\right)_{i \leq \HER_h}}{} \\
&\qquad \leq h^{-\frac{1}{(\alpha-1) d}} \lambda^{-1} \econdb{\left(\sum_{i \leq \HER_h} \tilde{Y}_{k_i}(\deg \sER_i)^{\frac{1}{d}} \right) \mathbbm{1}\left\{F_{h} \text{ and} \sum_{i \leq \HER_h} \tilde{Y}_{k_i} < h \lambda^{q}\right\}}{ \HER_h, \left(\tilde{Y}_{k_i}\right)_{i \leq \HER_h}}{} \\
&\qquad = h^{-\frac{1}{(\alpha-1) d}} \lambda^{-1} \left( \sum_{i \leq \HER_h} \tilde{Y}_{k_i} \econdb{(\deg \sER_i)^{\frac{1}{d}} \mathbbm{1}\left\{F_{h}\right\}}{ \HER_h, \left(\tilde{Y}_{k_i}\right)_{i \leq \HER_h}}{} \right) \mathbbm{1}\left\{ \sum_{i \leq \HER_h} \tilde{Y}_{k_i} < h \lambda^{q}\right\} \\
&\qquad \leq h^{-\frac{1}{(\alpha-1) d}} \lambda^{-1} c'h^{\frac{1}{(\alpha-1) d}(1-\sad)} \lambda^{1-\sad}  \sum_{i \leq \HER_h} \tilde{Y}_{k_i} \mathbbm{1}\left\{ \sum_{i \leq \HER_h} \tilde{Y}_{k_i} < h \lambda^{q}\right\} \\
&\qquad \leq c' \lambda^{-(\sad-q)}.
\end{split}
\end{align}
Finally, we integrate over $\HER_h$ and $(\tilde{Y}_{k_i})_{i \leq \HER_h}$ to obtain the same bound for the unconditional probability. Then, applying this in a union bound with the events above we obtain that
\begin{align*}
\prb{ \sum_{i \leq \HER_h} \tilde{Y}_{k_i}(\deg \sER_i)^{\frac{1}{d}} \geq h^{\frac{1}{(\alpha-1) d}} \lambda} &\leq \prb{F_{h}^c} + \prb{\sum_{i \leq \HER_h} \tilde{Y}_{k_i} \geq h \lambda^{q}} + c' \lambda^{-(\sad-q)} \\
&\leq c'\lambda^{-d\alpha} + C_{q, \epsilon}\lambda^{-\left(\frac{q\alpha}{\alpha - 1}-a\right)} + c'(\log \lambda)^2\lambda^{-\frac{a}{\alpha}} + c' \lambda^{-(\sad-q)}.
\end{align*}
(Here to obtain that $\prb{F_{h}^c} \leq c'\lambda^{-d\alpha}$ we use \cite[Equation (3)]{KortSubexp} with Bayes' formula, Lemma \ref{lem:prog tail bound}$(i)$ and Lemma \ref{lem:tree height bound slack} to deduce that $|T_{[h, 2h]}|$ is typically of order $h^{\frac{\alpha}{\alpha - 1}}$ and is very well concentrated, and then combine with Lemma \ref{lem:GW typical degree prob bound} in a union bound).

Recall we defined $q=\frac{\alpha - 1}{2\alpha - \sad}$ and $a=\frac{\alpha}{2(2\alpha - \sad)}$, which gives an upper bound of $c_{\epsilon}\lambda^{-\frac{4}{9(2\alpha - \sad)}}$ (actually even better than this), where $c_{\epsilon}$ depends on the $\epsilon$ appearing in Assumption \ref{assn:sec}(D).

In the cases $\sad = 1$ and $\sad >1$, we instead have
\[
\econdb{(\deg \sER_i)^{\frac{1}{d}} \mathbbm{1}\{ F_{h}\}}{ \HER_h, \left(\tilde{Y}_{k_i}\right)_{i \leq \HER_h}}{} \leq \begin{cases} c'\log (h \lambda) &\text{ if } \sad=1, \\
c' &\text{ if } \sad>1.
\end{cases}
\]
For the computation with Markov's inequality in \eqref{eqn:Markov conditional comp} to work as above we therefore need to replace $h^{\frac{1}{(\alpha-1) d}} \lambda$ with $h \lambda \log h$ and $h \lambda$ respectively.

When $\sad = 1$ the final line in \eqref{eqn:Markov conditional comp} is therefore instead
\[
 h^{-1} (\log h)^{-1} \lambda^{-1} c' (\log (h \lambda)) h \lambda^{q} = c \lambda^{-(1-q)} \log \lambda.
\]
Again since $q=\frac{\alpha - 1}{2\alpha - 1}$ and $a=\frac{\alpha}{2(2\alpha - 1)}$ in this case, this gives overall tail decay (in fact a bit better than) order $c_{\epsilon}\lambda^{-\frac{4}{9(2\alpha - 1)}}$. Here we also use that $\log (h \lambda) \leq \log (h) \log ( \lambda)$ whenever $h \wedge \lambda \geq e^2$.

When $\sad > 1$ the final line in \eqref{eqn:Markov conditional comp} is therefore instead
\[
h^{-1} \lambda^{-1} c' h \lambda^{q} = c' \lambda^{-(1-q)}.
\]
We also do not need to include the event $F_{h}$ in the argument in this case, so the same argument therefore gives that
\begin{align*}
\prb{ \sum_{i \leq \HER_h} \tilde{Y}_{k_i}(\deg \sER_i)^{\frac{1}{d}} \geq h \lambda} 
&\leq C_{q, \epsilon}\lambda^{-\left(\frac{q\alpha}{\alpha - 1}-a\right)} + c'(\log \lambda)^2\lambda^{-\frac{a}{\alpha}} + c' \lambda^{-(1-q)}.
\end{align*}
Again since $q=\frac{\alpha - 1}{2\alpha - 1}$ and $a=\frac{\alpha}{2(2\alpha - 1)}$, this easily gives an upper bound of $c_{\epsilon}\lambda^{-\frac{4}{9(2\alpha - 1)}}$.
\end{proof}


We are left with proving Lemma \ref{lem:tilde Y tail decay}. 

\textbf{Proof of Lemma \ref{lem:tilde Y tail decay}}

In what follows we use the following notation. First let $(W_m)_{m \geq 0}$ be a random walk with jump distribution $\xi-1$, and let $(H_m)_{m \geq 0}$ be defined from $W$ via \eqref{eqn:H from W}. Given $h \geq 1$, we consider $W$ and $H$ on the time interval $[0, n_{h, \lambda}]$ where $n_{h, \lambda} = h^{\frac{\alpha}{\alpha - 1}}\lambda^a$ and where $a$ is as in \eqref{eqn:kq defs}, and let $T^{(0)}, T^{(1)}, \ldots, T^{(N_{h,\lambda})}$ denote the ordered set of complete trees coded by $W$ on the interval $[0, n_{h,\lambda}]$. If there is an incomplete tree at the end of the sequence, we discard it by defining $(H^{(h, \lambda)}_m)_{m \in [0, n_{h, \lambda}]}$ by
\begin{equation}
H^{(h, \lambda)}_m = H_m \mathbbm{1}\{\arg \min_{[0, n_{h, \lambda}]} X \geq m\} \qquad \text{for } m \in [0, n_{h, \lambda}].
\end{equation}
(This means that $H^{(h, \lambda)}$ codes precisely the first $N_{h,\lambda}$ trees appearing in the forest). We extend the lexicographical ordering to the whole forest and let $u_0, u_1, \ldots, u_{n_{h, \lambda}}$ denote the ordered vertices of the trees of the forest. For $h \geq 1$ we let $\tau_h = \inf\{i \geq 1: \Height (T^{(i)}) \in [h,2h]\}$, or $0$ if this set is empty, and decorate $T^{(\tau_h)}$ as described in Section \ref{sctn:dec tree def}.

Recall that $(\tilde{Y}_i)_{i <\infty}$ is a sequence of i.i.d. random variables satisfying the tail bound of Assumption \ref{assn:sec}(D). For each $j \in [0, \infty)$ we let $\Znv = \sum_{i:u_i \preceq u_j} \Ynu$ denote the corresponding snake process on the interval $[0, \infty)$, where $u_i \preceq u_j$ has exactly the same meaning as in Section \ref{sctn:GW tree defs}.

We first state a lemma controlling $\max_{j \leq n_{h, \lambda}} \Znv$.

\begin{lem}\label{lem:snake holder result}
Take $\epsilon > 0$ as in Assumption \ref{assn:sec}$(D)$. There exists $c' < \infty$ such for all $k>0$ and $p \in (\frac{\alpha}{\alpha -1}, {\frac{\alpha +\epsilon}{\alpha -1}})$ the following holds. There exists $N_p < \infty$ and an event $E_{h, \lambda,k}$ such that for any $q>0$ there exists $C_{k, p, q} < \infty$ such that for all $h \geq N_{p}$ and all $\lambda \geq 1$,
\begin{align*}
\prb{E_{h, \lambda,k}^c} \leq c'\lambda^{-k}, \qquad \text{ and } \qquad \prb{\sup_{i \leq n_{h, \lambda}} \Znu \geq \lambda^q n_{h, \lambda}^{1-\frac{1}{\alpha}} \text{ and } E_{h, \lambda, k}} \leq C_{k, p, q}\lambda^{-qp}.
\end{align*}
\end{lem}

We now show how this implies Lemma \ref{lem:tilde Y tail decay}. We give the proof of Lemma \ref{lem:snake holder result} afterwards.

\begin{proof}[Proof of Lemma \ref{lem:tilde Y tail decay}, assuming Lemma \ref{lem:snake holder result}]
First note that on the event $\tau_h \neq 0$, we can couple $T^{(\tau_h)}$ and $T^{[h,2h]}$ so that they are identical. By Lemmas \ref{lem:prog tail bound} and \ref{lem:stable sum tail prob app}$(i)$ there exist $c,c' \in (0, \infty)$ such that for all $\lambda, h \geq 1$, the probability that this coupling fails is upper bounded by
\begin{align}\label{eqn:failed coupling}
\begin{split}
\pr{N_{h,\lambda} < h^{\frac{1}{\alpha - 1}}(\log \lambda)^2} + \pr{\tau_h = 0, N_{h,\lambda} \geq h^{\frac{1}{\alpha - 1}}(\log \lambda)^2}
&\leq (\log \lambda)^2\lambda^{-\frac{a}{\alpha}} + (1-ch^{-\frac{1}{\alpha-1}})^{h^{\frac{1}{\alpha - 1}}(\log \lambda)^2} \\
&\leq c'(\log \lambda)^2\lambda^{-\frac{a}{\alpha}}.
\end{split}
\end{align}
We henceforth now assume that $\tau_h \neq 0$, and therefore that $T^{[h,2h]} = T^{(\tau_h)}$ under this coupling. By substituting the value of $n_{h, \lambda}$, we have from Lemma \ref{lem:snake holder result} that for any $k \geq 1, q >0$ and $p \in (\frac{\alpha}{\alpha -1}, {\frac{\alpha +\epsilon}{\alpha -1}})$ there exists $C_{k,p,q} < \infty$ and $N_p<\infty$ such that for all $h \geq N_p$ and all $\lambda \geq 1$:
\[
\prb{\sup_{i \leq n_{h, \lambda}} \Znu \geq \lambda^q h} \leq \prb{E_{h,\lambda, k}^c} + \prb{\sup_{i \leq n_{h, \lambda}} \Znu \geq \lambda^q h \text{ and } E_{h, \lambda, k}} \leq c'\lambda^{-k} + C_{k, p, q}\lambda^{-[q-a(1-\frac{1}{\alpha})]p}.
\]
Moreover, clearly
\begin{equation}\label{eqn:snake branch bound}
\sum_{i=(\HER_h)^{\wedge}}^{\HER_h} \tilde{Y}_i 
\leq \sup_{i \leq n_{h,\lambda}} \Znu.
\end{equation}
Lemma \ref{lem:tilde Y tail decay} therefore follows on fixing $p = \frac{2\alpha + \epsilon}{2(\alpha - 1)}$ and $k=1$ and combining with \eqref{eqn:failed coupling} in a union bound.
\end{proof}

\textbf{Proof of Lemma \ref{lem:snake holder result}}

The arguments to prove Lemma \ref{lem:snake holder result} use the same principles as those employed in \cite{MarzoukStableSnake}, but we just tweak some arguments as our aims here are slightly different (in particular, we would like to obtain a tail bound that decays uniformly in $\lambda$ and $n$ or $h$, so cannot discard $o_n(1)$ or $o_h(1)$ terms in the same way as in \cite{MarzoukStableSnake}). 

The proof is really just a minor modification of the arguments of \cite[proof of Lemma 1.4]{MarzoukStableSnake}; we give the details for completeness. We start with a definition.

If $i,j \in [0, n_{h, \lambda}]$, we write $i \overset{T}{\sim} j$ if $u_i$ and $u_j$ are in the same tree of the forest. We let $d(u_i,u_j)$ denote the distance between them in this case, and extend the distance function to the forest by writing
\[
d^{h, \lambda}(u_i,u_j) = \begin{cases}
d(u_i,u_j) \mathbbm{1}\{\arg \min_{[0, n_{h, \lambda}]} X \geq i \vee j\} &\text{ if } u_i \overset{T}{\sim} u_j, \\
H^{(h, \lambda)}_i + H^{(h, \lambda)}_j &\text{ otherwise}.
\end{cases}
\]
In the latter case we will sometimes think of the respective roots as representing the most recent common ancestor of $u_i$ and $u_j$, and think of the union of the paths to their respective roots as the branch between $u_i$ and $u_j$.

\begin{defn}\label{def:kappan def}
Fix any $\gamma < \frac{\alpha - 1}{\alpha}$. For each $h\geq 1, \lambda \geq 1$ we define the random variable $\kappa_{h,\gamma,\lambda}$ by
\begin{align*}
\kappa_{h,\gamma,\lambda} = \sup_{i,j \leq n_{h, \lambda}} \left\{ \frac{n_{h, \lambda}^{-(1-\frac{1}{\alpha})}d^{h, \lambda}(u_i,u_j)}{(\log \lambda)\left(n_{h,\lambda}^{-1}|i-j|\right)^{\gamma}} \right\}.
\end{align*}
\end{defn}

The result of the next lemma will enable us to later apply Kolmogorov's continuity criterion to deduce tightness of the rescaled snake process on the interval $[0, n_{h, \lambda}]$.

\begin{lem}\label{lem:marzouk height holder result}
Take any $\gamma < \frac{\alpha - 1}{\alpha}$. There exists $N_{\gamma}<\infty$ such that for all $h \geq N_{\gamma}$ the following holds. Fix any $k\geq 1$ and any $\lambda \geq 3$. There exists an event $E_{h, \lambda,k}$ and a constant $c'$ such that for all $p>0$ there exists a constant $C_{p, \gamma, k}< \infty$, not depending on $\lambda$, such that
\begin{align*}
\prb{E_{h, \lambda,k}^c} \leq c'\lambda^{-k} \qquad \text{ and } \qquad \Eb{\kappa_{h,\gamma,\lambda}^p\mathbbm{1}\{E_{h, \lambda,k}\}} &< C_{p, \gamma, k}.
\end{align*}
\end{lem}
\begin{proof}
For now assume a constant $A_{\xi}$ is fixed (we will choose it later), take $\lambda \geq 1$ and for $k \geq 1$ define $E_{h, \lambda,k}$ to be the event (cf \cite[Lemma 2.1]{MarzoukStableSnake})
\[
\left\{\Theta_{i,j} \leq 1 - \frac{\xi (0)}{2} \text{ for all } i \preceq j \in [0, n_{h, \lambda}] \text{ such that } i \overset{T}{\sim}j \text{ and } d^{h, \lambda}(u_i,u_j) \geq A_{\xi} \log (n_{h, \lambda} \lambda^k)\right\},
\]
where 
\[
\Theta_{i,j} = \frac{|\{\ell \in (i,j]: \chi_{u_\ell}= \deg (\p (u_\ell)) - 1\}|}{d^{h,\lambda}(u_i, u_j)}.
\]
(Recall from Section \ref{sctn:GW tree defs} that for a vertex $u$, $\chi_u$ denotes the order of $u$ amongst its siblings, and $\p(u)$ denotes its parent). Note that, even if the interval $[0, n_{h, \lambda}]$ does not code an integer number of trees, then this event is not affected by what happens on the final partial tree since $d^{h,\lambda}(u_i, u_j)=0$ for all $i \preceq j$ in this final partial tree. (Although the event would still be well-defined on this partial tree anyway, since the Lukasiewicz path tells us the degrees of all the vertices coded by the interval $[0, n_{h, \lambda}]$).

We will prove that following three claims hold, provided that $h$ is sufficiently large depending on $\gamma$. Throughout we let $u_0, \ldots, u_{n_{h,\lambda}}$ denote the vertices of the forest listed in the lexicographical ordering, and assume that $k$ and $\gamma$ are fixed as described in the statement of the lemma.
\begin{enumerate}[A.]
\item We can choose $A_{\xi}< \infty$ depending only on $\xi(0)$ and $c' < \infty$ depending only on the constant $c$ in \eqref{eqn:offspring tails} so that $\prb{E_{h, \lambda,k}^c} \leq c'\lambda^{-k}$ for all $k \geq 1$ and all $\lambda \geq 3$.
\item There exist constants $B_{\xi}< \infty, b_{\xi}>0$, depending only on $\xi(0)$ and $\alpha$, such that for all $i \leq j \leq n_{h,\lambda}$, all $k \geq 1$, all $\lambda \geq 3$ and all $x \geq 1$ we have that
\begin{align}
\label{eqn:marz2} \prb{\frac{n_{h, \lambda}^{-(1-\frac{1}{\alpha})}|H^{(h, \lambda)}_j - \inf_{i \leq \ell \leq j} H^{(h, \lambda)}_{\ell}|}{\left(n_{h, \lambda}^{-1}|i-j|\right)^{\gamma}}>x} &\leq B_{\xi}e^{-b_{\xi}x} \\ 
\label{eqn:marz1}
\prb{\frac{n_{h, \lambda}^{-(1-\frac{1}{\alpha})}|H^{(h, \lambda)}_i - \inf_{i \leq \ell \leq j} H^{(h, \lambda)}_{\ell}|\mathbbm{1}\{|H^{(h, \lambda)}_i-\inf_{i \leq \ell \leq j} H^{(h, \lambda)}_{\ell}| \geq A_{\xi} \log (n_{h, \lambda} \lambda^k)\}}{\left(n_{h, \lambda}^{-1}|i-j|\right)^{\gamma}}>x  \text{ and } E_{h, \lambda, k} } &\leq B_{\xi}e^{-b_{\xi}x}.
\end{align}
\item We can choose $A_{\xi}$ in Claim A so that if instead $i, j \leq n_{h,\lambda}$ and $\lambda \geq 1$ but $|H^{(h, \lambda)}_i- H^{(h, \lambda)}_j| < A_{\xi} \log (n \lambda^k)$, we instead have for all $k \geq 1$ and all $\lambda \geq 3$:
\begin{align*}
\frac{n_{h,\lambda}^{-(1-\frac{1}{\alpha})}|H^{(h, \lambda)}_i- H^{(h, \lambda)}_j|}{\left(n_{h,\lambda}^{-1}|i-j|\right)^{\gamma}} \leq A_{\xi}k \log \lambda.
\end{align*}
\end{enumerate}

\textbf{Proof of Claim A.} Take $\lambda \geq 3$, and first note that we can replace $\deg (\p (u_\ell)) - 1$ with $1$ by symmetry. We obtain that $\prb{E_{h, \lambda,k}^c} \leq c'\lambda^{-k}$ by repeating the proof of \cite[Lemma 2.1]{MarzoukStableSnake}. In particular, since the event $E_{h, \lambda, k}$ is not affected by the possible partial tree coded at the end of the interval $[0, n_{h, \lambda}]$, we just need to prove that this event holds on all of the complete trees $T^{(0)}, T^{(1)}, \ldots, T^{(N_{h,\lambda})}$. Therefore, we first consider a single unconditioned Galton--Watson tree $T$ with graph metric $d_T$ and note that, by Markov's inequality,
\begin{align*}
&\prb{|T| \leq n_{h, \lambda} \text{ and } \exists i \preceq j \in [0, |T|-1]: d_T(u_i,u_j) \geq A_{\xi} \log (n_{h, \lambda} \lambda^k) \text{ and } \frac{|\{\ell \in (i,j]: \chi_{u_\ell}= \deg (\p (u_\ell)) - 1\}|}{d_T(u_i, u_j)} \geq 1 - \frac{\xi (0)}{2}} \\
&\leq \Eb{\mathbbm{1}\{|T| \leq n_{h, \lambda}\}\sum_{m =1}^{\Height (T)} \sum_{\substack{j \leq |T|-1: \\ d_T(\rho, u_j)=m}} \mathbbm{1}\left\{\exists i \preceq j: d_T(u_i,u_j) \geq A_{\xi} \log (n_{h, \lambda} \lambda^k) \text{ and } \frac{|\{\ell \in (i,j]: \chi_{u_\ell}= 1\}|}{d_T(u_i, u_j)} \geq 1 - \frac{\xi (0)}{2}\right\}} \\
&\leq \Eb{\sum_{m =1}^{n_{h, \lambda}} \sum_{\substack{v \in T:\\ d_T(\rho, v)=m}} \mathbbm{1}\left\{\exists u \prec v: d_T(u,v) \geq A_{\xi} \log (n_{h, \lambda} \lambda^k)  \text{ and }  \frac{|\{x \in ((u,v]]: \chi_x=1\}|}{d_T(u,v)} \geq 1 - \frac{\xi (0)}{2} \right\}} \\
&\leq \sum_{m =A_{\xi} \log (n_{h, \lambda} \lambda^k)}^{n_{h, \lambda}} \prb{ \exists u \prec s_m^{\infty}:  d_T(u,s_m^{\infty}) \geq A_{\xi} \log (n_{h, \lambda} \lambda^k)  \text{ and }  \frac{|\{x \in ((u,s_m^{\infty}]]: \chi_x=1\}|}{d_{\Ti}(u,s_m^{\infty})} \geq 1 - \frac{\xi (0)}{2}} \\
&\leq \sum_{m =A_{\xi} \log (n_{h, \lambda} \lambda^k) \ \ }^{n_{h, \lambda}} \sum_{i=A_{\xi} \log (n_{h, \lambda} \lambda^k)}^m \prb{\frac{|\{x \in ((s_{m-i}^{\infty},s_m^{\infty}]]: \chi_x=1\}|}{d_{\Ti}(s_{m-i}^{\infty},s_m^{\infty})} \geq 1 - \frac{\xi (0)}{2}}.
\end{align*}
Here the penultimate line follows by an application of the absolute continuity relation of \cite[Equation (24)]{DuqHawkes} between an unconditioned Galton-Watson tree $T$ and Kesten's tree $\Ti$, which entails that
\begin{align*}
&\Eb{\sum_{\substack{v \in T:\\ d_T(\rho, v)=m}} \mathbbm{1}\left\{\exists u \prec v: d_T(u,v) \geq A_{\xi} \log (n \lambda^k)  \text{ and }  \frac{|\{x \in ((u,v]]: \chi_x=1\}|}{d_T(u,v)} \geq 1 - \frac{\xi (0)}{2} \right\}} \\
&\qquad = \Eb{ \mathbbm{1}\left\{\exists u \prec s_m^{\infty}: d_{\Ti}(u,s_m^{\infty}) \geq A_{\xi} \log (n \lambda^k)  \text{ and }  \frac{|\{x \in ((u,s_m^{\infty}]]: \chi_x=1\}|}{d_{\Ti}(u,s_m^{\infty})} \geq 1 - \frac{\xi (0)}{2} \right\}},
\end{align*}
where $s_m^{\infty}$ denotes the $m^{th}$ backbone vertex of $\Ti$. As noted in \cite[proof of Lemma 2.1]{MarzoukStableSnake}, the number of $x \in ((s_{m-i}^{\infty},s_m^{\infty}]]$ fulfilling the event above is a $\textsf{Binomial}(i, 1-\xi(0))$ random variable. Therefore, by a Chernoff bound and since $m \leq n_{h, \lambda}$, we can write:
\begin{align*}
&\sum_{i=A_{\xi} \log (n_{h, \lambda} \lambda^k)}^m \prb{\frac{|\{x \in ((s_{m-i}^{\infty},s_m^{\infty}]]: \chi_x=1\}|}{d_{\Ti}(s_{m-i}^{\infty},s_m^{\infty})} \geq 1 - \frac{\xi (0)}{2}} \\
&=\sum_{i=A_{\xi} \log (n_{h, \lambda} \lambda^k)}^m \prb{\textsf{Binomial}(i, 1-\xi(0)) \geq \left(1-\frac{\xi(0)}{2}\right)i} \leq \sum_{i=A_{\xi} \log (n_{h, \lambda} \lambda^k)}^m e^{-a_{\xi}i} \leq n_{h, \lambda} e^{-a_{\xi}A_{\xi} \log (n_{h, \lambda} \lambda^k)},
\end{align*}
where $a_{\xi}>0$ is a constant depending on $\xi (0)$ only. Substituting into the previous calculation we obtain an upper bound of $n_{h, \lambda}^2 e^{-a_{\xi}A_{\xi} \log (n_{h, \lambda} \lambda^k)}$. Finally, to get an upper bound for $\prb{E_{h, \lambda,k}^c}$ note that there are deterministically at most $n_{h, \lambda}$ complete trees coded on the interval $[0, n_{h, \lambda}]$, and all must satisfy $|T| \leq n_{h, \lambda}$. Therefore we can take a union bound to deduce that, for all $\lambda, k \geq 1$ and all sufficiently large $h$,
\begin{align*}
\prb{E_{h, \lambda,k}^c} \leq n_{h, \lambda}^3 e^{-a_{\xi}A_{\xi} \log (n_{h, \lambda} \lambda^k)}.
\end{align*}
In particular we can choose $A_{\xi}$ large enough, depending only on $\xi(0)$, that the claim holds for all $h \geq 1, \lambda \geq 3$.

\textbf{Proof of Claim B.} Take any $k \geq 1, \lambda \geq 3$. The first statement \eqref{eqn:marz2} is essentially \cite[Equation (5)]{MarzoukStableSnake}; the only difference is that here we are considering a forest rather than a single tree. However, it follows from \eqref{eqn:H from W} and from a time and space reversal, that if $\tau(k)$ denotes the $k^{th}$ record time of $W$, then the probability in question is equal to 
\begin{align*}
\prb{\tau \left(\frac{x\left(n_{h, \lambda}^{-1}|i-j|\right)^{\gamma}}{n_{h, \lambda}^{-(1-\frac{1}{\alpha})}}\right) \leq |j-i|}.
\end{align*}
It is shown in \cite[proof of (5)]{MarzoukStableSnake} that this is upper bounded by the expression on the RHS of \eqref{eqn:marz2}.

We therefore just outline the proof of \eqref{eqn:marz1}. As noted in \cite[proof of (5)]{MarzoukStableSnake}, the two events considered in \eqref{eqn:marz2} and \eqref{eqn:marz1} are not \textit{quite} symmetric, because the children but not the ancestors of $u_i$ are contained in the interval $[i,j]$ in the lexicographical ordering, whereas some ancestors but not the children of $u_j$ are.

To overcome this difficulty, given a tree $T$ Marzouk introduces a \textit{mirror tree} $\tilde{T}$ (cf \cite[Figure 8]{MarzoukStableSnake}), obtained from $T$ by first exchanging the subtrees rooted at $u_i$ and $u_j$ to obtain an intermediate tree, and then flipping (i.e. reversing the contour ordering of) this intermediate tree.

In our case, it is sufficient to again exchange the subtrees rooted at $u_i$ and $u_j$ to obtain an intermediate forest, and then flip all the complete trees in the forest (i.e. reverse the contour ordering on the interval $[0, \arg \min_{[0, n_{h, \lambda}]} X]$). We call the new set of ordered trees the \textit{mirrored forest} and denote it $\tilde{F}$ (note that it depends on $i$ and $j$, but for simplicity we omit this from the notation). By symmetry, $\tilde{F}$ has the same law as the original forest, and moreover if $\tilde{i}$ and $\tilde{j}$ respectively denote the indices of the mirror images of $i$ and $j$ in $\tilde{F}$, it follows by construction that (letting $i\wedge j$ denote the index of the most recent common ancestor of $u_i$ and $u_j$)
\[
|j-i| = \begin{cases} |\tilde{i} - \tilde{j}| - d(u_i, u_{i \wedge j}) + d(u_j, u_{i \wedge j}) &\text{ if } u_i \overset{T}{\sim} u_j, \\
|\tilde{i} - \tilde{j}| - H^{(h, \lambda)}_i + H^{(h, \lambda)}_j &\text{ otherwise. }
\end{cases}
\]
However, on the event $E_{h, \lambda, k}$, we have that
\begin{align*}
|j-i| \geq \begin{cases}
\frac{\xi(0)}{2} d(u_i, u_{i \wedge j}) &\text{ if } u_i \overset{T}{\sim} u_j, \\
\frac{\xi(0)}{2} H^{(h, \lambda)}_i &\text{ otherwise, }
\end{cases}
\end{align*}
and therefore that
\[
|j-i| \geq |\tilde{i} - \tilde{j}| - \frac{2}{\xi(0)}|j-i|, 
\]
which rearranges to 
\begin{equation}\label{eqn:i-j bound}
|j-i| \geq \frac{|\tilde{i} - \tilde{j}|}{1+\frac{2}{\xi(0)}}.
\end{equation}
In particular, applying \eqref{eqn:marz2} on the mirror forest $\tilde{F}$, then applying \eqref{eqn:i-j bound} and using that $\gamma < 1-\frac{1}{\alpha}$, we instead obtain a bound for the tails of the quantity
\[
\frac{n_{h, \lambda}^{-(1-\frac{1}{\alpha})}|H^{(h, \lambda)}_i - \inf_{i \leq \ell \leq j} H^{(h, \lambda)}_{\ell}|\mathbbm{1}\{|H^{(h, \lambda)}_i - \inf_{i \leq \ell \leq j} H^{(h, \lambda)}_{\ell}| \geq A_{\xi} \log (n_{h, \lambda} \lambda^k)\}}{\left(n_{h, \lambda}^{-1}|i-j|\right)^{\gamma}\left(1+\frac{\xi(0)}{2}\right)^{1-\frac{1}{\alpha}}}.
\]
so we can obtain \eqref{eqn:marz1} from \eqref{eqn:marz2} just by replacing $b_{\xi}$ with $\frac{b_{\xi}}{\left(1+\frac{\xi(0)}{2}\right)^{1-\frac{1}{\alpha}}}$.

\textbf{Proof of Claim C.} Take any $k \geq 1, \lambda \geq 3$. This claim is deterministic. If instead $|H^{(h, \lambda)}_i-H^{(h, \lambda)}_j| < A_{\xi} \log (n \lambda^k)$ we instead have, since $|i-j| \geq 1$ and $k \log \lambda$ is bounded away from $0$, that provided $h$ is sufficiently large as a function of $(1-\frac{1}{\alpha}) - \gamma$ then
\begin{align*}
\frac{n_{h, \lambda}^{-(1-\frac{1}{\alpha})}|H^{(h, \lambda)}_i - H^{(h, \lambda)}_j|}{\left(n_{h, \lambda}^{-1}|i-j|\right)^{\gamma}} \leq A_{\xi}n_{h, \lambda}^{\gamma - (1-\frac{1}{\alpha})} (\log n_{h, \lambda} + k \log \lambda) \leq A_{\xi}k \log \lambda.
\end{align*}
\textbf{Conclusion.} Either way, by breaking paths at the most recent common ancestor (or at the roots, if $u_i$ and $u_j$ are in different trees) we can certainly modify $b_{\xi}$ and $B_{\xi}$ a bit if necessary (in a way that depends only on $A_{\xi}$) so that there exists $N_{\gamma} <\infty$ such that for all $h \geq N_{\gamma}$, all $k \geq 1$, all $\lambda \geq 3$ and for all $i,j \leq n_{h, \lambda}$ we have that
\[
\prb{\frac{n_{h, \lambda}^{-(1-\frac{1}{\alpha})}|H^{(h, \lambda)}_i-H^{(h, \lambda)}_j|}{k \log \lambda \left(n_{h, \lambda}^{-1}|i-j|\right)^{\gamma}}>x \text{ and } E_{h, \lambda,k}} \leq B_{\xi}e^{-b_{\xi}x}.
\]
Since this holds uniformly in $h$, by breaking at the most recent common ancestor (or again at the roots, if $u_i$ and $u_j$ are in different trees) and by Kolmogorov's continuity criterion (applied to the rescaled interpolated height function $H^{(h, \lambda)}$, which is indexed by $[0,1]$) gives that, for all $k \geq 1, p>0$ and all $\gamma'<\gamma$, there exists $C_{p, \gamma, \gamma', \xi} < \infty$ such that for all $h \geq N_{\gamma}, \lambda \geq 3$,
\begin{align*}
\Eb{\sup_{i \leq n_{h, \lambda}} \left(\frac{n_{h, \lambda}^{-(1-\frac{1}{\alpha})}d^{h, \lambda}(u_i, u_j)}{k \log \lambda \left(n_{h, \lambda}^{-1}|i-j|\right)^{\gamma'}}\right)^p \mathbbm{1}\{E_{h, \lambda,k}\}} < C_{p, \gamma', \gamma, \xi}.
\end{align*}
In particular, since $\gamma < 1-\frac{1}{\alpha}$, for any fixed $\gamma' < 1-\frac{1}{\alpha}$ we can run the argument with $\gamma = \frac{1}{2} \left(1-\frac{1}{\alpha} + \gamma'\right)$ to obtain a constant $C_{p, \gamma', \xi}$ on the right hand side above. Since the offspring distribution is assumed to be fixed throughout the paper, we suppress the dependence on $\xi$, and moreover remove $k$ from the denominator on the LHS by taking $C_{p, \gamma',k} = k^p C_{p, \gamma', \xi}$. (We replaced $\gamma'$ with $\gamma$ in the final statement).
%
\end{proof}

\begin{proof}[\textbf{Proof of Lemma \ref{lem:snake holder result}}] 
Fix $k \geq 1$. We take $E_{h, \lambda, k}$ as in Lemma \ref{lem:marzouk height holder result}. The proof essentially follows the same steps as \cite[Theorem 1.2]{MarzoukStableSnake}. 
 Similarly to there, we first note by Jensen's inequality that if $S_m = \sum_{i=1}^m \Ynui$ for some $m \geq 1$ and some fixed subsequence $(j_i)_{j=1}^m$, and if $p \in (\frac{\alpha}{\alpha -1} , {\frac{\alpha + \epsilon}{\alpha -1}})$ where $\epsilon>0$ is as in Assumption \ref{assn:sec}(D) (note this implies that $p>1$), then
\[
\Eb{S_m^p} \leq m^{p-1} \Eb{\sum_{i=1}^m (\Ynui)^p} = m^{p} \Eb{(\tilde{Y}_1)^p}.
\]
Now choose $\gamma < \frac{\alpha-1}{\alpha}$ such that $\gamma p > 1$. Applying the above relation on the branch between $u_i$ and $u_j$ with $m=d^{h, \lambda}(u_i, u_j)$ (and assuming that $i\neq j$), we deduce that, for any $h \geq 1, \lambda \geq 3$,
\begin{align}\label{eqn:marginal snake moment}
\begin{split}
\econdb{(n_{h, \lambda}^{-\left(1-\frac{1}{\alpha}\right)}|\Znu - \Znv|)^p \mathbbm{1}\{E_{h, \lambda, k}\}}{W_{[0, n_{h, \lambda}]}}{} &\leq n_{h, \lambda}^{-p\left(1-\frac{1}{\alpha}\right)} d^{h, \lambda}(u_i,u_j)^p \Eb{\tilde{Y}_{1}^p} \mathbbm{1}\{E_{h, \lambda, k}\} \\
&= \left(n_{h, \lambda}^{-\left(1-\frac{1}{\alpha}\right)} d^{h, \lambda}(u_i,u_j)\right)^p \Eb{\tilde{Y}_{1}^p} \mathbbm{1}\{E_{h, \lambda, k}\} \\
&\leq \left(\kappa_{h,\gamma,\lambda} (\log \lambda)\left(n_{h, \lambda}^{-1}|i-j|\right)^{\gamma}\right)^p \Eb{\tilde{Y}_{1}^p} \mathbbm{1}\{E_{h, \lambda, k}\}.
\end{split}
\end{align}
(Here in the third line we used the definition of $\kappa_{h,\gamma,\lambda}$ from Definition \ref{def:kappan def}). Now note that, using the uniform bound on $\Eb{\tilde{Y}_{1}^p}$ from Assumption \ref{assn:sec}(D) and Lemma \ref{lem:marzouk height holder result}, we know that there exist constants $c_{p,k,\gamma} < \infty$ and $N_{\gamma} < \infty$ such that
\[
\Eb{\tilde{Y}_{1}^p}^{\frac{1}{p}} \Eb{ \kappa_{h,\gamma,\lambda}^{p} \mathbbm{1}\{E_{h, \lambda, k}\}}^{\frac{1}{p}} \leq c_{p,k,\gamma}
\]
for all $h \geq N_{\gamma}, \lambda \geq 3$. 
%
In particular, defining $f^{(h,\lambda)}(s)$ for $s \in [0,1]$ by  $f^{(h,\lambda)}(s) = Z_{sn_{h, \lambda}}$ when $sn_{h, \lambda} \in \N$ and then interpolating, then using this with the tower property in \eqref{eqn:marginal snake moment} we have for all $h \geq N_{\gamma}, \lambda \geq 3$ that:
\begin{align*}
\Eb{(n_{h, \lambda}^{-\left(1-\frac{1}{\alpha}\right)} (\log \lambda)^{-1}|f^{(h,\lambda)}(s) - f^{(h,\lambda)}(t)|)^p \mathbbm{1}\{E_{h, \lambda, k}\}}^{\frac{1}{p}} &\leq \Eb{\tilde{Y}_{1}^p}^{\frac{1}{p}} |s-t|^{\gamma} \Eb{ \kappa_{h,\gamma,\lambda}^{p} \mathbbm{1}\{E_{h, \lambda, k}\}}^{\frac{1}{p}} \\
&\leq c_{p,k,\gamma} |s-t|^{\gamma}.
\end{align*}
Therefore, letting $\eta = \frac{1}{2}\left(\gamma - \frac{1}{p}\right)>0$, Kolmogorov's continuity criterion gives that there exists a constant $C_{p,k,\gamma} < \infty$ so that for all $h \geq N_{\gamma}, \lambda \geq 3$,
\[
\Eb{\left(\sup_{s,t \in [0,1]} \frac{n_{h, \lambda}^{-\left(1-\frac{1}{\alpha}\right)} (\log \lambda)^{-1}|f^{(h,\lambda)}(s) - f^{(h,\lambda)}(t)|}{|s-t|^{\eta}} \right)^p \mathbbm{1}\{E_{h, \lambda, k}\}} < C_{p,k,\gamma}.
\]
In particular, by Markov's inequality, for any $q>0$ it holds with probability at least $1-C_{p,k,\gamma}\lambda^{-pq}$ on the event $E_{h, \lambda, k}$ that
\[
\sup_{s,t \in [0,1]} \frac{n_{h, \lambda}^{-\left(1-\frac{1}{\alpha}\right)} (\log \lambda)^{-1}|f^{(h,\lambda)}(s) - f^{(h,\lambda)}(t)|}{|s-t|^{\eta}} < \lambda^q
\] 
for all $h \geq N_{\gamma}, \lambda \geq 3$. In particular, on this event 
\[
\sup_{i \leq n_{h, \lambda}} \Znu \leq \lambda^q n_{h, \lambda}^{1-\frac{1}{\alpha}} (\log \lambda).
\]
Since the choice of $\gamma$ depended only on $p$, we can tie up to deduce that there exists $C_{k, p} < \infty$ and $N_p < \infty$ such that for all $h \geq N_{p}$, all $q>0$ and all $\lambda \geq 3$,
\begin{align*}
\prb{\sup_{i \leq n_{h, \lambda}} \Znu \geq \lambda^q (\log \lambda) n_{h, \lambda}^{1-\frac{1}{\alpha}} \text{ and } E_{h, \lambda, k}} \leq C_{k, p}\lambda^{-qp},
\end{align*}
which implies that for any $\delta > 0$, we can choose $\lambda_{q,\delta} < \infty$ such that for all $\lambda > \lambda_{q,\delta}$,
\begin{align*}
\prb{\sup_{i \leq n_{h, \lambda}} \Znu \geq \lambda^q n_{h, \lambda}^{1-\frac{1}{\alpha}} \text{ and } E_{h, \lambda, k}} \leq C_{k, p}\lambda^{-(q-\delta)p}.
\end{align*}
In particular, given $q>0$ and $p' \in \left(\frac{\alpha}{\alpha -1}, p \right)$ we can choose $\delta_{p, p', q}>0$ small enough that $(q-\delta_{p, p', q})p > qp'$ so that the right hand side above is also upper bounded by $C_{k, p}\lambda^{-(q-\delta)p}$. To conclude, for fixed $p' \in \left(\frac{\alpha}{\alpha -1}, \frac{\alpha + \epsilon}{\alpha -1} \right)$ we can first run the argument with $p = \frac{1}{2}\left(p' + \frac{\alpha + \epsilon}{\alpha -1} \right)$ so that $\delta_{p, p', q}$ depends only on $p$ and $q$, and we obtain that, for all $p' \in (\frac{\alpha}{\alpha -1}, {\frac{\alpha +\epsilon}{\alpha -1}})$ there exists $C_{k, p'} < \infty$ and $N_{p'} < \infty$ and $\lambda_{p',q}<\infty$ such that for all $h \geq N_{p'}$ and all $\lambda \geq \lambda_{p',q}$,
\begin{align*}
\prb{\sup_{i \leq n_{h, \lambda}} \Znu \geq \lambda^q n_{h, \lambda}^{1-\frac{1}{\alpha}} \text{ and } E_{h, \lambda, k}} \leq C_{k, p'}\lambda^{-qp'}.
\end{align*}
Finally, we can replace the requirement that $\lambda \geq \lambda_{p',q}$ the requirement that $\lambda \geq 1$ by increasing the constant $C_{k,p'}$ so that it becomes $C_{k,p',q}$, and similarly with the requirement that $\lambda \geq 3$ for the bound on $\prb{E_{h, \lambda, k}^c}$.
\end{proof}

\subsection{Decorated volume bounds}
Recall that $\TER$ is the decorated version of an unconditioned Galton--Watson tree $T$ with offspring distribution satisfying \eqref{eqn:offspring tails}. We now give the asymptotics for the tail decay for the volume of $\TER$. Recall the fragmental volume exponent $\fv = \frac{\alpha}{v}$.

\begin{prop}\label{prop:GW vol tail}
If $\fv \neq 1$, there exist constants $c', C' \in (0, \infty)$, such that for all $x \geq 1$:
\[
c'x^{\frac{-(\fv \wedge 1)}{\alpha}} \leq \prb{\V (\TER) \geq x} \leq C'x^{\frac{-(\fv \wedge 1)}{\alpha}}
\]
If $\fv = 1$, we instead have
\[
c'x^{\frac{-1}{\alpha}} \leq \prb{\V (\TER) \geq x \log x} \leq C'x^{\frac{-1}{\alpha}}.
\]
\end{prop}
\begin{proof}
\textbf{Upper bounds.} We start with the first case $\fv \neq 1$. We first consider what happens on $\TER_n$, i.e. when the underlying tree is conditioned to have $n$ vertices for some $n \geq 1$. Letting ${U_n}$ be uniform on $\{0, \ldots, n-1\}$ and $u_0, \ldots, u_{n-1}$ denote the vertices of $\TER_n$ listed in the lexicographical ordering, we have that, for any $\lambda \geq 1$,
\begin{align}\label{eqn:vol sum split}
\prb{\V (\TER_n) \geq n^{\frac{1}{\fv \wedge 1}} \lambda} \leq \prb{\sum_{i=0}^{\lfloor \frac{n}{2}\rfloor - 1} \V (G(u_{{U_n}+i})) \geq \frac{1}{2}n^{\frac{1}{\fv \wedge 1}} \lambda} + \prb{\sum_{i=\lfloor \frac{n}{2} \rfloor}^{n-1} \V (G(u_{{U_n}+i})) \geq \frac{1}{2}n^{\frac{1}{\fv \wedge 1}} \lambda}.
\end{align}
Similarly to \eqref{eqn:Verv trans abs cont rel} and by the tower property, the first probability is equal to
\begin{align*}
\Eb{ \prcondb{\sum_{i=0}^{\lfloor \frac{n}{2} \rfloor - 1} \V (G(u_{{U_n}+i})) \geq \frac{1}{2}n^{\frac{1}{\fv \wedge 1}} \lambda}{\deg (u_{{U_n}+i}) = {W}(i+1) - {W}(i)}{} \frac{p_{n-\lfloor \frac{n}{2} \rfloor}(-W(\lfloor \frac{n}{2} \rfloor)-1)}{p_n(-1)}},
\end{align*}
where $W$ is an unconditioned random walk path with jump distribution $\xi -1$. As explained below \eqref{eqn:Verv trans abs cont rel}, the ratio $\frac{p_{n-\lfloor \frac{n}{2} \rfloor}(-W(\lfloor \frac{n}{2} \rfloor)-1)}{p_n(-1)}$ is upper bounded by a constant, uniformly in $n$. Therefore it suffices to consider the variables $\left(\V (G(u_{{U_n}+i}))\right)_{i=0}^{\lfloor \frac{n}{2} \rfloor - 1}$ under the corresponding joint law with $\left({W}(i+1) - {W}(i)\right)_{i=0}^{\lfloor \frac{n}{2} \rfloor - 1}$.

By \eqref{eqn:offspring tails}, there exists $c< \infty$ such that $\prb{{W}(i+1) - {W}(i) > x} \sim cx^{-\alpha}$ as $x \to \infty$. Moreover, by Assumption \ref{assn:sec}(V), there exist constants $c_1, c_2 \in (0, \infty)$ and $\epsilon>0$ such that for all $\lambda \geq 1$,
\begin{align*}
\prcondb{\V (G(u_{{U_n}+i})) \geq x^v \lambda}{\deg u_{{U_n}+i}= {W}(i+1) - {W}(i) = x}{} &\leq c_1\lambda^{-\frac{\alpha + \epsilon}{v}} \\
\prcondb{\V (G(u_{{U_n}+i})) \geq x^v}{\deg u_{{U_n}+i}= {W}(i+1) - {W}(i) = x}{} &\geq c_2 > 0.
\end{align*}
For each $i \leq \lfloor \frac{n}{2} \rfloor - 1$, we can therefore apply Lemma \ref{lem:stable composition tail}$(i)$ to the pair $(X,Y) = (\deg (u_{{U_n}+i}), \V (G(u_{{U_n}+i})))$ with $\beta = \alpha, z = v, m = \frac{\alpha + \epsilon}{v}$ to deduce that there exist constants $c', C' \in (0, \infty)$ such that
\begin{align}\label{eqn:vol graph bound}
c' x^{-\fv} \leq \prb{\V (G(u_{{U_n}+i})) \geq x} \leq C' x^{-\fv}.
\end{align}
Now take any $\epsilon > 0$. Since the variables $\left(\V (G(u_{{U_n}+i}))\right)_{i=0}^{\lfloor \frac{n}{2} \rfloor - 1}$ are now independent, it therefore follows from Lemma \ref{lem:stable sum tail prob app}$(i)$ that there exists $C_{\epsilon} < \infty$ such that for all $n, \lambda \geq 1$,
\[
\prb{\sum_{i=0}^{\lfloor \frac{n}{2} \rfloor} \V (G(u_{{U_n}+i})) \geq \frac{1}{2}n^{\frac{1}{\fv \wedge 1}} \lambda} \leq C_{\epsilon}\lambda^{-(\fv - \epsilon)}
\]
Moreover, by symmetry, the same bound holds for the second sum in \eqref{eqn:vol sum split}. To conclude the proof of the upper bound, we now apply Lemma \ref{lem:stable composition tail}$(i)$ to the pair $(|T|, \V(\TER))$ with $\beta = \frac{1}{\alpha}, z = \frac{1}{\fv \wedge 1}$ and $m= \fv - \epsilon$, which immediately gives the first upper bound of the proposition (provided we chose $\epsilon > 0$ sufficiently small that $\fv - \epsilon > \frac{\fv}{\alpha}$).

For the second statement: if instead $\fv=1$, equation \eqref{eqn:vol graph bound} still holds and we similarly obtain from Lemma \ref{lem:stable sum tail prob app}$(ii)$ that for all $n, \lambda \geq 1$,
\[
\prb{\sum_{i=0}^{\lfloor \frac{n}{2} \rfloor} \V (G(u_{{U_n}+i})) \geq \frac{1}{2}n\log n \lambda} \leq \lambda^{-(\fv - \epsilon)}
\]
and again the result follows from Lemma \ref{lem:stable composition tail}$(ii)$ applied to the pair $(|T|, \V(\TER))$, exactly as above.

\textbf{Lower bounds.} Again we start with the case $\fv \neq 1$. Similarly to above, note that for any $n \geq 1$,
\begin{align*}
\prb{\V (\TER_n) \geq n^{\frac{1}{\fv \wedge 1}}} &\geq \prb{\sum_{i=0}^{\lfloor \frac{n}{2} \rfloor} \V (G(u_{{U_n}+i})) \geq n^{\frac{1}{\fv \wedge 1}}} \\
&=\Eb{ \prcondb{\sum_{i=0}^{\lfloor \frac{n}{2} \rfloor} \V (G(u_{{U_n}+i})) \geq n^{\frac{1}{\fv \wedge 1}}}{\deg (u_{{U_n}+i}) = {W}(i+1) - {W}(i)}{} \frac{p_{n-\lfloor \frac{n}{2} \rfloor}(-W(\lfloor \frac{n}{2} \rfloor)-1)}{p_n(-1)}}.
\end{align*}
First, note that by \eqref{eqn:vol graph bound}, Lemma \ref{lem:stable sum tail prob lower tail app} and independence we can find $c'>0$ such that for all $n \geq 1$ (here $\mathbf{E}$ denotes expectation over $W$),
\[
\Eb{ \prcondb{\sum_{i=0}^{\lfloor \frac{n}{2} \rfloor} \V (G(u_{{U_n}+i})) \geq n^{\frac{1}{\fv \wedge 1}}}{\deg (u_{{U_n}+i}) = {W}(i+1) - {W}(i)}{}} \geq c'
\]
Moreover, as in the proof of Lemma \ref{lem:max degree tail}, it follows from \cite[Section VIII, Proposition 4]{BertoinLevy} and \eqref{eqn:Gned LLT} that we can firstly choose $A< \infty$ so that 
\[
\prb{\left|W\left(\left\lfloor \frac{n}{2}\right\rfloor \right) + 1\right| > An^{\frac{1}{\alpha}}} \leq \frac{c'}{2},
\]
and secondly choose $a>0$ so that on the event $\{|W\left(\lfloor \frac{n}{2}\rfloor \right) + 1| \leq An^{\frac{1}{\alpha}}\}$, $\frac{p_{\lceil \frac{n}{2}\rceil}(-W(\lfloor \frac{n}{2}\rfloor)-1)}{p_{n}(-1)} > a$.

Together, these imply that for all $n \geq 1$,
\begin{align}\label{eqn:tree n dec vol}
\begin{split}
&\prb{\V (\TER_n) \geq n^{\frac{1}{\fv \wedge 1}}} \\
&\geq \Eb{ \prcondb{\sum_{i=0}^{\lfloor \frac{n}{2} \rfloor} \V (G(u_{{U_n}+i})) \geq n^{\frac{1}{\fv \wedge 1}}}{\deg (u_{{U_n}+i}) = {W}(i+1) - {W}(i)}{} \frac{p_{n-\lfloor \frac{n}{2} \rfloor}(-W(\lfloor \frac{n}{2} \rfloor)-1)}{p_n(-1)}} \geq \frac{ac'}{2},
\end{split}
\end{align}
which verifies the lower bound condition of Lemma \ref{lem:stable composition tail}$(i)$ applied to $(|T|, \V(\TER))$ with $\beta = \frac{1}{\alpha}, z = \frac{1}{\fv \wedge 1}$, from which the lower bound follows. For the second statement with $\fv=1$, we again similarly consider the quantity $\prb{\sum_{i=0}^{\lfloor \frac{n}{2} \rfloor} \V (G(u_{{U_n}+i})) \geq n \log n}$, and repeat the same proof, replacing $n^{\frac{1}{\fv \wedge 1}}$ with $n \log n$ throughout.
\end{proof}

\section{Volume bounds for $\TERa$}\label{sctn:volume bounds decorated}
In this section we prove volume bounds for $\TERa$ under Assumption \ref{assn:sec} and \eqref{eqn:offspring tails}. The proof for the upper bounds follow a simplified version of the strategy for stable looptrees in \cite[Section 5.2]{ArchBMCompactLooptrees}, but due to the variability of the inserted graphs we do not optimise the argument as precisely as we did in \cite{ArchBMCompactLooptrees}.

Recall from \eqref{eqn:volume exponent def} that 
\[
\dERa = \frac{\alpha (\sad \wedge 1)}{(\alpha - 1)(\fv \wedge 1)}.
\]
In this section we will show that this is the correct volume growth exponent for $\TERa$, and in particular prove Theorem \ref{thm:dec vol growth main}.

In Theorem \ref{thm:dec vol growth main} we have stated the volume bounds with respect to the decorated metric as defined by \eqref{eqn:decorated metric def}, since these are of independent interest aside from determining the random walk exponents. However, the construction in \eqref{eqn:decorated metric def} holds when $d_{G(v_i)}$ is an arbitrary metric on $G(v_i)$, so in particular, on replacing $\sad$ with $\saR$ we obtain the exponent for volume growth with respect to the effective resistance metric.
\vspace{.6cm}


\subsection{Volume upper bounds}\label{sctn:dec vol UBs}
The main result is as follows.

\begin{prop}\label{prop:vol UB prob UB}
Take $b_1$ as in Theorem \ref{thm:dec vol growth main}. For any $\epsilon > 0$ there exists a constant $C_{\epsilon} < \infty$ such that for all $r, \lambda \geq 1$,
\begin{align*}
\prb{\V(\BT(\rho, r)) \geq r^{\dERa} (\log r)^{b_1} \lambda} \leq C_{\epsilon}\lambda^{-\frac{(\fv \wedge 1)(\alpha - 1 - \epsilon)}{\alpha^2}}.
\end{align*}
\end{prop}



\begin{rmk}\label{rmk:nonoptimal exponent}
\begin{enumerate}
\item The tail decay here is not optimal. In all of the propositions in the rest of the section, the precise tail decay is not important for our purpose, other than that it is of polynomial form.
\item In the case $\fv > 1$ and $\sad < 1$, if the graphs inserted are deterministic or do not have too much randomness, it should be possible to extend the arguments of this section to get stretched exponential decay for the upper volume bounds, similarly to how we did for stable looptrees in \cite[Section 5.2]{ArchBMCompactLooptrees}. This would involve defining an iterative procedure on large subtrees that fall close to the root, analogously to that on \cite[p. 23]{ArchBMCompactLooptrees}.
\end{enumerate}
\end{rmk}

This therefore gives half of Theorem \ref{thm:dec vol growth main}$(i)$. The quenched upper bound in Theorem \ref{thm:dec vol growth main}$(ii)$ follows from this proposition by applying Borel-Cantelli along the subsequence $r_n = 2^n$ with $\lambda_n = (\log r_n)^{\beta}$, where $\beta = \frac{\alpha^2 + \epsilon}{(\fv \wedge 1)(\alpha - 1 - \epsilon)}$, and using monotonicity of $\V ( \BT (\rER, r))$. (We then take $\beta_1 = \beta + b_1$ for the upper bound in Theorem \ref{thm:dec vol growth main}$(ii)$). 

\subsubsection{Heuristics}\label{sbsctn:vol UB heuristics}
Fix $r \geq 1$. Before starting the proof, we briefly outline the strategy, which has several steps.

\begin{enumerate}
\item Consider the vertices of the underlying tree $\Tai$ along its infinite backbone in sequential order of their distance from the root, and label them in order as $\rho = s_0, s_1, \ldots$.
\item We will make an appropriate choice of an index ${n_{r,\epsilon,\lambda}}$ so that w.h.p. all of $\BT(\rER, r)$ is completely contained within the inserted graphs corresponding to the segment of backbone from $\rho$ to $s_{n_{r,\epsilon,\lambda}}$ and the decorated subtrees attached to these graphs. It turns out that a good choice is $n_{r,\epsilon,\lambda} = r^{\sad \wedge 1} \lambda^{\epsilon}$, for some small $\epsilon>0$ and some $\lambda \geq 1$.
\item We will first bound the quantity
\[
\sum_{j \leq n_{r,\epsilon,\lambda}} \deg (s_j),
\]
which gives an upper bound for $\Nfr$, the number of subtrees attached to the backbone within (decorated) distance $r$ of the root. More specifically, we will show that, w.h.p. as $r, \lambda \rightarrow \infty$ appropriately, $\Nfr \leq r^{\frac{\sad \wedge 1}{\alpha - 1}} \lambda$.
\item On this event, we proceed as follows. We first define vertex sets
\begin{align*}
\Sr &= \bigcup_{j=0}^{\infty} G({s_j}) \cap \BT(\rho, r), &\partial \Sr = \bigcup_{j=0}^{\infty} \partial G({s_j}) \cap \BT(\rho, r),
\end{align*}
where for $u \in \Ti$, $\partial G(u)$ denotes the set of boundary vertices of $G(u)$. On the high probability events in points 2 and 3 above, it is sufficient to instead take the above unions only over the set $\{j \leq n_{r,\epsilon,\lambda}\}$. We then continue by analysing the tree structures of all the decorated subtrees that are grafted to a vertex $v \in \partial \Sr$. Recall that such a vertex $v$ corresponds to an edge of the underlying tree $\Tai$ joining a backbone vertex to one of its offspring. Let this offspring vertex be $w_v$. By Definition \ref{def:Kesten's tree}, $T_{w_v}$ is a critical Galton--Watson tree, again with offspring distribution $\xi$. We let $\TER_v$ denote the subgraph of $\TERa$ corresponding to $\cup_{u \in T_{w_v}}G(u)$, with metric and measure defined as in points 5 and 6 on page \pageref{eqn:decorated metric def}.

On the high probability events above, it is then the case that
\begin{align}\label{eqn:large small subtree decomp}
\V(\BT(\rho, r)) &\leq \sum_{v \in \partial \Sr} \V (\TER_v) + \sum_{j \leq n_{r,\epsilon,\lambda}} \V (G({s_j})).
\end{align}
\end{enumerate}

We formalise this and bound each of these terms separately in the next subsection.

\subsubsection{Main argument}\label{sctn:vol UB main argument}

%
In order to implement the method outlined above, we need to bound the following quantities:
\begin{enumerate}
\item The number of subtrees grafted to the backbone of $\TERa$ within decorated distance $r$ of $\rER$.
\item The sum of the volumes of these decorated subtrees.
\item The volume contribution from the graphs inserted along the backbone.
\end{enumerate}

We consider these one by one in the next subsections. At many points, this will involve adding up sums of random variables with exponents related to those we introduced in Assumption \ref{assn:sec} and Section \ref{sctn:finite tree bounds dec}. In many cases, the relevant exponent may be more than or less than $1$, so we will have to consider two regimes: one in which the sum follows law of large numbers type behaviour, and the other in which the tails are heavier and we see stable-type behaviour. This will eventually give rise to several phase transitions in the value of $\dERa$, which can also be surmised from its expression as
\[
\dERa = \frac{\alpha (\sad \wedge 1)}{(\alpha - 1)(\fv \wedge 1)}.
\]

\subsubsection{Controlling the number of subtrees}
For $r\geq 1$, we let $\Nfr$ denote the number of decorated subtrees that are grafted to the decorated backbone within decorated distance $r$ of $\rER$.

\begin{prop}\label{prop:spinal vol bound}
Suppose that $\sad \neq 1$. Then for any $\epsilon > 0$ there exists $C_{\epsilon} < \infty$ such that for all $r, \lambda \geq 1$,
\[
\prb{\Nfr \geq r^{\frac{\sad \wedge 1}{\alpha -1}} \lambda} \leq 
C_{\epsilon}\lambda^{-(\alpha - 1 - \epsilon)}.
\]
If $\sad = 1$, then we similarly have
\[
\prb{\Nfr \geq \lambda r^{\frac{1}{\alpha -1}} (\log r)^{\frac{-1}{\alpha -1}}} \leq C_{\epsilon}\lambda^{-(\alpha - 1- \epsilon)}.
\]
\end{prop}
\begin{proof}
We prove the case $\sad < 1$. Take $r, \lambda \geq 1$, choose $\epsilon > 0$, and set $n_{r,\epsilon,\lambda} = r^{\sad} \lambda^{\epsilon}$. Consider the sequence of backbone vertices $s_0, s_1, \ldots, s_{n_{r,\epsilon,\lambda}}$. We will show that, with quantifiable high probability, $\Sr$ does not extend to graphs corresponding to backbone vertices beyond $s_{n_{r,\epsilon,\lambda}}$, in which case $\Nfr$ is bounded by $\sum_{i=1}^{n_{r,\epsilon,\lambda}} \deg (s_i)$. By Lemmas \ref{lem:ER spine offspring dist}$(ii)$ and \ref{lem:stable sum tail prob lower tail app}$(i)$, there exists $c'>0$, depending only on $c$ in \eqref{eqn:offspring tails} and the quantities appearing in Assumption \ref{assn:sec}(D), such that for all $r, \lambda \geq 1$,
\begin{align}\label{eqn:bound on nr}
\prb{\sum_{i \leq n_{r,\epsilon,\lambda}} d^U(s_i) \leq r} 
\leq e^{-c'\lambda^{\epsilon}}.
\end{align}
Also, setting $\epsilon' = \frac{\epsilon}{\alpha - 1}$, we have by Lemma \ref{lem:stable sum tail prob app}$(i)$ that there exists $C'<\infty$, depending only on $c$ in \eqref{eqn:offspring tails}, such that for all $r, \lambda \geq 1$,
\begin{align*}
\prb{\sum_{i=0}^{n_{r,\epsilon,\lambda}} \deg (s_i) \geq r^{\frac{\sad}{\alpha - 1}}\lambda} &= \prb{n_{r,\epsilon,\lambda}^{\frac{-1}{\alpha - 1}}\sum_{i=0}^{n_{r,\epsilon,\lambda}} \deg (s_i) \geq \lambda^{1-\epsilon'}} \leq C'\lambda^{-(1-\epsilon')(\alpha - 1)}.
\end{align*}
Combining these in a union bound, we deduce that there exists $C_{\epsilon}<\infty$ such that for all $r, \lambda \geq 1$,
\[
\prb{\Nfr \geq r^{\frac{\sad}{\alpha - 1}} \lambda} \leq e^{-c' \lambda^{\epsilon}} +  C'\lambda^{-(1-\epsilon')(\alpha - 1)} \leq C_{\epsilon}\lambda^{-(\alpha - 1-\epsilon)}.
\]
The proofs of the two other cases $\sad = 1$ and $\sad>1$ are identical, except that we respectively apply Lemmas \ref{lem:stable sum tail prob lower tail app}$(ii)$ and \ref{lem:stable sum tail prob lower tail app}$(iii)$ in place of Lemma \ref{lem:stable sum tail prob lower tail app}$(i)$, and respectively set $n_{r,\epsilon,\lambda} = r (\log r)^{-1} \lambda^{\epsilon}$ and $n_{r,\epsilon,\lambda} = r \lambda^{\epsilon}$.
\end{proof}

\subsubsection{Controlling the volumes of the decorated subtrees}

We now turn to bounding the quantity
\[
\sum_{v \in \partial \Sr} \V (\TER_v).
\]
Recall that $\Nfr= |\partial \Sr|$. We will work on the event $\left\{\Nfr < r^{\frac{\sad \wedge 1}{\alpha -1}} \lambda^{\phi} \right\}$ for some $\phi>0$, and recall from Proposition \ref{prop:GW vol tail} and Definition \ref{def:Kesten's tree} that there exists $c' < \infty$ such that for a vertex $v \in \partial \Sr$, $\prb{\V(\TER_v) \geq x (\log x)^{\mathbbm{1}\{\fv = 1\}}} \leq c'x^{\frac{-(\fv \wedge 1)}{\alpha}}$ for all $x \geq 1$, independently for each such $v$. We deduce that this expression falls into the framework of Lemmas \ref{lem:stable sum tail prob app}$(i)$ and \ref{lem:stable sum tail prob log app}, with $\beta = \frac{\fv \wedge 1}{\alpha}$ and $n = \Nfr$; the next proposition is an immediate consequence.

\begin{prop}\label{prop:small fragments}
Take $b_1$ as defined in Theorem \ref{thm:dec vol growth main}. If $\sad \neq 1$, there exists $c'<\infty$ such that for any $\phi \in \left(0, \frac{\fv \wedge 1}{\alpha}\right)$, and any $r, \lambda \geq 1$,
\begin{align*}
&\prb{\sum_{v \in \partial \Sr} \V(\TER_v) \geq \lambda r^{\dERa} (\log r)^{b_1}, \Nfr < r^{\frac{\sad \wedge 1}{\alpha -1}}\lambda^{\phi}} \leq c'\lambda^{-\left(\frac{\fv \wedge 1}{\alpha} - \phi\right)}.
\end{align*}
If $\sad = 1$, there exists $c'<\infty$ such that for any $\phi \in \left(0, \frac{\fv \wedge 1}{\alpha}\right)$, and any $r, \lambda \geq 1$,
\begin{align*}
&\prb{\sum_{v \in \partial \Sr} \V(\TER_v) \geq \lambda r^{\dERa} (\log r)^{b_1}, \Nfr < \lambda^{\phi} r^{\frac{1}{\alpha -1}} (\log r)^{\frac{-1}{\alpha -1}}} \leq c'\lambda^{-\left(\frac{\fv \wedge 1}{\alpha} - \phi\right)}.
\end{align*}
\end{prop}

\subsubsection{Controlling the spinal volume}

Finally, we bound the spinal volume.
%

\begin{prop}\label{prop:spinal vol bound edge}
For all $\epsilon > 0$, there exists $C_{\epsilon}< \infty$ such that, if $\sad \neq 1$, we have for all $r, \lambda \geq 1$:
\[
\prb{\V (\Sr) \geq r^{\dERa} \lambda} \leq C_{\epsilon}\lambda^{-(\sav - \epsilon)}.
\]
If instead $\sad = 1$, we instead have that
\[
\prb{\V (\Sr) \geq  \lambda r^{\dERa} (\log r)^{\frac{-1}{\sav}}} \leq
C_{\epsilon}\lambda^{-(\sav - \epsilon)}.
\]
\end{prop}
\begin{proof}
Take $r, \lambda \geq 1$. Recall the spinal volume exponent defined on page \pageref{box:decorated exponents paper} by $\sav = \frac{\alpha - 1}{v}$. The result follows identically to the proof of Proposition \ref{prop:spinal vol bound}, with $\alpha - 1$ replaced by $\sav$ (both are less than $1$) to bound the probabilities $\prb{\V (\Sr) \geq r^{\frac{\sad \wedge 1}{\sav}} \lambda}$ (if $\sad \neq 1$) and $\prb{\V (\Sr) \geq  \lambda r^{\frac{1}{\sav}} (\log r)^{\frac{-1}{\sav}}}$ (if $\sad = 1$). More precisely, taking $n_{r,\epsilon,\lambda}= r^{\sad \wedge 1} (\log r)^{-\mathbbm{1}\{\sad=1\}} \lambda^{\epsilon}$ as in Proposition \ref{prop:spinal vol bound}, we first note that by Assumption \ref{assn:sec}(V) and Lemma \ref{lem:stable composition tail}$(i)$ applied to the pair $(\deg s_i, \V (G (s_i)))$ with $\beta = \alpha - 1, z = v$ and $m=\frac{\alpha}{v}$, there exists $C'<\infty$ such that for each $i \leq n_{r,\epsilon,\lambda}$ and each $x \geq 1$,
\[
\prb{\V (G (s_i)) \geq x} \leq C'x^{-\sav}.
\]
Moreover, it follows from Definition \ref{def:Kesten's tree} that the random variables $(\V (G (s_i)))_{i \leq n_{r,\epsilon,\lambda}}$ are i.i.d.. Therefore, setting $\epsilon' = \frac{\epsilon}{\sav}$, we have by Lemma \ref{lem:stable sum tail prob app}$(i)$ that for all $r, \lambda \geq 1$,
\begin{align*}
\prb{\sum_{i=0}^{n_{r,\epsilon,\lambda}} \V (G (s_i)) \geq r^{\frac{\sad \wedge 1}{\sav}}(\log r)^{-\frac{\mathbbm{1}\{\sad=1\}}{\sav}}\lambda} &\leq \prb{n_{r,\epsilon,\lambda}^{\frac{-1}{\sav}}\sum_{i=0}^{n_{r,\epsilon,\lambda}} \V (G (s_i)) \geq \lambda^{1-\epsilon'}} \leq C'\lambda^{-(1-\epsilon')\sav}.
\end{align*}

 To replace the exponent $\frac{\sad \wedge 1}{\sav}$ with $\dERa$, we note that
\[
\frac{\sad \wedge 1}{\sav} = \frac{v(\sad \wedge 1)}{\alpha - 1} \leq \frac{v(\sad \wedge 1)}{(\alpha - 1)} \vee \frac{\alpha (\sad \wedge 1)}{(\alpha - 1)} = \frac{\alpha (\sad \wedge 1)}{(\alpha - 1)(\fv \wedge 1)} = \dERa
\]
(with equality if $\alpha \leq v$). The result on follows on combining with \eqref{eqn:bound on nr} in a union bound exactly as in the proof of Proposition \ref{prop:spinal vol bound}.
\end{proof}

\begin{rmk}
In particular, the proof above shows that the spinal volume can never dominate the volume contribution from the fragments. 
\end{rmk}

\subsubsection{Proof of Proposition \ref{prop:vol UB prob UB}}

\begin{proof}[Proof of Proposition \ref{prop:vol UB prob UB}]\label{pf:vol UB}
We prove the case $\fv, \sad \neq 1$ (i.e. the case $b_1 =0$) for simplicity. The key to the proof is \eqref{eqn:large small subtree decomp} and in particular the decomposition
\begin{align*}
\V(\BT(\rho, r)) &\leq \sum_{v \in \partial \Sr} \V (\TER_v) + \V (\Sr).
\end{align*}
We therefore deduce that, for any $\phi \in \left(0, \frac{\fv \wedge 1}{\alpha} \right)$, and any $r, \lambda \geq 1$,
\begin{align*}
\prb{\V(\BT(\rho, r)) \geq r^{\dERa} \lambda} \leq \ &\prb{\Nfr \geq r^{\frac{\sad \wedge 1}{\alpha -1}}\lambda^{\phi}} + \prb{\V (\Sr) \geq \frac{1}{2} r^{\dERa} \lambda} \\
&+ \prb{\sum_{v \in \partial \Sr} \V(\TER_v) \geq \frac{1}{2} r^{\dERa} \lambda, \Nfr < r^{\frac{\sad \wedge 1}{\alpha -1}}\lambda^{\phi}}.
\end{align*}
In this case, it therefore follows from Propositions \ref{prop:spinal vol bound}, \ref{prop:small fragments} and \ref{prop:spinal vol bound edge} together with a union bound that for any $\epsilon > 0$, there exists a constant $C_{\epsilon}<\infty$ such that
\begin{align*}
\prb{\V(\BT(\rho, r)) \geq r^{\dERa} \lambda} \leq C_{\epsilon}\lambda^{-\phi(\alpha - 1- \epsilon)} + C_{\epsilon}\lambda^{-(\sav - \epsilon)} + c'\lambda^{-\left( \frac{\fv \wedge 1}{\alpha} - \phi \right)},
\end{align*}
so we optimise by taking $\phi=\frac{\fv \wedge 1}{\alpha^2}$, which gives
\begin{align*}
\prb{\V(\BT(\rho, r)) \geq r^{\frac{\alpha (\sad \wedge 1)}{(\alpha - 1)(\fv \wedge 1)}} \lambda} \leq C_{\epsilon}\lambda^{\frac{-(\fv \wedge 1)(\alpha - 1 - \epsilon)}{\alpha^2}} + C_{\epsilon}\lambda^{-(\sav - \epsilon)}.
\end{align*}
Provided $\epsilon>0$ is small enough, the tail decay of the first term is always worse, so we deduce the result of Proposition \ref{prop:vol UB prob UB}. 

The cases where $b_1 \neq 0$ follow in exactly the same way by applying the previous propositions and incorporating the extra log terms.
\end{proof}

\subsection{Volume lower bounds}\label{sctn:vol LBs}
The main result is as follows.
\begin{prop}\label{prop:vol LB prob UB 2}
Take $b_1$ as in Theorem \ref{thm:dec vol growth main}. For any $\epsilon>0$ there exists a constant $c_{\epsilon} < \infty$ such that for all $r, \lambda \geq 1$,
\begin{align*}
\prb{\V(\BT(\rho, r) \leq \lambda^{-1} r^{\dERa} (\log r)^{b_1}} \leq c_{\epsilon}\lambda^{-\frac{\fv \wedge 1}{\alpha}((\sad \wedge \frac{\alpha}{\alpha - 1}) \vee 1)(\alpha - 1- \epsilon)}.
\end{align*}
\end{prop}
Again, the tail decay here is not optimal and the precise value of the exponent in Proposition \ref{prop:vol LB prob UB 2} is not significant. Together with Proposition \ref{prop:vol UB prob UB}, this therefore completes the proof of Theorem \ref{thm:dec vol growth main}$(i)$. The quenched lower bound in Theorem \ref{thm:dec vol growth main}$(ii)$ similarly follows from this proposition by applying Borel-Cantelli along the subsequence $r_n = 2^n$ with $\lambda_n = (\log r_n)^{-\beta}$, where $\beta = -\frac{\alpha}{(\fv \wedge 1)((\sad \wedge \frac{\alpha}{\alpha - 1}) \vee 1)(\alpha - 1- \epsilon)}$, and using monotonicity of $\V ( \BT (\rER, r))$. (We then take $\beta_1 = \beta - b_1$ for the lower bound in Theorem \ref{thm:dec vol growth main}$(ii)$).

As before, we start with a result on the number of subtrees grafted to the decorated backbone.
\begin{prop}\label{prop:spinal vol for LB stable 2}
Suppose $\sad \neq 1$. For any $\epsilon > 0$ there exists $c_{\epsilon} \in (0, \infty)$ such that for all $r, \lambda \geq 1$:
\[
\prb{\Nf2r \leq r^{\frac{\sad \wedge 1}{\alpha - 1}} \lambda^{-1}} \leq c_{\epsilon}\lambda^{-((\sad \wedge \frac{\alpha}{\alpha - 1}) \vee 1)(\alpha - 1- \epsilon)} .
\]
If $\sad = 1$, then
\[
\prb{N^f_{\frac{r}{2}} \leq r^{\frac{1}{\alpha - 1}} (\log r)^{\frac{-1}{\alpha - 1}} \lambda^{-1}} \leq c_{\epsilon}\lambda^{-(\alpha - 1 - \epsilon)}.
\]
\end{prop}
\begin{proof}
Similarly to previous proofs, suppose for now that $\sad \neq 1$, and take $\epsilon>0$ and $\epsilon'>0$ with $\epsilon' \ll \epsilon$. Given $r, \lambda \geq 1$, set $m = \alpha - 1 - \epsilon'$, let $m_{r,\lambda} = r^{\sad \wedge 1} \lambda^{-m}$, and let $\rho=s_0, s_1, \ldots, s_{m_{r,\lambda}}$ denote the first $m_{r,\lambda}$ vertices on the backbone of $\Tai$. Then, by Lemmas \ref{lem:ER spine offspring dist}$(i)$ and \ref{lem:stable sum tail prob app}, we have that there exist $c, c_{\epsilon'} < \infty$ such that for all $r, \lambda \geq 1$,
\begin{align*}
\prb{1 + \sum_{i=0}^{m_{r,\lambda}} \diam (G(s_i)) \geq \frac{r}{2}} \leq \begin{cases} c\lambda^{-m} &\text{ if } \sad < 1 \\
c_{\epsilon'}\lambda^{-m[(\sad \wedge \frac{\alpha}{\alpha - 1}) - \epsilon']} &\text{ if } \sad > 1.
\end{cases}
\end{align*}
(Here the term $+1$ is added because of our rooting convention). Also, note that, by Lemma \ref{lem:stable sum tail prob lower tail app}$(i)$, there exists $c'>0$ such that for all $r, \lambda \geq 1$
\begin{align*}
\prb{\sum_{i=0}^{m_{r,\lambda}} \deg (s_i) \leq r^{\frac{\sad \wedge 1}{\alpha - 1}} \lambda^{-1}} \leq e^{-c'\lambda^{\left(\alpha - 1-m \right)}}.
\end{align*}
Provided we take $\epsilon'$ small enough compared to $\epsilon$, this gives the result by a union bound. If $\sad = 1$, we instead take $m_{r,\lambda} = r (\log r)^{-1} \lambda^{-m}$.
\end{proof}

\begin{proof}[Proof of Proposition \ref{prop:vol LB prob UB 2}]
We first deal with the case $\fv \neq 1, \sad \neq 1$. Take some $\epsilon>0$ and choose $\epsilon'>0$ so that $\epsilon' \ll \epsilon$. Also set $q={(\fv \wedge 1)(1 - \epsilon')}{(\frac{1}{\alpha d} \vee \frac{\alpha - 1}{\alpha})}$. Note that $\dERa (\fv \wedge 1)(\frac{1}{\alpha d} \vee \frac{\alpha - 1}{\alpha})=1$. 

Now take $r, \lambda \geq 1$. Because of our rooting convention, we can assume wlog that $r^{\dERa} \lambda^{-1} \geq 1$. Since each of the subtrees grafted to the backbone are independent of each other, and letting $\left( T_i \right)_{i=1}^{\Nf2r}$ denote the first $\Nf2r$ subtrees (ordered by distance to the root and breaking ties arbitrarily), and $\left( \TER_i \right)_{i=1}^{\Nf2r}$ the corresponding decorated subgraphs of $\TER$, we have from Lemma \ref{lem:prog tail bound}$(i)$, Proposition \ref{prop:decorated height bound without log term prog} and \eqref{eqn:tree n dec vol} that there exists $c'>0$ and $\lambda_{\epsilon'}<\infty$ such that for all $r \geq 1, \lambda \geq \lambda_{\epsilon'}$ each $i \leq \Nf2r$,
\begin{align*}
&\prb{\V (\TER_i) \geq r^{\dERa} \lambda^{-1}, \Heightdec (\TER_i) \leq \frac{1}{2}r } \\
&\geq \prb{\frac{r^{\dERa (\fv \wedge 1)} \lambda^{-[(\fv \wedge 1) (1- \epsilon')]}}{|T_i|} \in [1,2]} \left(1 - \prb{\V (\TER_i) \leq \lambda^{-\epsilon'}|T_i|^{\frac{1}{\fv \wedge 1}}}  - \prb{\Heightdec (\TER_i) > \frac{1}{2} |T_i|^{(\frac{1}{\alpha d} \vee \frac{\alpha - 1}{\alpha})} \lambda^q} \right) \\
&\geq c'r^{-\dERa\left(\frac{\fv \wedge 1}{\alpha}\right)} \lambda^{\frac{(\fv \wedge 1)(1-\epsilon')}{\alpha}}.
\end{align*}
Now take $h = \frac{\fv \wedge 1}{\alpha} - \epsilon'$. Since $\left( T_i \right)_{i=1}^{\Nf2r}$ is an independent sequence, we have for all $r \geq 1, \lambda \geq \lambda_{\epsilon'}$ that
\begin{align*}
\prb{\left\{\nexists i: \V (\TER_i) \geq r^{\dERa} \lambda^{-1}, \Heightdec (\TER_i) \leq \frac{1}{2}r\right\}, \Nf2r \geq r^{\frac{\sad \wedge 1}{\alpha - 1}} \lambda^{-h}}{} 
&\leq \left(1-c'r^{-\dERa\left(\frac{\fv \wedge 1}{\alpha}\right)} \lambda^{\frac{(\fv \wedge 1)(1-\epsilon')}{\alpha}}\right)^{r^{\frac{\sad \wedge 1}{\alpha - 1}} \lambda^{-h}} \\
&\leq e^{-c'\lambda^{\epsilon'\left(1-\frac{\fv \wedge 1}{\alpha}\right)}}.
\end{align*}
Therefore, by a union bound and Proposition \ref{prop:spinal vol for LB stable 2} it follows that for any $\epsilon>0$ we can choose $\epsilon'>0$ small enough such that there exist $c_{\epsilon'}, C_{\epsilon}, \lambda_{\epsilon'} < \infty$ such that for all $r \geq 1, \lambda \geq \lambda_{\epsilon'}$,
\begin{align*}
&\prb{\V(\BT(\rho, r) \leq \lambda^{-1} r^{\dERa}} \\
&\leq \prb{\Nf2r \leq r^{\frac{\sad \wedge 1}{\alpha - 1}} \lambda^{-h}} + \prb{\left\{\nexists i: \V (\TER_i) \geq r^{\dERa} \lambda^{-1}, \Heightdec (\TER_i) \leq \frac{1}{2}r\right\}, \Nf2r \geq r^{\frac{\sad \wedge 1}{\alpha - 1}} \lambda^{-h}}{}\\
&\leq c_{\epsilon'}\lambda^{-h((\sad \wedge \frac{\alpha}{\alpha - 1}) \vee 1)(\alpha - 1- \epsilon')} + e^{-c'\lambda^{\epsilon'\left(1-\frac{\fv \wedge 1}{\alpha}\right)}} \leq C_{\epsilon}\lambda^{-\frac{\fv \wedge 1}{\alpha}((\sad \wedge \frac{\alpha}{\alpha - 1}) \vee 1)(\alpha - 1- \epsilon)}.
\end{align*}
This extends to the case where $1 \leq \lambda < \lambda_{\epsilon'}$ by modifying $C_{\epsilon}$ if necessary, giving the stated result.

The other cases where $\sad =1$ or $\fv = 1$ can be treated by exactly the same proof on incorporating the factors of $\log r$.
\end{proof}


\section{Resistance on $\TERa$}\label{sctn:dec res bounds}

In order to apply results of \cite{KumMisumiHKStronglyRecurrent} about random walk exponents, we also need to understand resistance on $\TERa$.

In Assumption \ref{assn:sec} we have assumed that the two-point function and diameters of the inserted graphs grow according to the same exponents. It would still be possible to get some kind of result if this was not the case, but we would need more information on the local geometry of the inserted graphs. When the two exponents are equal, it means that we are able to cut the decorated backbone at an appropriate cut point such that the cut point is roughly distance $r$ from the root, and all vertices contained in the decorated graphs corresponding to ancestors of that cutpoint are also roughly within distance $r$ of the root. The same holds for the resistance distance, so this gives a concrete way to separate $B^{\text{dec}}_{res}(\rER, r)$ from $B^{\text{dec}}_{res}(\rER, 2r)^c$, for example. However, if the diameters of the inserted graphs grow differently to the two-point function, then we cannot separate balls just by exploiting the underlying tree structure.

Take $i \geq 0$ and denote by $s_i$ the $i^{th}$ vertex from the root on the infinite backbone of $\Tai$. Recall from Assumption \ref{assn:sec}(R) and Lemma \ref{lem:stable composition tail}$(i)$ applied to the pair $(\deg s_i, \diam_{res} (G(s_i)))$ that the exponent $\saR$ is defined so that there exist constants $c_1, c_2 \in (0, \infty)$ such that
\begin{align*}
\prb{\diam_{res} (G(s_i)) \geq r} &\leq c_1 r^{- \saR} \\
\prb{R^U (G(s_i)) \geq r} &\geq c_2 r^{- \saR}
\end{align*}
for all $i \geq 0$ and all $r \geq 1$. 

In what follows, we will use the subscript ``res'' to indicate that distances are defined with respect to the resistance metric. For example, $B^{\text{dec}}_{res}(\rER, r)$ refers to a ball of radius $r$ with respect to the effective resistance metric on $\TERa$, and $\BT(\rER, r)$ still refers to a ball with respect to the decorated metric $\dg$. For a decorated tree $\TER, \Heightdec_{\text{res}} (\TER) = \sup_{x \in \TER} \Ref (\rER, x)$.
%
Recall from Proposition \ref{prop:decorated height bound without log term} that
\[
\tad=\frac{1}{3(2\alpha - (\sad \wedge 1))}.
\]
We analogously define 
\begin{align}\label{eqn:tar def}
\taR&=\frac{\alpha}{3(2\alpha - (\saR \wedge 1))}.
\end{align}

To understand a simple random walk on $\TERa$, we will need to estimate the effective resistance from the root to the boundary of a ball. We will do this by defining an appropriate cutvertex on the infinite backbone which separates a large subset of $B^{\text{dec}}_{res}(\rER, r)$ from infinity. To prove that the candidate cutvertex indeed does this, it will be necessary to bound $\sup_i \Heightdec_{\text{res}} ((\TER)^{(i)})$ for a collection of subtrees attached to the backbone.

Before stating the main resistance result, we give a lemma that will enable us to do this. If the tail decay in Proposition \ref{prop:decorated height bound without log term} was stronger, we would be able to apply Lemma \ref{lem:stable composition tail} to get a good bound on $\prb{\Heightdec_{\text{res}} (\TER) \geq x}$ for arbitrary $x \geq 1$ and then apply a straightforward union bound. However, the tail decay of Proposition \ref{prop:decorated height bound without log term} is not strong enough to do this, so we have to strengthen Proposition \ref{prop:decorated height bound without log term} using another trick.

\begin{lem}\label{lem:height bound GW union}
Set $\car = \frac{1}{\alpha R} \vee (1-\frac{1}{\alpha})$. For $n \geq 1$, let $((\TER)^{(i)})_{i=1}^{n^{\frac{1}{\alpha}}}$ denote a set of independent decorated unconditioned Galton--Watson trees.
\begin{enumerate}[(i)]
\item Assume $\saR \neq 1$. Then, for any $\epsilon>0$ there exists $c_{\epsilon}<\infty$ such that for any $n, \lambda \geq 1$,
\begin{align*}
\prb{\bigcup_{i=1}^{n^{\frac{1}{\alpha}}} \left\{\Heightdec_{\text{res}} ((\TER)^{(i)}) \geq n^{\frac{1}{\alpha R} \vee 1-\frac{1}{\alpha}} \lambda \right\}} &\leq c_{\epsilon}\lambda^{-\frac{(\alpha+1)\taR}{\alpha+1+2\alpha\car\taR}(1-\epsilon)}.
\end{align*}
\item 
If $\saR=1$, we instead have that
\begin{align*}
\prb{\bigcup_{i=1}^{n^{\frac{1}{\alpha}}(\log n)^{-\frac{1}{\alpha-1}}} \left\{\Heightdec_{\text{res}} ((\TER)^{(i)}) \geq n^{1-\frac{1}{\alpha}} \lambda \right\}} &\leq c_{\epsilon}\lambda^{-\frac{(\alpha+1)\taR}{\alpha+1+2\alpha\car\taR}(1-\epsilon)}.
\end{align*}
\end{enumerate}
\end{lem}
\begin{proof}
\begin{enumerate}
\item
We start with the case $\saR \neq 1$. Take $n, \lambda \geq 1$, and recall from Proposition \ref{prop:decorated height bound without log term} and \eqref{eqn:tar def} that for a finite decorated tree $\TER_n$, where the underlying Galton--Watson tree $T_{n,2n]}$ is conditioned to have between $n$ and $2n$ vertices, we have that
\begin{equation}\label{eqn:height bound repeat}
\prb{\Heightdec_{\text{res}} (\TER_{[n,2n]}) \geq n^{\frac{1}{\alpha R} \vee \left(1-\frac{1}{\alpha}\right)} \lambda} \leq c\lambda^{-t^R_{\alpha}}
\end{equation}
(with an extra $\log n$ factor if $\saR=1$). Now observe that we can also get a lower bound for the probability of this event as follows. Take a constant $A>1$ (we will fix it later in the proof), let $\tilde{T}_n$ be a Galton--Watson tree conditioned on the number of vertices in generation $n^{1-\frac{1}{\alpha}}$ being in $[A^{-1}n^{\frac{1}{\alpha}}, An^{\frac{1}{\alpha}}]$, and let $\tTER_n$ denote its decorated version. We consider the the event that one of the subtrees in $\tilde{T}_n$ emanating from level $n^{1-\frac{1}{\alpha}}$ has decorated resistance height at least $n^{\left(\frac{1}{\alpha R} \vee 1-\frac{1}{\alpha}\right)} \lambda$. Due to the Galton--Watson structure, all of the subtrees emanating from level $n^{1-\frac{1}{\alpha}}$ are independent Galton--Watson trees, so letting $(T^{(i)})_{i=1}^{A^{-1}n^{\frac{1}{\alpha}}}$ denote (a subset of) these Galton--Watson trees on the event described, and $((\TER)^{(i)})_{i=1}^{A^{-1}n^{\frac{1}{\alpha}}}$ their decorated versions, we have that
\begin{align}\label{eqn:res LB union}
\begin{split}
\prb{\Heightdec_{\text{res}} (\tTER_n) \geq n^{\left(\frac{1}{\alpha R} \vee 1-\frac{1}{\alpha}\right)} \lambda} &\geq \prb{\bigcup_{i=1}^{A^{-1}n^{\frac{1}{\alpha}}} \left\{ \Heightdec_{\text{res}} ((\TER)^{(i)}) \geq n^{\left(\frac{1}{\alpha R} \vee 1-\frac{1}{\alpha}\right)} \lambda\right\}} \\
\end{split}
\end{align}
Moreover, letting $Z_{n^{1-\frac{1}{\alpha}}}$ denote the size of generation $n$ and $B =  \frac{2\alpha\taR}{\alpha+1+2\alpha\car\taR}$ we have that
\begin{align}\label{eqn:res trick decomp}
\begin{split}
&\prb{\Heightdec_{\text{res}} (\tTER_n) \geq n^{\left(\frac{1}{\alpha R} \vee 1-\frac{1}{\alpha}\right)} \lambda} \\
&\leq \prcondb{\Heightdec_{\text{res}} (\tTER_n) \geq n^{\left(\frac{1}{\alpha R} \vee 1-\frac{1}{\alpha}\right)} \lambda}{|\tilde{T}_n| \in [\lambda^{-\epsilon}n, \lambda^{B}n]}{} + \prb{|\tilde{T}_n| < \lambda^{-\epsilon}n} + \prb{|\tilde{T}_n| > \lambda^{B}n}.
\end{split}
\end{align}
We claim that the right hand side decays polynomially in $\lambda$. Indeed, for the first term in the last line of \eqref{eqn:res trick decomp} we have by \cite[Corollary 7]{KortSubexp}, Lemma \ref{lem:prog tail bound}$(i)$ and \eqref{eqn:height bound repeat} that we can choose $A>1$ large enough (not depending on $n$) so that there exists a constant $c_A<\infty$ such that for all $n, \lambda \geq 1$,
\begin{align*}
&\prcondb{\Heightdec_{\text{res}} (\tTER_n) \geq n^{\frac{1}{\alpha R} \vee \left(1-\frac{1}{\alpha}\right)} \lambda}{|\tilde{T}_n| \in [\lambda^{-\epsilon}n, \lambda^{B}n]}{} \\
&\leq \frac{\prcondb{\Heightdec_{\text{res}} (T) \geq n^{\frac{1}{\alpha R} \vee \left(1-\frac{1}{\alpha}\right)} \lambda}{|T| \in [\lambda^{-\epsilon}n, \lambda^{B}n]}{}}{\prcondb{Z_{n^{1-\frac{1}{\alpha}}} \in [A^{-1}n^{\frac{1}{\alpha}}, An^{\frac{1}{\alpha}}]}{|T| \in [\lambda^{-\epsilon}n, \lambda^{B}n]}{}} \\
&\leq \frac{\prcondb{\Heightdec_{\text{res}} (T) \geq n^{\frac{1}{\alpha R} \vee \left(1-\frac{1}{\alpha}\right)} \lambda}{|T| \in [\lambda^{-\epsilon}n, \lambda^{B}n]}{}}{\prcondb{Z_{n^{1-\frac{1}{\alpha}}} \in [A^{-1}n^{\frac{1}{\alpha}}, An^{\frac{1}{\alpha}}]}{|T| \in [n, 2n]}{} \prcondb{|T| \in [n, 2n]}{|T| \in [\lambda^{-\epsilon}n, \lambda^{B}n]}{}} \leq c_A\lambda^{\frac{\epsilon}{\alpha} - \taR (1-\car B)},
\end{align*}
recalling that $\car = \frac{1}{\alpha R} \vee (1-\frac{1}{\alpha})$.

For the second term in \eqref{eqn:res trick decomp}, note that by the same logic as above we can increase $A$ and $c_A$ a bit if necessary (independently of $n$ and $\lambda$) so that we also have for all $n, \lambda \geq 1$ that
\begin{align*}
\prb{|\tilde{T}_n| < \lambda^{-\epsilon}n} &\leq \frac{\prb{|T| \leq n\lambda^{-\epsilon} \text{ and } Z_{n^{1-\frac{1}{\alpha}}} \in [A^{-1}n^{\frac{1}{\alpha}}, An^{\frac{1}{\alpha}}]}}{\prcondb{Z_{n^{1-\frac{1}{\alpha}}} \in [A^{-1}n^{\frac{1}{\alpha}}, An^{\frac{1}{\alpha}}]}{|T| \in [n, 2n]}{} \prb{|T| \in [n, 2n]}} \\
&\leq c_A n^{\frac{1}{\alpha}} \prb{|T| \leq n\lambda^{-\epsilon} \text{ and } Z_{n^{1-\frac{1}{\alpha}}} \in [A^{-1}n^{\frac{1}{\alpha}}, An^{\frac{1}{\alpha}}]}.
\end{align*}
For this latter probability we can use Lemma \ref{lem:prog tail bound}$(i)$ and \cite[Corollary 7]{KortSubexp} to deduce that there exist $c', c_A'<\infty$ such that, for all $n, \lambda \geq 1$,
\begin{align*}
\prb{|T| \leq n\lambda^{-\epsilon} \text{ and } Z_{n^{1-\frac{1}{\alpha}}} \in [A^{-1}n^{\frac{1}{\alpha}}, An^{\frac{1}{\alpha}}]}
&\leq c'\int_0^{n\lambda^{-\epsilon}} (1 \wedge x^{-(1+\frac{1}{\alpha})}) \exp \left\{ -\left( \frac{A^{-\alpha}n}{x} \right) \right\} dx \\
&\leq c''n^{-\frac{1}{\alpha}}\int_{\lambda^{\epsilon}}^{\infty} y^{\frac{1}{\alpha}-1} \exp \left\{ -A^{-\alpha}y \right\} dy \\
&\leq c_A' n^{-\frac{1}{\alpha}} \exp \left\{ -A^{-\alpha}\lambda^{\epsilon} \right\}.
\end{align*}
Plugging back into the previous inequality we deduce that there exists $c_A'<\infty$ such that the second term in \eqref{eqn:res trick decomp} is upper bounded by $c_A' \exp \left\{ - A^{-\alpha}\lambda^{\epsilon}\right\}$.

%
%

For the third term in \eqref{eqn:res trick decomp} we again use \cite[Theorem 5 and Corollary 7]{KortSubexp} and Lemma \ref{lem:prog tail bound}$(i)$ to deduce that for any $\epsilon>0$ there exists $c_{A, \epsilon}<\infty$ such that
\begin{align*}
&\prb{|\tilde{T}_n| > \lambda^{B}n} \\
&= \prcondb{Z_{n^{1-\frac{1}{\alpha}}} \in [A^{-1}n^{\frac{1}{\alpha}}, An^{\frac{1}{\alpha}}]}{|T| > \lambda^{B}n}{}\frac{\prb{|T| > \lambda^{B}n}}{\prb{Z_{n^{1-\frac{1}{\alpha}}} \in [A^{-1}n^{\frac{1}{\alpha}}, An^{\frac{1}{\alpha}}]}} \\
&\leq \prcondb{Z_{n^{1-\frac{1}{\alpha}}} \leq An^{\frac{1}{\alpha}}}{|T| > \lambda^{B}n}{}\frac{\prb{|T| > \lambda^{B}n}}{\prcondb{Z_{n^{1-\frac{1}{\alpha}}} \in [A^{-1}n^{\frac{1}{\alpha}}, An^{\frac{1}{\alpha}}]}{|T| \in [n, 2n]}{}\prb{|T| \in [n, 2n]}} \\
&\leq c_{A, \epsilon} \lambda^{-\frac{[B (\alpha - 1)-\epsilon]}{2\alpha}} \lambda^{-\frac{B}{\alpha}} = c_{A, \epsilon}\lambda^{-\frac{[B(\alpha+1)-\epsilon]}{2\alpha}}.
\end{align*}
(Again, this holds provided we chose $A$ large enough, but not in a way that depends on $n$ or $\lambda$). Inserting these estimates into \eqref{eqn:res trick decomp} and combining with \eqref{eqn:res LB union} and a union bound, we deduce (updating the value of $c_A$ as appropriate) that
\begin{align}\label{eqn:tree res abs cont}
\begin{split}
\prb{\bigcup_{i=1}^{n^{\frac{1}{\alpha}}} \left\{ \Heightdec_{\text{res}} ((\TER)^{(i)}) \geq n^{\frac{1}{\alpha R} \vee \left(1-\frac{1}{\alpha}\right)} \lambda \right\}} &\leq A\prb{\Heightdec_{\text{res}} (\tTER_n) \geq n^{\frac{1}{\alpha R} \vee \left(1-\frac{1}{\alpha}\right)} \lambda} \\
&\leq c_A\lambda^{\frac{\epsilon}{\alpha} - \taR (1-\car B)} + c_{A,\epsilon}\lambda^{-\frac{[B(\alpha+1)-\epsilon]}{2\alpha}}\\
&\leq c_{\epsilon'}\lambda^{-\frac{(\alpha+1)\taR}{\alpha+1+2\alpha\car\taR}(1-\epsilon')}.
\end{split}
\end{align}
In particular since $B = \frac{2\alpha\taR}{\alpha+1+2\alpha\car\taR}$, the final inequality follows by fixing $A$ and provided we chose $\epsilon'>0$ sufficiently small compared to $\epsilon$.
\item In the case $\saR=1$, we instead consider a tree $\tilde{T}_n$ conditioned so that the number of vertices in generation $\frac{n^{1-\frac{1}{\alpha}}}{\log n}$ is in $\left[ \frac{A^{-1}n^{\frac{1}{\alpha}}}{(\log n)^{\frac{1}{\alpha-1}}}, \frac{An^{\frac{1}{\alpha}}}{(\log n)^{\frac{1}{\alpha-1}}}\right]$. In this case the total volume of $\tilde{T}_n$ is typically of order $\frac{n}{(\log n)^{\frac{\alpha}{\alpha - 1}}}$, so by Proposition \ref{prop:decorated height bound without log term} its decorated resistance height is typically of order $n^{1-\frac{1}{\alpha}}$. We can then repeat the proof above with these modified quantities.
\end{enumerate}

\end{proof}

We then have the following result on effective resistance in $\TERa$.
 
\begin{prop}\label{prop:res to boundary general}
There exist constants $a>0, A< \infty$ such that for all $r, \lambda \geq 1$:
\begin{align*}
\prb{\Ref (\rER, B^{\text{dec}}_{res}(\rER, r)^c) \leq r \lambda^{-1}} \leq A\lambda^{-a}.
\end{align*}
\end{prop}
\begin{proof}
We write the proof for $\saR \neq 1$. Note that, due to our rooting convention, we can assume wlog that $r \lambda^{-1} \geq 1$. First fix some $\epsilon>0$ and define two constants $k,l$ by  $k = (\saR \wedge 1)(1-\epsilon)$ and $l = \frac{(\saR \wedge 1)(1-\epsilon)}{2(\alpha - 1)}$. Recall that $s_0, s_1, \ldots$ denote the backbone vertices of $\Tai$, ordered by their distance from the root. Similarly to previous proofs, given $r, \lambda \geq 1$ we define a number $N_{r,\lambda}$ that corresponds to the index of the vertex where we ``cut off'' the infinite backbone. This time, we take $N_{r,\lambda} = r^{\saR \wedge 1}\lambda^{-k}$ (note that $N_{r, \lambda} \geq 1$ under the assumption that $r \lambda^{-1} \geq 1$) and first observe that, by Lemma \ref{lem:stable sum tail prob app} and Assumption \ref{assn:sec}(R), there exists $c_{\epsilon} < \infty$ such that for all $r, \lambda \geq 1$,
\begin{align}\label{eqn:sum diam UB res}
\prb{1+\sum_{i \leq N_{r,\lambda}} \diam_{res} (s_i)  \geq \frac{1}{2}r} &\leq c_{\epsilon}\lambda^{-k[(\saR \vee 1)-\epsilon]}.
\end{align}
(Here the extra $+1$ term is there on the LHS because of our rooting convention). Additionally, by Proposition \ref{lem:ER spine offspring dist}$(ii)$ (since the metric $\dg$ is generic the result there also holds for the resistance distance) and Lemma \ref{lem:stable sum tail prob lower tail app}, there exist $c, C \in (0,\infty)$ such that for all $r, \lambda \geq 1$ satisfying $r \lambda^{-1} \geq 1$,
\begin{align}\label{eqn:res ball distance LB}
\prb{1+\sum_{i \leq N_{r,\lambda}} R^U(s_i) \leq r\lambda^{-k(\saR \wedge 1)^{-1}-\epsilon}} \leq Ce^{-c\lambda^{\epsilon (\saR \wedge 1)}}.
\end{align}
Moreover, by Definition \ref{def:Kesten's tree} and Lemma \ref{lem:stable sum tail prob app} that there exists $c' < \infty$ such that for all $r, \lambda \geq 1$,
\begin{align}\label{eqn:sum deg UB res}
\prb{\sum_{i \leq N_{r,\lambda}} \deg (s_i) \geq r^{\frac{\saR \wedge 1}{\alpha - 1}} \lambda^{l-\frac{k}{\alpha - 1}}} \leq c'\lambda^{-l(\alpha - 1)}.
\end{align}
The three probabilistic bounds above show that, with (quantified) high probability, for all $r, \lambda \geq 1$ satisfying $r \lambda^{-1} \geq 1$: 
\begin{itemize}
\item all of the vertices contained in $\bigcup_{i \leq N_{r,\lambda}} G(s_i)$ are within resistance distance $\frac{1}{2}r$ of the root;
\item the vertex joining $G({s_{N_{r,\lambda}-1}})$ to $G({s_{N_{r,\lambda}}})$ is at resistance distance at least $r\lambda^{-k(\saR \wedge 1)^{-1}-\epsilon} = N_{r,\lambda}^{(\saR \wedge 1)^{-1}} \lambda^{-\epsilon}$ from the root;
\item the number of subtrees joined to the decorated backbone at a vertex contained in $\bigcup_{i \leq N_{r,\lambda}} G(s_i)$ is at most $r^{\frac{\saR \wedge 1}{\alpha - 1}} \lambda^{l-\frac{k}{\alpha - 1}} = N_{r,\lambda}^{\frac{1}{\alpha - 1}}\lambda^{l}$.
\end{itemize}
We will additionally show that, with high probability on the three above events, all of the subtrees joined to the decorated backbone at a vertex contained in $\bigcup_{i \leq N_{r,\lambda}} G(s_i)$ have decorated resistance height at most $\frac{1}{2}r$. 

For the rest of the proof we set $K_{\epsilon} = \frac{(\alpha+1)\taR}{\alpha+1+2\alpha\car\taR}(1-\epsilon)$, where $\car = \frac{1}{\alpha R} \vee (1-\frac{1}{\alpha})$ as in Lemma \ref{lem:height bound GW union}, and $\taR$ is given by \eqref{eqn:tar def}.

We apply Lemma \ref{lem:height bound GW union}$(i)$ by choosing $n_{r, \lambda}$ so that $n_{r, \lambda}^{\frac{1}{\alpha}}=N_{r,\lambda}^{\frac{1}{\alpha - 1}}\lambda^{l} = r^{\frac{\saR \wedge 1}{\alpha-1}}\lambda^{l-\frac{k}{\alpha -1}}$ (again note that with our choices of $k$ and $l$, it follows that $n_{r, \lambda}\geq 1$ under the assumption that $r\lambda^{-1}\geq 1$). Let $C_r$ be the event that there exists a subtree joined to the decorated backbone of $\TERa$ at a vertex in $\bigcup_{i \leq N_{r,\lambda}} G(s_i)$ with decorated height at least $\frac{1}{2}r$. Conditionally on the number of such subtrees being upper bounded by $r^{\frac{\saR \wedge 1}{\alpha-1}}\lambda^{l-\frac{k}{\alpha -1}}$, we have by independence of these subtrees and Lemma \ref{lem:height bound GW union} that there exists $c_{\epsilon}<\infty$ such that, for all $r, \lambda \geq 1$ satisfying  $r\lambda^{-1}\geq 1$,
\begin{align*}
\prb{C_r} &\leq \prb{\bigcup_{i=1}^{r^{\frac{\saR \wedge 1}{\alpha-1}}\lambda^{l-\frac{k}{\alpha -1}}} \left\{\Heightdec_{\text{res}} ((\TER)^{(i)}) \geq \frac{1}{2}r \right\}} \\
&= \prb{\bigcup_{i=1}^{n_{r, \lambda}^{\frac{1}{\alpha}}} \left\{\Heightdec_{\text{res}} ((\TER)^{(i)}) \geq \frac{1}{2} n_{r, \lambda}^{\left[\frac{1}{\alpha R} \vee (1-\frac{1}{\alpha})\right]} \lambda^{\alpha (\frac{k}{\alpha -1}-l)\left[\frac{1}{\alpha R} \vee (1-\frac{1}{\alpha})\right]} \right\}} \\
&\leq c_{\epsilon}\lambda^{-\alpha (\frac{k}{\alpha -1}-l)\left[\frac{1}{\alpha R} \vee (1-\frac{1}{\alpha})\right]K_{\epsilon}}.
\end{align*}

To conclude, note that on the event $C_r^c$ intersected with the events in the bullet point above, the vertex $v_r$ defined to be the vertex at which $G(s_{N_{r,\lambda}})$ and $G(s_{N_{r,\lambda} + 1})$ intersect is a single point that separates the root from $B^{\text{dec}}_{res}(\rER, r)^c$, and itself is at resistance distance at least $N_{r,\lambda}^{(\saR \wedge 1)^{-1}} \lambda^{-\epsilon}$ from the root, so that
\begin{align*}
\Ref (\rER, B^{\text{dec}}_{res}(\rER, r)^c) \geq \Ref (\rER, v_r) \geq N_{r,\lambda}^{(\saR \wedge 1)^{-1}}\lambda^{-\epsilon} &= r^{(\saR \wedge 1)(\saR \wedge 1)^{-1}} \lambda^{-k(\saR \wedge 1)^{-1}} \lambda^{-\epsilon} \\
&= r\lambda^{-k(\saR \wedge 1)^{-1}-\epsilon} = r\lambda^{-1}.
\end{align*}
Therefore, combining all the probabilistic bounds obtained with a union bound, we see that
\begin{align*}
\prb{\Ref (\rER, B^{\text{dec}}_{res}(\rER, r)^c) \leq r \lambda^{-1}} 
&\leq  c_{\epsilon}\lambda^{-k[(\saR \vee 1)-\epsilon]} + c'\lambda^{-l(\alpha - 1)} + c_{\epsilon}\lambda^{-\alpha \left(\frac{k}{\alpha -1}-l\right)\left[\frac{1}{\alpha R} \vee (1-\frac{1}{\alpha})\right]K_{\epsilon}}.
\end{align*}
In particular, provided we originally fixed $\epsilon>0$ small enough this gives the result (recall that $k = (\saR \wedge 1)(1-\epsilon)$ and $l = \frac{(\saR \wedge 1)(1-\epsilon)}{2(\alpha - 1)}$).

If $\saR=1$, we instead take $N_{r,\lambda} = r (\log (r\lambda^{-k}))^{-1}\lambda^{-k}$, and repeat the same proof. \eqref{eqn:sum diam UB res} is unchanged, and \eqref{eqn:res ball distance LB} and \eqref{eqn:sum deg UB res} can be written as
\begin{align*}
\prb{1+\sum_{i \leq N_{r,\lambda}} R^U(s_i) \leq r\lambda^{-k(\saR \wedge 1)^{-1}-\epsilon}} \leq Ce^{-c\lambda^{\frac{\epsilon}{2} (\saR \wedge 1)}},
\prb{\sum_{i \leq N_{r,\lambda}} \deg (s_i) \geq \left(\frac{r}{\log r}\right)^{\frac{1}{\alpha - 1}} \lambda^{l-\frac{k}{\alpha - 1}}} \leq c'\lambda^{-l(\alpha - 1)}.
\end{align*}
We can then choose $1 = k > (\alpha-1)l$ and then choose $n$ so that $n^{\frac{1}{\alpha}}=r^{\frac{1}{\alpha-1}}\lambda^{l-\frac{k}{\alpha -1}}$, which implies that $n^{\frac{1}{\alpha}} (\log n)^{-\frac{1}{\alpha-1}} \geq \left(\frac{r}{\log r}\right)^{\frac{1}{\alpha-1}}\lambda^{l-\frac{k}{\alpha -1}}$. Lemma \ref{lem:height bound GW union}$(ii)$ then gives
\begin{align*}
\prb{C_r} &\leq \prb{\bigcup_{i=1}^{\left(\frac{r}{\log r}\right)^{\frac{1}{\alpha-1}}\lambda^{l-\frac{k}{\alpha -1}}} \left\{\Heightdec_{\text{res}} ((\TER)^{(i)}) \geq \frac{1}{2}r \right\}} \leq c_{\epsilon}\lambda^{- (k-l(\alpha-1))K_{\epsilon}},
\end{align*}
and the proof continues as in $(i)$.
\end{proof}

\begin{rmk}\label{rmk: res comments}
Similarly to the volume bounds, in some cases where we have better control over the inserted graphs, we may be able to improve this to exponential decay by introducing an interative process that breaks up tall subtrees, similarly to how we did in \cite[Proposition 6.4]{ArchInfiniteLooptrees} in the stable looptree case.
\end{rmk}

We will also need the following result in order to determine the random walk displacement exponent in terms of the intrinsic metric. To save space in the proposition below we write $B^R(r) = B^{\text{dec}}_{res} (\rER, r)$ and $B^d(r) = B^{\text{dec}} (\rER, r)$.

\begin{prop}\label{prop:metric balls compare}
There exists a deterministic $a<\infty$ such that there $\bPb$-almost surely exists $r_0 < \infty$ such that for all $r \geq r_0$, 
\begin{align*}
B^d({r^{(\saR \wedge 1)(\sad \wedge 1)^{-1}} (\log r)^{-a}}) &\subset B^R (r) \subset B^d ({r^{(\sad \wedge 1)^{-1}(\saR \wedge 1)} (\log r)^{a}}) \\
B^R ({r^{(\sad \wedge 1)(\saR \wedge 1)^{-1}} (\log r)^{-a}}) &\subset B^d (r) \subset B^R ({r^{(\saR \wedge 1)^{-1}(\sad \wedge 1)} (\log r)^{a}}).
\end{align*}
Moreover, there exist $b>0$ and $c<\infty$ such that for all $r, \lambda \geq 1$:
\begin{itemize}
\item If $\mathbbm{1}\{\sad = 1\} = \mathbbm{1}\{ \saR = 1\}$, then $$\prb{B^R({r^{(\sad \wedge 1)(\saR \wedge 1)^{-1}} \lambda^{-1}}) \subset B^d(r) \subset B^R ({r^{(\saR \wedge 1)^{-1}(\sad \wedge 1)} \lambda})} \geq 1-c\lambda^{-b}.$$
\item If $\sad=1 \neq \saR$, then $\prb{B^R({r^{(\sad \wedge 1)(\saR \wedge 1)^{-1}} \lambda^{-1}}) \subset B^d({r \log r}) \subset B^R({r^{(\saR \wedge 1)^{-1}(\sad \wedge 1)} \lambda}})  \geq 1-c\lambda^{-b}$.
\item If $\sad\neq 1 = \saR$, then $\prb{B^R({ \lambda^{-1} r^{(\sad \wedge 1)(\saR \wedge 1)^{-1}} \log r}) \subset B^d({r}) \subset B^R({\lambda r^{(\saR \wedge 1)^{-1}(\sad \wedge 1)} \log r})}  \geq 1-c\lambda^{-b}$.
\end{itemize}
\end{prop}
\begin{proof}
%
Assume for now that $\sad, \saR \neq 1$. By replacing $R^U$ with $d^U$ in \eqref{eqn:res ball distance LB}, the same proof as in Proposition \ref{prop:res to boundary general} shows that there exist $b>0$ and $c<\infty$ such that, for all $r, \lambda \geq 1$, it holds with probability at least $1-c\lambda^{-b}$ that there exists a single vertex at intrinsic distance at least $1 \vee r^{(\saR \wedge 1)(\sad \wedge 1)^{-1}} \lambda^{-1}$ from the root, separating the root from $(B^R(r))^c$. Therefore, all vertices in $(B^R(r))^c$ lie at least distance $1 \vee r^{(\saR \wedge 1)(\sad \wedge 1)^{-1}} \lambda^{-1}$ from the root, so that $B^d({r^{(\saR \wedge 1)(\sad \wedge 1)^{-1}} \lambda^{-1}}) \subset B^R(r)$. This gives one side of the probabilistic statement; the other side follows by symmetry.

For the almost sure results, note that by setting $r_n = 2^n$, $\lambda_n = (\frac{1}{2}\log r_n)^{b + 1}$, applying Borel Cantelli and using monotonicity, we deduce that $\bPb$-almost surely, there exists $r_0 < \infty$ such that 
\[
B^d({r^{(\saR \wedge 1)(\sad \wedge 1)^{-1}} (\log r)^{-(b + 1)}}) \subset B^R(r)
\]
for all $r \geq r_0$. By symmetry, we can also use the same argument to go in the other direction, and also deduce that there exists $b'>0$ such that for any \[
B^R({r^{(\sad \wedge 1)(\saR \wedge 1)^{-1}} (\log r)^{-(b' + 1)}}) \subset B^d(r)
\]
eventually $\bPb$-almost surely. 

Now note that, if $\tilde{r}=r^x(\log r)^{-y}$, then $\tilde{r}^{x^{-1}} (\log \tilde{r})^{x^{-1}y+1} \geq r$ for all sufficiently large $r$, so $B^R({\tilde{r}}) \subset B^d(r)$ implies that $B^R({\tilde{r}}) \subset B^d({\tilde{r}^{x^{-1}} (\log \tilde{r})^{x^{-1}y+1}})$ for all sufficiently large $r$. The second inclusion above therefore gives the result as stated. The same proof works for the other values of $\sad, \saR$.
\end{proof}

\section{Random walk exponents}\label{sctn:RW exponents}

The purpose of this section is to use the volume and resistance results of the previous sections to determine the exponents for a simple random walk on $\TERa$. To do this, we will apply results of \cite{KumMisumiHKStronglyRecurrent}. In order to fit in with the framework of \cite{KumMisumiHKStronglyRecurrent}, we will specifically consider the measure $\mu$ on $\TERa$ defined by $\mu (x) = \deg x$ for all $x \in \TERa$,  rather than a general decorated measure $\V$. (This is because $\mu$ defined as such is the ``natural speed measure'' for a simple random walk).

To directly apply the results of \cite{KumMisumiHKStronglyRecurrent} to get exponents for the decorated metric $d$, we would need to define deterministic functions $v(\cdot)$ and $r(\cdot)$ that govern the volume and resistance growth of the space, and for a given $\lambda > 1$ define the set
\begin{align*}
J(\lambda) = &\left\{R \in [1, \infty]: \lambda^{-1} v(R) \leq \mu (B^d(R)) \leq \lambda v(R), \Ref (\rER, (B^d(R))^c) \geq \lambda^{-1} r(R) \right\} \\
&\hspace{3.2cm} \cap \left\{R \in [1, \infty]: \forall y \in B^d(R), \Ref (\rER, y) \leq \lambda r(d(\rER, y)) \right\},
\end{align*}
and show that $\prb{R \in J(\lambda)} \rightarrow 1$ as $\lambda \rightarrow \infty$, uniformly in $R>1$ (cf \cite[Definition 1.1, Assumption 1.2(1)]{KumMisumiHKStronglyRecurrent}.

If we ignore the special cases with logarithmic corrections for a moment, by Propositions \ref{prop:vol UB prob UB}, \ref{prop:vol LB prob UB 2} and \ref{prop:res to boundary general} the appropriate volume function to take would be $v(R) = R^{\frac{\alpha (\sad \wedge 1)}{(\alpha - 1)(\fv \wedge 1)}}$, and the appropriate resistance function would be $r(R) = R^{(\sad \wedge 1)(\saR \wedge 1)^{-1}}$. However, we encounter some technical difficulties with the final condition in the definition of $J(\lambda)$, in that it requires $\Ref (\rER, y) \leq \lambda r(d(\rER, y))$ \textbf{for all} $y \in B_d(R)$.

For a general graph, it is usually only possible to achieve this kind of control uniformly when there is some deterministic relation between the resistance metric and the intrinsic metric, for example as is the case for random trees and looptrees. In our decorated tree setting, this is probably still achievable in the case when we decorate the tree with deterministic graphs, but in the case when the inserted graphs are random we anticipate that there will be genuine multiplicative fluctuations in the relationship between the resistance metric and the intrinsic metric (for example these could be on the order of $\log R$ on the ball of radius $R$), so it is not always possible to bound $\prb{\Ref (\rER, y) \leq \lambda d(\rER, y) \forall y \in B^d(R)}$ uniformly in $R$.

Although this obstruction can probably be dealt with by tweaking the proofs of \cite{KumMisumiHKStronglyRecurrent}, it is easier to circumvent this problem by instead using the results of \cite{KumMisumiHKStronglyRecurrent} as a black box to first estimate displacement with respect to the resistance metric, and then use Proposition \ref{prop:metric balls compare} to account for the fluctuations and state the results in terms of the intrinsic metric. We therefore define the set
\begin{equation}\label{eqn:Jlambda def}
J^R(\lambda) = \left\{r \in [1, \infty]: \lambda^{-1} v^R(r) \leq \mu (B^R(r)) \leq \lambda v^R(r), \Ref (\rhoERA, (B^R(r))^c) \geq \lambda^{-1} r \right\},
\end{equation}
where $v^R(r) = r^{\frac{\alpha (\saR \wedge 1)}{(\alpha - 1)(\fv \wedge 1)}}$, with appropriate logarithmic corrections if $\saR=1$ or $\fv=1$ as in Theorem \ref{thm:dec vol growth main}. It then follows from Propositions \ref{prop:vol UB prob UB}, \ref{prop:vol LB prob UB 2} and \ref{prop:res to boundary general} that there exist $c < \infty, \gamma > 0$ such that for all $\lambda \geq 1$,
\begin{equation}\label{eqn:prob in Jlambda}
\sup_{r \geq 1} \prb{r \notin J^R(\lambda)} \leq c\lambda^{-\gamma},
\end{equation}
which allows us to directly apply \cite[Proposition 1.3]{KumMisumiHKStronglyRecurrent} to give exponents with respect to the resistance metric.

In what follows, we let $(X_n)_{n \geq 0}$ denote a simple random walk on $\TERa$, started at $\rER$, we let $\tau^d_r$ (respectively $\tau^R_r$) denote the exit time of a simple random walk from $B^d(r)$ (respectively $B^R(r)$), and let $p_{2n}(\rhoERA, \rhoERA)$ denote its transition density at the root. Also let $I^R$ denote the inverse of the function $r \mapsto rv^R(r)$, so that $I^R(n) = n^{\frac{(\alpha -1)(\fv \wedge 1)}{(\alpha -1)(\fv \wedge 1) + \alpha(\saR \wedge 1)}}$ when $\saR, \fv \neq 1$. The next result is a direct consequence of \cite[Proposition 1.3]{KumMisumiHKStronglyRecurrent}.

\begin{prop}[Probabilistic results w.r.t. resistance metric]\label{prop:RW exponents wrt res}
Assume that Assumption \ref{assn:sec} and \eqref{eqn:offspring tails} hold. Then, as $\theta \rightarrow \infty$, we have that
\begin{align*}
\liminf_{r \to \infty} \prb{ \theta^{-1} \leq \frac{\estart{\tau^R_r}{\rhoERA}{}}{rv^R(r)} \leq \theta} &\rightarrow 1, \\
\liminf_{n \to \infty} \prb{ \theta^{-1} \leq v^R ( I^R(n)) p_{2n}(\rhoERA, \rhoERA) \leq \theta} &\rightarrow 1, \\
\liminf_{n \to \infty} \bPb \times \prstart{\frac{\Ref (\rhoERA, X_n)}{I^R(n)} \leq \theta}{\rhoERA}{} &\rightarrow 1, \\
\liminf_{n \to \infty} \bPb \times \prstart{ \theta^{-1} \leq \frac{1 +\Ref (\rhoERA, X_n)}{I^R(n)}}{\rhoERA}{} &\rightarrow 1.
\end{align*}
\end{prop}

While the definition of the transition density does not depend on the intrinsic metric, we can combine Proposition \ref{prop:RW exponents wrt res} with Proposition \ref{prop:metric balls compare} to obtain displacement results in terms of the intrinsic metric $d$. 

Before doing so, we recall from \eqref{eqn:spec dis exponent def} that the (claimed) spectral dimension and displacement exponent for $\TERa$ are respectively given by the following exponents:
\begin{align*}
\begin{split}
\dspec &= \frac{2\alpha (\saR \wedge 1)}{\alpha (\saR \wedge 1) + (\alpha - 1)(\fv \wedge 1)}, \\
\ddis &= \frac{(\saR \wedge 1)(\alpha -1)(\fv \wedge 1)}{(\alpha -1)(\fv \wedge 1)(\sad \wedge 1) + \alpha(\saR \wedge 1)(\sad \wedge 1)}.
\end{split}
\end{align*}
To ease the notation in what follows, we further define:
\begin{enumerate}
\item The \textbf{walk dimension}, $\ea = \frac{\alpha (\sad \wedge 1)}{(\alpha - 1)(\fv \wedge 1)} +\frac{\sad \wedge 1}{\saR \wedge 1}$.
\item The \textbf{transition density exponent}, $\ka = \frac{\alpha (\saR \wedge 1)}{(\alpha - 1)(\fv \wedge 1)} \left( \frac{\alpha (\saR \wedge 1)}{(\alpha - 1)(\fv \wedge 1)} + 1 \right)^{-1}$. 
\end{enumerate}

Note that $\Da = \ea^{-1}$, and $\dspec = 2\ka$.

The following proposition then follows from \cite[Proposition 1.3]{KumMisumiHKStronglyRecurrent} and Proposition \ref{prop:metric balls compare}.

\begin{prop}[Probabilistic results w.r.t. intrinsic metric]\label{thm:RW exponents res prob}
Assume that Assumption \ref{assn:sec} and \eqref{eqn:offspring tails} hold. Then, there exist (explicit) deterministic $\ba, \bat, \batt > 0$ such that as $\theta \rightarrow \infty$,
\begin{align*}
\liminf_{r \to \infty} \bPb \times \prstart{ \theta^{-1} \leq \frac{\tau^d_r}{r^{\ea} (\log r)^{\ba}} \leq \theta}{\rER}{} &\rightarrow 1, \\
\liminf_{r \to \infty} \prb{ \theta^{-1} \leq \frac{\estart{\tau^d_r}{\rhoERA}{}}{r^{\ea} (\log r)^{\ba}} \leq \theta} &\rightarrow 1, \\
\liminf_{n \to \infty} \prb{ \theta^{-1} \leq n^{\ka} (\log n)^{\bat} p_{2n}(\rhoERA, \rhoERA) \leq \theta} &\rightarrow 1, \\
\liminf_{n \to \infty} \bPb \times \prstart{\frac{\sup_{k \leq n} \dER (\rhoERA, X_k)}{n^{\Da} (\log n)^{\batt}} \leq \theta}{\rhoERA}{} &\rightarrow 1, \\
\liminf_{n \to \infty} \bPb \times \prstart{ \theta^{-1} \leq \frac{1 +\dER (\rhoERA, X_n)}{n^{\Da} (\log n)^{\batt}}}{\rhoERA}{} &\rightarrow 1. 
\end{align*}
If $\saR, \sad, \fv \neq 1$, then $\ba = \bat = \batt = 0$. Otherwise they can be computed by carrying through the logarithmic terms from Theorem \ref{thm:dec vol growth main} and Proposition \ref{prop:metric balls compare}.
\end{prop}
\begin{proof}
The second and fifth points are both a direct consequence of Proposition \ref{prop:RW exponents wrt res} and Proposition \ref{prop:metric balls compare}, and the third point can be obtained from the analogous statement in Proposition \ref{prop:RW exponents wrt res} just by substituting the forms of the functions $v^R$ and $I^R$. The upper bound in the first point follows from the fifth point.

The lower bound in the first point can be proved using exactly the same arguments used to prove the analogous result for trees in \cite[Proof of Theorem 1.1]{CroyKumRWGWTreeInfiniteVar} applied instead to the exit time $\tau^R_r$, and can then be transferred to $\tau^d_r$ using Proposition \ref{prop:metric balls compare}. In particular, it follows from \cite[Proposition 3.5(a)]{KumMisumiHKStronglyRecurrent} that there exist constants $\kappa, \xi>0$ (not depending on $\lambda$ or $\theta$ or $r$) such that for every $\lambda \geq 1$ there exist constants $c_1, \ldots, c_4$, depending on $\lambda$, such that
\begin{align}\label{eqn:poly tail exit time LB}
\bPb \times \prstart{\tau^R_r < \theta^{-\kappa}c_1 rv^R(r)}{\rER}{} \leq c_2 \theta^{-\xi} + \prb{r \text{ or } \theta^{-1}r \text{ or } c_3\theta^{-1}r \notin J^R(\lambda)} \qquad \forall \theta > c_4,
\end{align}
where $J^R(\lambda)$ is as in \eqref{eqn:Jlambda def}. Moreover, $c_4$ can be chosen to be polynomial in $\lambda$. For any $\delta > 0$, we can first choose $\lambda > 0$ large enough that the probability in \eqref{eqn:prob in Jlambda} is at most $\delta$, and therefore the second probability on the RHS above is at most $3\delta$ provided $r, \theta^{-1}r$ and $c_3\theta^{-1}r$ all exceed $1$. Then, we choose $\theta$ large enough that $\theta> c_4$ and the first probability on the RHS is at most $\delta$. Then, for all $r \geq 1 \vee \theta \vee \theta c^{-1}_3$, the RHS above is at most $4 \delta$. Since $\delta > 0$ was arbitrary, this establishes the analogous claim for $\tau^R_r$, and can then be transferred to $\tau^d_r$ as stated using Proposition \ref{prop:metric balls compare}.

The fourth point also follows from the lower bound in the first point, completing the proof.
\end{proof}


With appropriate control, we can also get quenched and annealed results for these exponents. We give the quenched result first: again this follows directly from \cite[Theorem 1.5]{KumMisumiHKStronglyRecurrent}, Theorem \ref{thm:dec vol growth main}$(i)$ and Proposition \ref{prop:metric balls compare}.

\begin{theorem}[Quenched random walk results]\label{thm:RW exponents quenched}
Under Assumption \ref{assn:sec} and \eqref{eqn:offspring tails}, $\bPb$-almost surely,
\begin{enumerate}[a)]
\item There exist constants $\beta_1, \beta_2, \beta_3, \beta_4 \in (0, \infty)$ such that
\begin{enumerate}[(i)]
\item There exists $R< \infty$ such that $r^{\ea} (\log r)^{-\beta_2} \leq \estart{\tau^d_r}{\rhoERA}{} \leq r^{\ea} (\log r)^{\beta_2}$ for all $r \geq R$.
\item There exists $N< \infty$ such that $n^{-\ka} (\log n)^{-\beta_1} \leq p_{2n}(\rhoERA, \rhoERA) \leq n^{-\ka} (\log n)^{\beta_1}$ for all $n \geq N$.
\item $\bP$-almost surely, there exist $N, R<\infty$ such that, conditionally on $X_0 = \rhoERA$,
\begin{align*}
r^{\ea} (\log r)^{-\beta_3} &\leq \tau^d_r \leq r^{\ea} (\log r)^{\beta_3} \hspace{5mm} \forall r \geq R \\
n^{\Da} (\log n)^{-\beta_4} &\leq \sup_{k \leq n} \dER (\rhoERA, X_k) \leq n^{\Da} (\log n)^{\beta_4} \hspace{5mm} \forall n \geq N
\end{align*}
\end{enumerate}
\item $d_S(\TERa) := -2 \lim_{n \rightarrow \infty} \frac{\log p_{2n}(\rhoERA, \rhoERA)}{\log n} = 2\ka$, and the random walk is recurrent.
\item $\lim_{n \rightarrow \infty} \frac{ \log \left( \estart{\tau^d_r}{\rhoERA}{}\right)}{\log r} = \ea$.
\item Let $W_n = \{X_0, X_1, \ldots, \X_n\}$, and let $S_n = \sum_{x \in W_n} \deg x$. Then $\bP$-almost surely, $\lim_{n \rightarrow \infty} \frac{\log S_n}{\log n} = \ka$.
\end{enumerate}
\end{theorem}

Theorems \ref{thm:recurrence} and \ref{thm:dec main quenched RW} therefore follow directly. The annealed results also follow from similar arguments to those of \cite[Proposition 1.4]{KumMisumiHKStronglyRecurrent}, but require some adaptation since \cite[Assumption 1.2(2)]{KumMisumiHKStronglyRecurrent} is not in general satisfied.

\begin{theorem}[Annealed random walk results]\label{thm:RW exponents annealed} 
Take $\ba, \bat, \batt$ as in Proposition \ref{thm:RW exponents res prob}. Under Assumption \ref{assn:sec} and \eqref{eqn:offspring tails}, we have that:
\begin{enumerate}[a)]
\item There exist constants $c_1 >0, c_2 < \infty, \gamma_1 > 0$ such that, for all $r \geq 1$, $c_1 r^{\ea} (\log r)^{\ba} \leq \Eb{\estart{\tau^d_r}{\rhoERA}{}}$, and $\Eb{\left(\estart{\tau^d_r}{\rhoERA}{}\right)^{\gamma_1}} \leq c_2 r^{\gamma \ea} (\log r)^{\gamma_1 \ba}$.
\item There exist constants $c_3 > 0, c_4 < \infty, \gamma_2 > 0$ such that $c_3 n^{-\ka} (\log n)^{\bat} \leq \Eb{p_{2n}(\rhoERA, \rhoERA)}$ for all $n \geq 1$, and $\Eb{\left(p_{2n}(\rhoERA, \rhoERA)\right)^{\gamma_2}} \leq c_3 n^{-\gamma_2\ka} (\log n)^{\gamma_2\bat}$.
\item There exist constants $c_5 > 0, c_6 < \infty, \gamma_3 > 0$ such that $c_5n^{\Da} (\log n)^{\batt} \leq \Eb{ \estart{\dER (\rhoERA, X_n)}{\rhoERA}{}}$ for all $n \geq 1$, and $\Eb{ \left(\estart{\dER (\rhoERA, X_n)}{\rhoERA}{}\right)^{\gamma_3}} \leq c_6n^{\gamma_3 \Da} (\log n)^{\gamma_3 \batt}$.
\end{enumerate}
\end{theorem}
\begin{proof}
We give the proof of the upper bounds since these are not immediate from \cite[Proposition 1.4]{KumMisumiHKStronglyRecurrent}. The strategy is similar to the proof of \cite[Theorem 1.2]{CroyKumRWGWTreeInfiniteVar}. The lower bounds all follow directly from Proposition \ref{thm:RW exponents res prob}, since we can choose $c_1, c_3, c_5 > 0$ and $\delta>0$ so that the required lower bounds hold with probability at least $\delta$.

For the upper bound in part a), we recall from \cite[Equation (3.7)]{KumMisumiHKStronglyRecurrent} that $$\estart{\tau^d_r}{\rhoERA}{}^{\gamma_1} < \Ref (\rER, \BT(\rER, r)^c)^{\gamma_1} \ \mu (\BT(\rER, r))^{\gamma_1}.$$ Therefore, using Propositions \ref{prop:vol UB prob UB}, \ref{prop:res to boundary general}, \ref{prop:metric balls compare} and Cauchy-Schwarz we can find $\gamma_1$ positive but small enough that the required expectation is finite.

For the upper bound in part b) we can work directly with the resistance metric. It follows from \cite[Proposition 3.1(a)]{KumMisumiHKStronglyRecurrent} that there exists $c_1<\infty$ such that on the event $\mu (B^R(r)) \geq \theta^{-1} v^R(r)$,
\[
p_{rv^R(r)}(0,0) \leq \frac{c_1\theta}{v^R(I^R(rv^R(r)))}.
\]
Note that Proposition \ref{prop:vol LB prob UB 2} gives polynomial lower tail decay for $\mu (B^R(r))$, which therefore transfers to polynomial upper tail decay for $p_{rv^R(r)}(0,0)$ and b) follows.


The upper bound in part c) follows from \eqref{eqn:poly tail exit time LB}. In particular, if for some $K \in (0, \infty)$ we replace $\theta$ with $\theta^K$ in \eqref{eqn:poly tail exit time LB}, and then set $\lambda=\theta$, since $c_4$ can be chosen to be polynomial in $\lambda$ we can then choose $K<\infty$ large enough that $\theta^K > c_4$ for all $\theta \geq 2$ and we obtain polynomial tail decay in $\theta$ for the probability in \eqref{eqn:poly tail exit time LB}.
\end{proof}
In particular, this gives all the results of Theorem \ref{thm:dec main ann RW}. Note that it should not in general be possible to obtain all of the same results without taking lower powers, since it is known that the expected volume of a ball in $\Tai$ is infinite, which has similar consequences for some of the quantities below. See \cite[Theorem 1.2]{CroyKumRWGWTreeInfiniteVar} for analogous results for $\Tai$.

\section{Examples}\label{sctn:examples}
In this section, we will consider several examples of graph sequences $(G_n)_{n \geq 1}$ and verify that they satisfy the conditions of Assumption \ref{assn:sec}.

In terms of exponents, the main results of the previous sections are that we established the following exponents of $\TERa$:
\begin{enumerate}
\item The \textbf{fractal dimension} is equal to $\frac{\alpha (\sad \wedge 1)}{(\alpha - 1)(\fv \wedge 1)}$.
\item The \textbf{spectral dimension} is equal to $\frac{2\alpha (\saR \wedge 1)}{\alpha (\saR \wedge 1) + (\alpha - 1)(\fv \wedge 1)}$.
\item The \textbf{displacement exponent} is equal to $\frac{(\saR \wedge 1)(\alpha -1)(\fv \wedge 1)}{(\alpha -1)(\fv \wedge 1)(\sad \wedge 1) + \alpha(\saR \wedge 1)(\sad \wedge 1)}$ (this is also the inverse of the \textbf{walk dimension}).
\end{enumerate}
Below, we give the values of these exponents for several examples of interest. Since several of these examples are just used as toy models, some of the discussion is informal and non-rigorous. As in the previous section, to study the random walk we are interested exclusively in the degree measure, which we will denote by $\mu$. Note that in many of the cases considered below this is equivalent to the counting measure (up to multiplication by a constant) so we may instead consider this in some of the volume estimates that we quote below.


\begin{table}[h!]
\begin{center}
\begin{tabular}{ |c|c|c|c|c|c| }
\hline
\textbf{Inserted graph} & \textbf{Range} & \textbf{Volume} & \textbf{Spectral dim} & \textbf{Displacement} \\ 
 \hline
 Star (tree) & all & $\frac{\alpha}{\alpha -1}$ & $\frac{2\alpha}{2\alpha - 1}$ & $\frac{\alpha - 1}{2\alpha - 1}$ \\ 
 \hline
 Loop (looptree) & all & $\alpha$ & $\frac{2\alpha}{\alpha + 1}$ & $\frac{1}{\alpha + 1}$ \\ 
 \hline
\multirow{2}{*}{$\beta$-stable trees} & $\frac{\beta}{\beta - 1} \geq \frac{1}{\alpha - 1}$ & $
 \frac{\alpha}{\alpha - 1}$ &  $\frac{2\alpha}{2\alpha - 1}$ & $\frac{\alpha - 1}{2\alpha - 1}$ \\
  \cline{2-5}
 & $\frac{\beta}{\beta - 1} < \frac{1}{\alpha - 1}$ & $\frac{\beta \alpha}{\beta-1}$ & $\frac{2\beta \alpha}{\beta - 1 + \beta \alpha}$ & $\frac{\beta - 1}{\beta - 1 + \beta \alpha}$ \\
 \hline
 Finite variance dissections, & $\alpha \geq \frac{3}{2}$ & $\frac{\alpha}{\alpha - 1}$ &
$\frac{2\alpha}{2\alpha-1}$ & $\frac{\alpha - 1}{2\alpha - 1}$ \\ 
\cline{2-5}
if $\mu (k, \infty) = O \left(k^{-(1+\frac{\alpha + \epsilon}{\alpha -1})}\right)$ & $\alpha \leq \frac{3}{2}$ & $2 \alpha$ & $\frac{4 \alpha}{2 \alpha + 1}$ & $\frac{1}{2\alpha  + 1}$ \\
 \hline
Critical & $\alpha \geq \frac{3}{2}$ & $
\frac{\alpha}{\alpha - 1}$ & $\frac{2\alpha}{2\alpha - 1}$ & $
\frac{\alpha -1}{2\alpha -1}$. \\
\cline{2-5}
\ER ($n$) & $\alpha \leq \frac{3}{2}$ & $2 \alpha$ &
$\frac{4\alpha}{2\alpha+1}$ & $\frac{1}{2\alpha + 1}$ \\
 \hline
Critical & $\alpha \geq \beta$ & $
\frac{\alpha}{\alpha - 1}$ & $
\frac{2\alpha }{2\alpha - 1}$ & $\frac{\alpha -1}{2\alpha -1}$ \\
 \cline{2-5}
\multirow{2}{*}{\ER ($n^{\beta}$)} & $\frac{\beta + 2}{2} \leq \alpha \leq \beta$ & $\frac{\beta}{\alpha - 1}$ & $\frac{2 \beta}{\beta + \alpha - 1}$ & $
\frac{\alpha -1}{\alpha -1 + \beta}$ \\
 \cline{2-5}
 & $\alpha \leq \frac{\beta + 2}{2}$ & $2$ & $\frac{4}{3}$ & $\frac{1}{3}$ \\
\hline
  Sierpinski & $\alpha \geq \frac{\log 10 - \log 3}{\log 2}$ & $\frac{\alpha}{\alpha -1}$ & $\frac{2\alpha}{2\alpha - 1}$ & $\frac{\alpha - 1}{2\alpha -1}$ \\ 
 \cline{2-5}
  \multirow{2}{*}{triangle} & $\frac{\log 3}{\log 2} \leq \alpha < \frac{\log 10 - \log 3}{\log 2} $ & $\alpha$ & $\frac{2\alpha \log 2}{\alpha \log 2 + \log5 - \log 3}$ & $\frac{\log 2}{\alpha \log 2 + \log5 - \log 3}$ \\ 
 \cline{2-5}
  & $\alpha < \frac{\log 3}{\log 2}$ & $\frac{\log 3}{\log 2}$ & $\frac{2\log 3}{\log 5}$ & $\frac{\log 2}{\log 5}$ \\ 
 \hline
Complete graph & all & $\frac{2}{\alpha - 1}$ & $\frac{4}{\alpha + 1}$ & $\frac{\alpha -1}{\alpha +1}$ \\ 
 \hline
2d box & all & $2$ & $\frac{4}{\alpha + 1}$ & $\frac{1}{\alpha +1}$ \\ 
 \hline
\end{tabular}
\end{center}\caption{Quenched exponents for the models considered below.}\label{table:RW exponents}
\end{table}

Some other examples of interest are also discussed in \cite[Section 7]{SenStefStufDecStableTrees}.

\subsection{Trees}\label{sctn:example just tree}
By inserting an appropriate ``star'' graph at every vertex, or simply repeating the arguments employed in the previous section directly for trees, we recover some results for random walks on critical Galton--Watson trees with offspring distribution satisfying \eqref{eqn:offspring tails}, conditioned to survive. We do not go into the details, but in this setting Assumption \ref{assn:sec} is effectively satisfied with $d = R= \infty$ and $v=1$ (though to make this rigorous, it easier just to repeat the arguments directly with these trees in mind). Random walks on these trees were studied by Croydon and Kumagai in \cite{CroyKumRWGWTreeInfiniteVar}, and we recover the exponents they establish there, as given in Table \ref{table:RW exponents}.

We note that it is possible to improve the polynomial tails of Section \ref{sctn:vol LBs} and \cite[Proposition 2.6]{CroyKumRWGWTreeInfiniteVar} by adapting the argument used for looptrees in \cite[Lemma 3.2]{BjornStef}, and obtain that, for all $r, \lambda \geq 1$ on $\Tai$,
\begin{align*}
\prb{\mu (B(\rho, r)) \leq r^{\frac{1}{\alpha - 1}} \lambda^{-1}} \leq Ce^{-c \lambda^{\frac{\alpha - 1}{\alpha}}}.
\end{align*}

For more detailed results on the tree case, see \cite{CroyKumRWGWTreeInfiniteVar} and \cite[Proposition 6]{KortSubexp}.

\subsection{Looptrees}\label{sctn:looptree example}
By inserting deterministic loops at each vertex, we also recover the discrete looptree model that was considered more thoroughly in \cite{BjornStef}. In this case, $d=R=v=1$. 
In fact, in this case it is also possible to repeat the arguments of \cite{ArchBMCompactLooptrees} used for the continuum case to get stronger tail decay, and show that, for any $\delta > 0$, there exist $C_{\delta}, c_{\delta} \in (0,\infty)$ such that for all $\lambda \geq 1$,
\begin{align*}
\prb{\mu(B(\rho, r)) \geq r^{\alpha} \lambda} &\leq C_{\delta}e^{-c_{\delta}\lambda^{\frac{\alpha - \delta}{\alpha + 1}}}.
\end{align*}
Similarly in terms of the lower tail decay, the following bound was shown in \cite[Equation (3.23)]{BjornStef} for all $\lambda \geq 1$:
\begin{align*}
\prb{\mu(B(\rho, r)) \leq r^{\alpha} \lambda^{-1}} &\leq C_{\delta} e^{-c_{\delta} \lambda^{\frac{1-\delta}{\alpha}}}.
\end{align*}
Since the inserted graphs are deterministic in this case, it follows straightforwardly that Assumption \ref{assn:sec} is satisfied, so that both the annealed and quenched spectral dimensions are $\frac{2\alpha}{\alpha + 1}$, and the displacement exponent is $\frac{1}{\alpha + 1}$. This recovers results of \cite{BjornStef}; see that paper for more details.

\subsection{Inserting trees}\label{sctn:inserting trees}
Since we have good control on volumes in trees, we could also insert a separate Galton--Watson tree conditioned to have $n$ leaves at each vertex of degree $n$. To establish the volume exponents in this case, note that since the number of edges of a tree is one less than the total number of vertices, if $T$ is an unconditioned Galton--Watson tree with offspring distribution $\hat{\xi}(x) \sim cx^{-\beta}$ as $x \rightarrow \infty$ for some $\beta \in (1,2)$, and $l(T)$ is its number of leaves, then (applying \cite[Proposition 1.6]{KortIVPLeaves}, Lemma \ref{lem:stable sum tail prob lower tail app}$(iii)$ and an LDP) there exist $c, C \in (0, \infty)$ such that
\begin{align*}
\prb{ l(T) = n, |T| \geq \lambda n} = \sum_{p \geq \lambda n} \frac{1}{p} \prb{ W_{p-n}' = n-1} \prb{S_p = n} &\leq \frac{C}{\lambda n} e^{-c\lambda - c \lambda n},
\end{align*}
where $W'$ is a random walk started from zero with jump distribution $\eta (i) = \frac{p_{i+1}}{1-p_0}$ for $i \geq 1$, and $S_p$ is a sum of $p$ independent \textsf{Bernoulli}($p_0$) random variables. 
Using also the asymptotic of \cite[Theorem 3.1(ii)]{KortIVPLeaves} that there exists $c' \in (0, \infty)$ such that $\prb{ l(T) = n} \sim c' n^{-(1+ \frac{1}{\beta})}$ as $n \rightarrow \infty$, we deduce that, uniformly in $n \geq 1$,
\[
\prcondb{|T| \geq \lambda n}{ l(T) = n}{} \leq \frac{Cn^{1+ \frac{1}{\beta}}}{c'\lambda n} e^{-c\lambda}e^{- c \lambda n} = o(e^{-c \lambda})
\]
as $\lambda \to \infty$. Clearly also
\[
\prcondb{|T| \leq \lambda^{-1} n}{ l(T) = n}{} =0
\]
for all $\lambda \geq 1, n \geq 1$, so we deduce that Assumption \ref{assn:sec}(V) is satisfied and $v=1, \sav = \alpha - 1$, and $\fv = \alpha$.

To bound $\prcondb{\diam (T) \geq \lambda n^{1 - \frac{1}{\beta}}}{ l(T) = n}{}$, we first bound the quantity
\[
\prcondb{ l(T) \in [n, 2n]}{\diam (T) \geq \lambda n^{1 - \frac{1}{\beta}}}{}
\]
by decomposing along the Williams' spine (the spine of maximal height), which we know has length at least $\frac{1}{2}\lambda n^{1 - \frac{1}{\beta}}$ on the event $\{\diam (T) \geq \lambda n^{1 - \frac{1}{\beta}}\}$. By Proposition \ref{prop:spinal offspring LB}, we know that for any vertex $v$ on the Williams' spine within distance $\frac{1}{4}\lambda n^{1 - \frac{1}{\beta}}$ of the root, there exists a constant $c$ such that for all $x \leq \lambda^{\frac{1}{\beta -1}} n^{\frac{1}{\beta}}$, $\prb{\deg v \geq x} \geq cx^{-(\beta - 1)}$, independently for each such $v$. Therefore, letting $v_1, v_2, \ldots v_{\frac{1}{4}\lambda n^{1 - \frac{1}{\beta}}}$ denote the vertices on the Williams' spine within distance $\frac{1}{4}\lambda n^{1 - \frac{1}{\beta}}$ of the root, listed in order, we have that the sequence $(\deg v_i)_{i=1}^{\frac{1}{4}\lambda n^{1 - \frac{1}{\beta}}}$ stochastically dominates a sequence $(Y_i)_{i=1}^{\frac{1}{4}\lambda n^{1 - \frac{1}{\beta}}}$ of independent random variables satisfying $\prb{Y_i \geq x} \geq cx^{-(\beta - 1)}$ for all $x \leq  \lambda^{\frac{1}{\beta -1}} n^{\frac{1}{\beta}}$ and $\prb{Y_i \geq \lambda^{\frac{1}{\beta -1}} n^{\frac{1}{\beta}}} =0$. In particular, if $Z$ is a non-negative random variable satisfying $\prb{Z \geq x} \geq cx^{-(\beta - 1)}$ for all $x \geq 0$, then
\[
\Eb{e^{-\theta (Y_i-2)}} = \Eb{e^{-\theta Y_i}\mathbbm{1}\{Y_i \leq \lambda^{\frac{1}{\beta -1}} n^{\frac{1}{\beta}}\}} + \Eb{e^{-\theta Y_i}\mathbbm{1}\{Y_i > \lambda^{\frac{1}{\beta -1}} n^{\frac{1}{\beta}}\}} \leq \Eb{e^{-\theta Z}} + \prb{Y_i \geq \lambda^{\frac{1}{\beta -1}} n^{\frac{1}{\beta}}}.
\]
In particular, taking $\theta = cn^{\frac{-1}{\beta}}$ then by Lemma \ref{lem:Tauberian at 0 app} we deduce that there exist $c', c''>0$ such that for all $n \geq 1$:
\begin{equation}\label{eqn:moment deg comparison}
\Eb{e^{-cn^{\frac{-1}{\beta}}(Y_i-2)}} \leq 1 - c'n^{\frac{-(\beta - 1)}{\beta}} \leq \exp \{ - c''n^{\frac{-(\beta - 1)}{\beta}}\}.
\end{equation}
Therefore, letting $v_1, v_2, \ldots v_{\frac{1}{4}\lambda n^{1 - \frac{1}{\beta}}}$ denote the vertices on the Williams' spine within distance $\frac{1}{4}\lambda n^{1 - \frac{1}{\beta}}$ of the root, $(T_{i,j})_{i=1}^{\deg v_j -2}$ denote the subtrees emanating from all of the non-spinal offspring of vertex $v_j$, $l(T_{i,j})$ denote the number of leaves in each $T_{i,j}$, again using the asymptotic of \cite[Theorem 3.1(ii)]{KortIVPLeaves} that $\prb{ l(T) = n} \sim c' n^{-(1+ \frac{1}{\beta})}$ as $n \rightarrow \infty$ in Lemma \ref{lem:Tauberian at 0 app}, and then taking $\theta = n^{-1}$ and using \eqref{eqn:moment deg comparison} in the final line, we deduce that there exist constants $c', c > 0$ and $C<\infty$ such that
\begin{align*}
    \prcondb{ l(T) \leq 2n}{\diam (T) \geq \lambda n^{1 - \frac{1}{\beta}}}{} &\leq \Eb{\exp \left\{ - \theta \sum_{j=1}^{\frac{1}{4}\lambda n^{1 - \frac{1}{\beta}}} \sum_{i=1}^{\deg(v_j)-2} l(T_{i,j}) \right\}} e^{2\theta n} \\
    &=\Eb{\econdb{\exp \left\{ - \theta \sum_{j=1}^{\frac{1}{4}\lambda n^{1 - \frac{1}{\beta}}} \sum_{i=1}^{\deg(v_j)-2} l(T_{i,j}) \right\}}{(\deg(v_j))_{j=1}^{\frac{1}{4}\lambda n^{1 - \frac{1}{\beta}}}}{}} e^{2\theta n}   \\
   &\leq\Eb{\Eb{\exp \left\{ - \theta l(T_{i,j}) \right\}}^{(Y_j-2)}}^{\frac{1}{4}\lambda n^{1 - \frac{1}{\beta}}} e^{2\theta n} \\
      &\leq\Eb{\exp \left\{ - c'\theta^{\frac{1}{\beta}} \right\}^{(Y_j-2)}}^{\frac{1}{4}\lambda n^{1 - \frac{1}{\beta}}} e^{2\theta n} \leq Ce^{-c\lambda}.
\end{align*}

To recover the desired bound, we then use monotonicity and Bayes' Law to deduce that there exists $C'< \infty$ such that for all $\lambda > 1$,
\begin{align*}
        \prcondb{\diam (T) \geq \lambda n^{1 - \frac{1}{\beta}}}{l(T) =n}{} &\leq \prcondb{\diam (T) \geq \lambda n^{1 - \frac{1}{\beta}}}{l(T) \in [n, 2n]}{} \leq C'e^{-c \lambda}.
\end{align*}
This establishes the required upper bound in Assumption \ref{assn:sec}(D). The required lower bound on $d^U(T)$ on the event $\{l(T)=n\}$ follows from \cite[Result (II) and Corollary 3.3]{KortIVPLeaves}.

Since resistance is equal to the graph distance on trees, this verifies the remaining conditions of Assumption \ref{assn:sec} with $R=d=\frac{\beta}{\beta - 1}$, so that $\sad = \saR = \frac{\beta (\alpha - 1)}{\beta - 1}$, so substituting these into our formulas we deduce the results in Table \ref{table:RW exponents}.

%
%

Note that we would expect the same results if we inserted a tree with $n$ vertices in total, rather than $n$ leaves, at a vertex of degree $n$, since the leaves asymptotically make up a constant proportion of the mass of the tree.

\subsection{Outerplanar maps: inserting dissected polygons}
Let $P_n$ be a convex polygon inscribed in the unit disc whose vertices correspond to the $n^{th}$ roots of unity. A \textit{dissection} of $P_n$ is obtained from $P_n$ by inserting a collection of chords that make up distinct diagonals of $P_n$: see Figure \ref{fig:BoltzmannDiss}.

\begin{figure}[h]
\includegraphics[width=12cm]{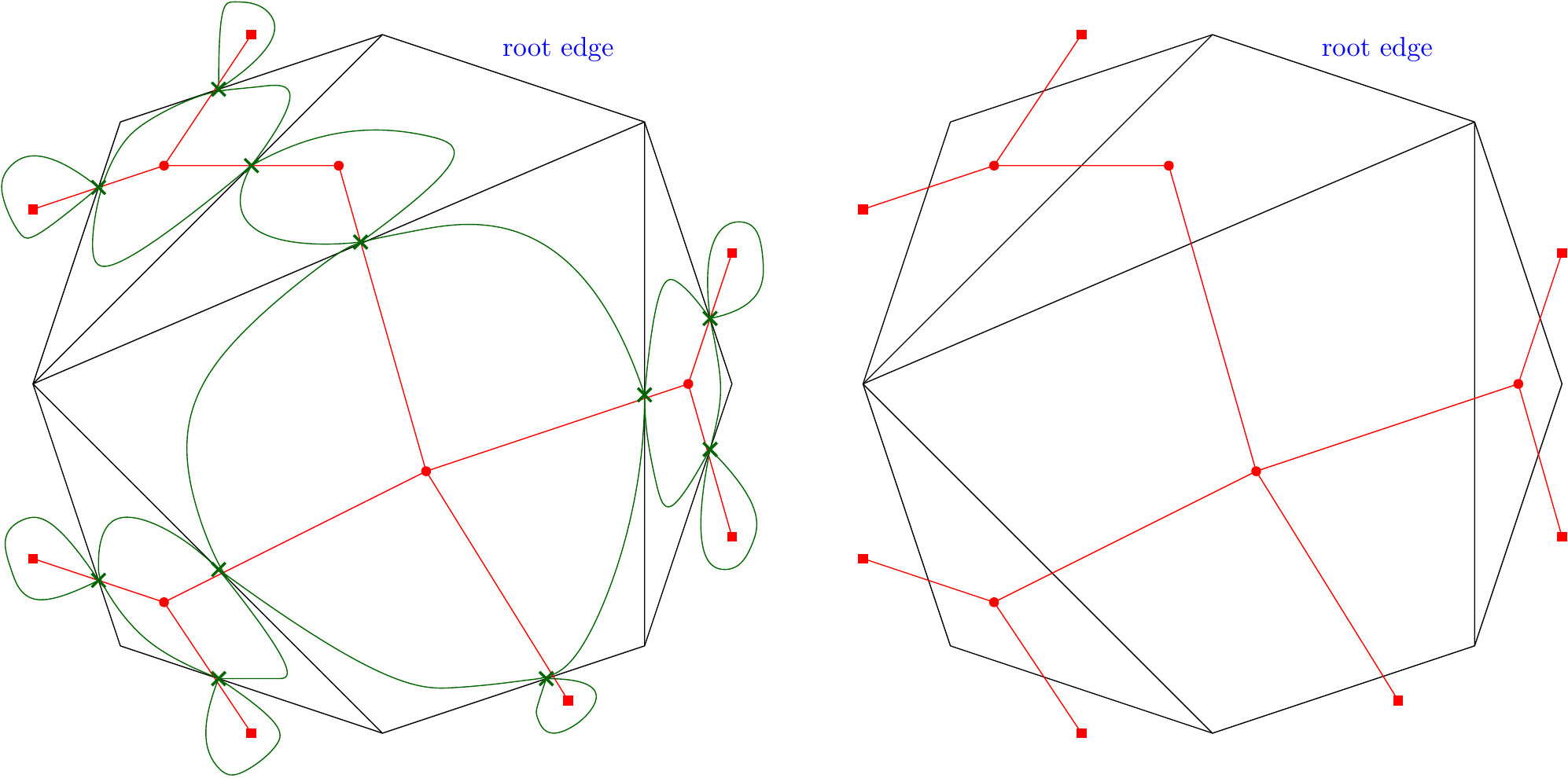}
\centering
\caption{A dissection and its inscribed tree and looptree.}\label{fig:BoltzmannDiss}
\end{figure}

If $\nu$ is a critical probability measure on the set $\{0, 2, 3 ,4, \ldots \}$, we can define a Boltzmann measure on dissections of a rooted $n$-gon, $\bP^{\nu}_n$, by setting
\[
\prnstart{D}{n}{\nu} \propto \prod_{F \in \textsf{Faces}(D)} \nu_{\deg f- 1}.
\]
For convenience we will assume that the support of $\nu$ is the entirety of the set $\{0, 2, 3 ,4, \ldots \}$. The measure $\bP^{\nu}_n$ is then well-defined under sensible assumptions on the tail of $\nu$ (see \cite[Section 1]{KortStableLam}).

Letting $\Dn$ denote a random Boltzmann dissection sampled according to $\bP^{\nu}_n$, $\Dn$ is now a natural candidate for decoration at a vertex of degree $n$ in $\Tai$. We will view $\Dn$ as a metric space (rather than as an embedding in the plane) by giving each edge of $\Dn$ length $1$. To establish exponents for the diameter and two-point function of $\Dn$, we will use a bijection between dissections of $P_n$ and trees with $n$ vertices, as illustrated in Figure \ref{fig:BoltzmannDiss}. It is shown in \cite[Proposition 1.4]{KortStableLam} that, if $T_n$ is the tree obtained from $\Dn$ in this way, then $T_n$ has the law of a Galton--Watson tree with offspring distribution $\nu$, conditioned on having $n-1$ leaves.

The main observation that will allow us to control the diameter and two-point function of $\Dn$ is that $\Dn$ looks a lot like $\Loop (T_n)$, as pointed out in \cite[Section 4.3]{RSLTCurKort}. We do not give the details, but Figure \ref{fig:BoltzmannDiss} suggests a natural way to define a correspondence between $\Dn$ and $\Loop (T_n)$. This straightforwardly enables us to verify Assumption \ref{assn:sec}(V), as in Section \ref{sctn:inserting trees}. However, the control on $\diam (T_n)$ and $\diam_{res} (T_n)$ is not necessarily as strong as we'd like, with
\[
\prcondb{\diam (\Dn) \geq n^{\frac{1}{\beta \wedge 2}}\lambda}{l(T_n) = n-1}{} \asymp \prcondb{\diam (\Loop (T_n)) \geq n^{\frac{1}{\beta \wedge 2}}\lambda}{l(T_n) = n-1}{} \leq c\lambda^{-(\beta -1)}
\]
where $\beta$ is such that $\nu (k, \infty) \leq ck^{-\beta}$, with $\diam_{res} (T_n)$ experiencing the same asymptotics. (This upper bound can be obtained by upper bounding $\diam (\Loop (T_n))$ by twice the maximum value attained by the Lukasiewicz path coding $T_n$). Assumption \ref{assn:sec} is therefore only satisfied in the finite variance case when $\beta > 1+\frac{\alpha}{\alpha -1}$. In this case $d=R=2$, and $v=1$, so we get the results in the table.
%

\subsection{Critical \ER}
Motivated by the example of critical percolation, we can also consider a model where we insert a critically percolated graph, or more precisely the connected component of an \ER graph in its critical window, by which we mean the graph $G(n, p)$ such that $p = \frac{1}{n} + \frac{t}{n^{\frac{4}{3}}}$ for some $t > 0$ (see e.g. \cite[Section 2]{GoldschmidtPIMS} for an introduction to this model and the critical window).

It is well-known that, at criticality, a connected component of $G(n,p)$ looks roughly like a critical Galton--Watson tree with an $O(1)$ number of ``surplus'' edges. Heuristically, this can be explained as follows: let $\C$ denote a connected component of $G(n,p)$, and let $v_0 \in \C$. We consider the ``exploration tree'' rooted at $v_0$, constructed as follows: first let $v_0$ be the root. Then consider all vertices connected to $v_0$ and let these form the next generation of the tree. The number of such vertices is \textsf{Binomial}($n-1, p$); denote this number $M_1$. Then, given a vertex $v_1$ in generation one, we can repeat this process to find all the \textit{new} neighbours of $v_1$, and define these to be the offspring of $v_1$: the number of offspring is therefore \textsf{Binomial}($n-1-M_1, p$). We can repeat this process inductively to explore the cluster in a depth-first way: this will produce a spanning tree of the cluster, and as long as the total number of vertices explored remains small compared to $n$, it is fairly accurate to approximate the offspring distribution of this tree by a \textsf{Binomial}($n-1, p$) distribution. At any stage, there is a small probability that a given vertex $v$ also has some neighbours that correspond to vertices that have already been discovered, so that in order to reconstruct $\C$ from its spanning tree we must add a few extra edges.

To avoid ambiguities, for this construction we will fix $t > 0$, set $p_n =\frac{1}{n^{\frac{3}{2}}} + \frac{t}{n^{2}}$ and let $G_n$ have the law of the largest component of $G(n^{\frac{3}{2}}, p_n)$ conditioned on having $n$ vertices (by \cite[Corollary 2]{AldousCritical97}, $n$ is therefore on the natural scale to be the size of the largest cluster of $G(n^{\frac{3}{2}}, p_n)$). Using the tree viewpoint, we can relate the volume, two-point function and diameter of critical connected \ER graphs to give the following results.

\begin{prop}\label{prop:ER results}
Take $G_n$ as above. Then there exist constants $c, C \in (0, \infty)$, depending on $t$, such that for all $n \geq 1, \lambda \geq 1$:
\begin{enumerate}[(i)]
\item $\prb{\diam (G_n) \geq \lambda \sqrt{n}} \vee \prb{\diam_{res} (G_n) \geq \lambda \sqrt{n}} \leq e^{-c \lambda^2}$.
\item $\prb{d^U_n (G_n) \geq \sqrt{n}} \wedge \prb{R^U_n (G_n) \geq \sqrt{n}} \geq c > 0$.
\item $\prb{\mu (G_n) \geq \lambda n} \leq Ce^{-c\lambda^{\frac{2}{3}}}$.
\item $\prb{\mu (G_n) \geq n-1} =1$.
\end{enumerate}
\end{prop}
\begin{proof}
We just sketch the proof. For part $(i)$ with the graph distance, the result follows by repeating the proof of the height bound of \cite[Theorem 1.1]{AddarioBShortandFat} (it does not quite follow directly since our exploration tree is not quite a critical Galton--Watson tree, but we are close enough that the proof still works, and being slightly subcritical is intuitively helpful for this bound anyway since this corresponds to more of a condensation regime). This also gives the resistance result since the resistance is upper bounded by the graph distance. For part $(ii)$, we first condition on having zero surplus, which has strictly positive probability in the limit. In this case the exploration tree is again close to a critical Galton--Watson tree, and the resistance is equal to the graph distance. The result then follows since the offspring distribution of the tree is close to \textsf{Poisson}($1$), which corresponds to a uniform labelled Galton--Watson tree, and in this case it is known that on rescaling by $\sqrt{n}$ the two-point function satisfies $\prb{d_n^U(T_n) = \lfloor \lambda \sqrt{n} \rfloor} \sim \frac{\lambda}{\sqrt{2n}} e^{-\frac{\lambda^2}{2}}$ as $n \rightarrow \infty$ (e.g. see \cite{FlajoletSolvableUrn}).

To control the volume (i.e. number of edges) we control the surplus. To do this, note that given a vertex $v_k$ in the exploration tree, there can only be extra edges joining it somewhere within the same generation, or to an adjacent generation (otherwise this disrupts the generation structure of the tree). Therefore, if $v$ is a vertex of the tree and $D_m$ denotes the $m^{th}$ generation of the tree, we can introduce a \textsf{Binomial}($|D_{|v|}| + |D_{|v|-1}|, p_n$) random variable which we denote $S_{v}$, and the total surplus is upper bounded by summing these over all vertices in the tree. Then, again after taking care of the necessary details that our tree is not quite a critical Galton--Watson tree, we have by \cite[Theorem 1.1]{AddarioBDevroyeJansonSubGaussianTree} that there exist constants $c_t, C_t \in (0, \infty)$ such that $\prb{\sup_m |D_m| \geq \sqrt{n} \lambda^p} \leq C_te^{-c_t\lambda^{2p}}$ and $\prb{\textsf{Binomial} (2n^{\frac{3}{2}} \lambda^p, \frac{1}{n^{\frac{3}{2}}} + \frac{t}{n^{2}}) \geq \lambda} \leq C_te^{-c_t \lambda^{1-p}}$. We take $p = \frac{1}{3}$. On the complement of these events, the surplus is less than $\lambda$, so $\mu (G_n) \leq n+\lambda$, which gives part $(iii)$ (and more).
\end{proof}

%
We therefore deduce that Assumption \ref{assn:sec} is satisfied with $d=R=2, v=1$, and the fundamental exponents take the following values:
\[
\fv = \alpha, \hspace{1cm} \sad = \saR = 2(\alpha - 1).
\]

Rather than forcing all vertices of the inserted critical graph to be boundary vertices, we could also consider inserting an independent copy of $G_{n^{\beta}}$ at a vertex of degree $n$, for some $\beta \geq 1$ and uniformly choosing $n$ distinct vertices to be boundary vertices. In this case, it follows from Proposition \ref{prop:ER results} that $d=r=\frac{2}{\beta}$ and $v=\beta$, so that
\[
\fv = \frac{\alpha}{\beta}, \hspace{1cm} \sad = \saR = \frac{2(\alpha - 1)}{\beta}.
\]
Note that if $\beta > 2$, the local geometry always dominates and we never see the tree geometry.

\begin{rmk}
We have not written the details, but one would expect the same result on taking a critical configuration model in place of the \ER graph. We also anticipate that we could insert a $\beta$-stable graph, as considered in \cite{goldschmidthaassen2018stable, conchonkerjangoldschmidt2020stable} and we would get the same results as for inserting $\beta$-stable trees by making similar arguments to the \ER example considered above.
\end{rmk}

%
%
%
%

\subsection{Sierpinski triangle}
In order to gain insight into the effect of inserting fractal-type graphs, again as in the gasket-type structures expected in critical percolation on random planar maps, one could also consider the exponents obtained when inserting a Sierpinski triangle. Letting $\Trin$ be the $n^{th}$ level approximation to the Sierpinski triangle as defined in \cite[Section 2]{BarDiffonFract} (also depicted in Figure \ref{fig:sierpinksi triangles}), it is always the case that the boundary length of $\Trin$ is equal to $3\cdot 2^n$: therefore, if $m \in (3\cdot 2^n, 3\cdot 2^{n+1})$ for some $n \geq 1$, one would have to do appropriate ``surgery'' to the graph $G_{n+1}$ in order to define an appropriate version of ``the Sierpinski triangle with boundary length $m$''. We will not do this is explicitly here, and just give the appropriate volume bounds for the level $n$ approximation $\Trin$.

\begin{figure}[h]
\begin{subfigure}{.3\textwidth}
\includegraphics[height=3cm]{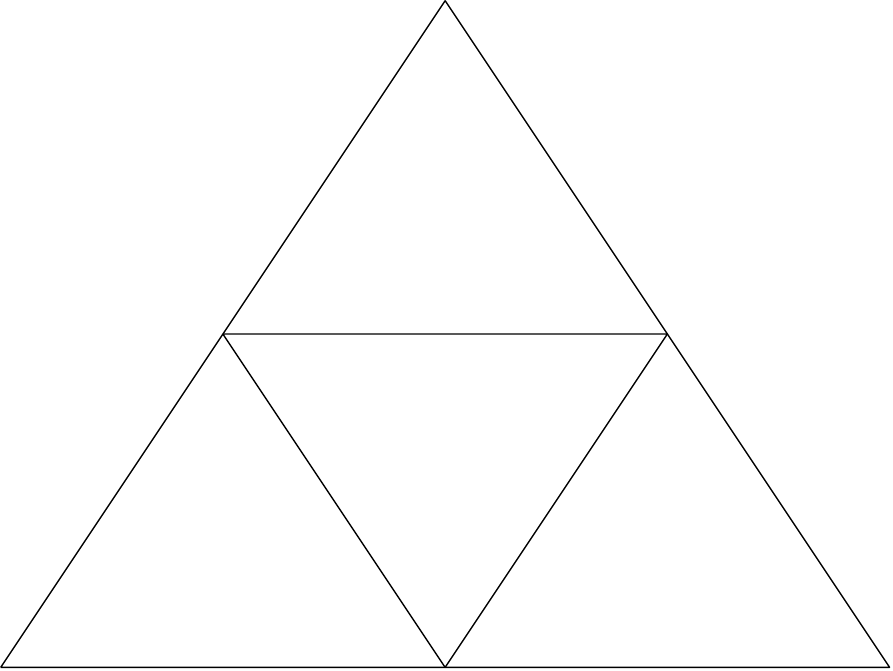}
\centering
\subcaption{$T^{\Delta}_1$}
\end{subfigure}
\begin{subfigure}{.3\textwidth}
\includegraphics[height=3cm]{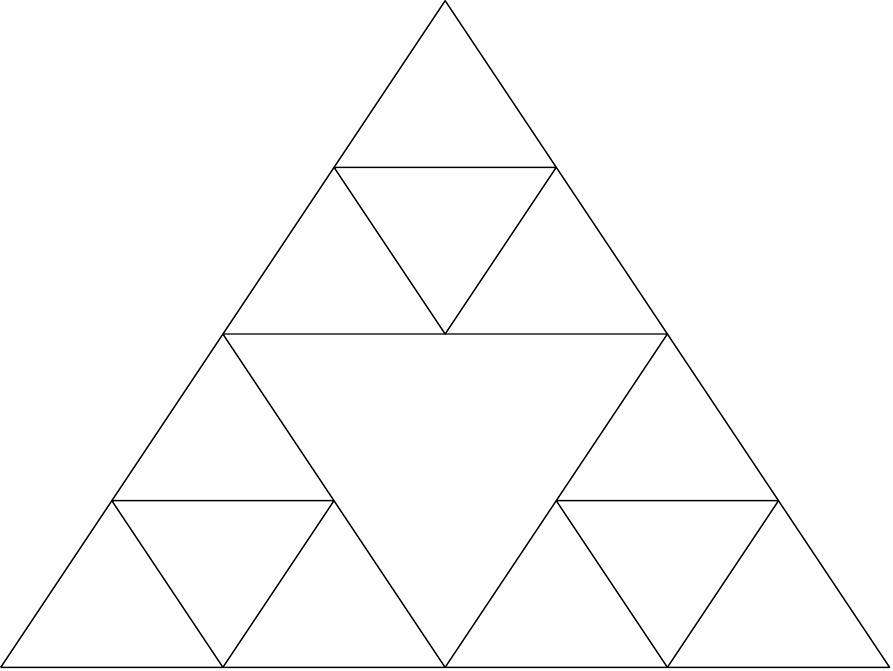}
\centering
\subcaption{$T^{\Delta}_2$}
\end{subfigure}
\begin{subfigure}{.3\textwidth}
\includegraphics[height=3cm]{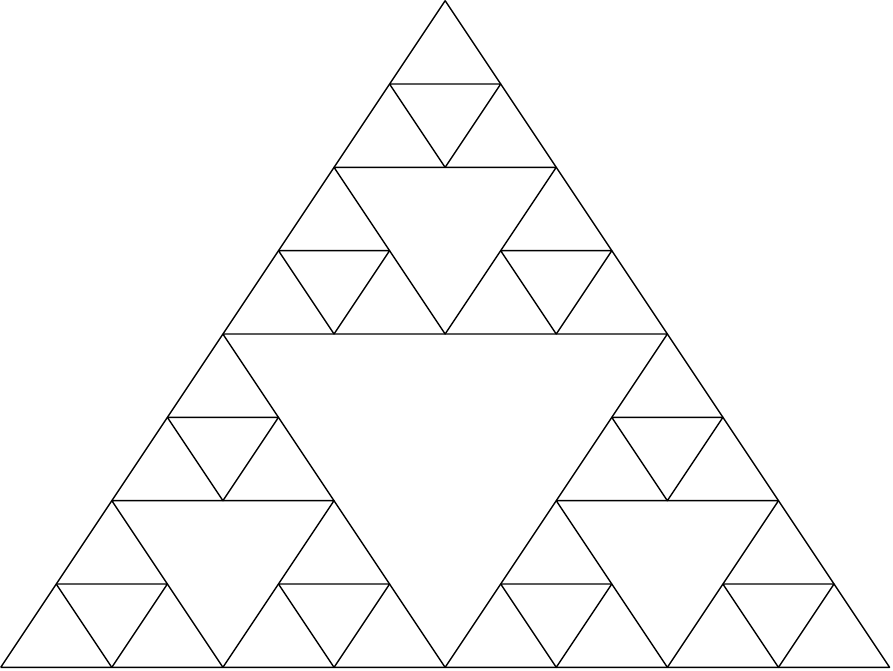}
\centering
\subcaption{$T^{\Delta}_3$}
\end{subfigure}
\caption{Sierpinski triangles}\label{fig:sierpinksi triangles}
\end{figure}

We can use the self-similarity of the Sierpinski triangle to study resistances, volumes and diameters of $\Trin$ as well, and give the (deterministic) results for these in Table \ref{table:Sier exponents}. The effective resistance bound can be obtained using the $\Delta - Y$ transformation (e.g. \cite[2.3.III]{LyonsPeresBook}): see also \cite[Section 2]{BarDiffonFract} for more explicit computations. Here we assume that $1$ and $2$ are the labels of two distinct extremal corners of $\Trin$.

\begin{table}[h!]
\begin{center}
\begin{tabular}{ |c|c|c|c|c|c| }
 \hline
\textbf{Boundary} & \textbf{Volume} & \textbf{$d(1,2)$} & \textbf{$\Ref (1,2)$} & \textbf{Diameter} & \textbf{Resistance diam} \\ 
 \hline
$3 \cdot 2^n$ & $3^{n+1}$ & $2^n$ & $\left( \frac{5}{3} \right)^n$ & $ 2^n$ & $\left( \frac{5}{3} \right)^n$ \\ 
\hline
$m$ & $m^{\frac{\log 3}{\log 2}}$ & $m$ & $m^{\frac{\log 5 - \log 3}{\log 2}}$ & $m$ & $m^{\frac{\log 5 - \log 3}{\log 2}}$ \\
 \hline
\end{tabular}
\caption{Volumes and distances in the $n^{th}$ level Sierpinski gasket, in terms of $m$ and $n$, where $m$ is the boundary length. Distances given up to multiplication by bounded positive constants.}\label{table:Sier exponents}
\end{center}
\end{table}

We can also (crudely) bound the diameters using the bound for the distances between extremal corners in the table: since to go from any point $x \in \Trin$ to any other point $y \in \Trin$ we have to pass through an edge of at most two triangles at each specific level $m \leq n$ (once on the ``way up'' from $x$, then once on the ``way down'' to $y$), we get that
\[
\diam (\Trin) \leq 2 \sum_{m=0}^n 2^m \leq 2^{n+2}, \hspace{1cm} \diamR (\Trin) \leq 2 \sum_{m=0}^n \left(\frac{5}{3} \right)^m \leq \left(\frac{5}{3} \right)^{n+1}.
\]
To get a similar lower bound for the distance between two uniform points on the boundary, note that there is a positive probability that the two points are in different ``sub-triangles'' of $T_1^{\Delta}$, and the distance will therefore be on the same order as $d(1,2)$ (and similarly for resistance).

Note also that the degrees of every vertex are either $2$ or $4$. Therefore the degree measure is always comparable to the number of vertices. Moreover, since the graphs and the bounds in Table \ref{table:Sier exponents} are deterministic, the functions giving the polynomial tail decay in Assumption \ref{assn:sec} are all zero for sufficiently large $\lambda$, so we have that $d = 1$, $R = \frac{\log 2}{\log 5 - \log 3}, v=\frac{\log 3}{\log 2}$, and obtain the following exponents:
\[
\sad = \alpha - 1, \hspace{1cm} \saR = \frac{(\alpha - 1) \log 2}{\log 5 - \log 3}, \hspace{1cm} \fv = \frac{\alpha \log 2}{\log 3},
\]
which lead to the claimed results.

\subsection{Supercritical \ER, the complete graph, or the 2-dimensional box}
This final subsection is only speculative. As well as critical Erd\"os-R\'enyi, one could also insert sufficiently supercritical \ER graphs, as well as the complete graph. It is well-known that, if $G(n,p)$ is the \ER graph on $n$ vertices and $p = \frac{\lambda}{n}$ for some $\lambda > 1$, then the largest connected component has order $n$ vertices, order $n^2$ edges, and diameter of order $\log n$ (e.g. see \cite[Theorem 1.1]{RiordanWormaldDiameterSparse}). The complete graph on $n$ vertices similarly has $n$ vertices, order $n^2$ edges, diameter $1$ and resistance diameter of order $\frac{1}{n}$. 

To fit these models into the framework of this paper, we therefore effectively want to take $d=R=\infty$, and $v=2$. Some care is needed to check that we can really do this, but we can dominate $\log n$ by $n^{\delta}$ for some sufficiently small $\delta>0$, and this also gives us very good control on the tail decay required for Assumption \ref{assn:sec} (D). Additionally, in sufficiently supercritical regimes resistance will actively stochastically decrease with $n$ which is clearly different to the assumptions of this paper; however, since ``most'' vertices in Kesten's tree are of low degree it is clear that asymptotically resistance in $\TERa$ should grow on the same order as distances in the underlying tree.

To define this model in the supercritical case, we let $G_n$ be the largest connected component of $G(Cn, p)$, conditioned to have $n$ vertices, where $C \geq 1$ is an appropriately chosen constant. Alternatively, we can let $G_n$ be the complete graph on $n$ vertices.

Setting $d=R=\infty$ (effectively) and $v=2$ therefore gives the results in the penultimate line of Table \ref{table:RW exponents}.

The same intuition works on inserting a two-dimensional lattice box with side lengths $\frac{1}{4}n$: the volume of the box is of order $n^2$, the graph distance across the box of order $n$, and the resistance of order $\log n$, which can be stochastically dominated by $n^{\delta}$ for all $\delta > 0$. We effectively want to take $R=\infty, d=1, v=2$, so that $\sad = \alpha - 1$, $\saR = \infty$, $\fv = \frac{\alpha}{2}$.

\begin{appendix}

\section{Appendix: sums of stable variables}
In this section we give some technical lemmas. The first result is used in the proof of Lemma \ref{lem:stable sum tail prob app}.


\begin{lem}\label{lem:truncated stable sum app}
Let $(X_i)_{i=1}^{\infty}$ be i.i.d. non-negative such that $\pr{X_1>x} \sim cx^{-\beta}$ as $x \to \infty$, for some $\beta \in (0,1]$ and $c \in (0, \infty)$, and for $k \geq 1$ let $\Tc = \inf\{i \geq 1: X_i > k\}$ (or equal to infinity if this set is empty). Set
\[
\Sc = \sum_{i=1}^{\Tc-1} X_i.
\]
\begin{enumerate}[(i)]
\item If $\beta \in (0, 1)$,
then 
there exist $c, C \in (0, \infty)$ such that $\pr{S^{(k)} \geq \lambda k} \leq Ce^{-c \lambda}$ for all $k, \lambda \geq 1$.
\item If $\beta = 1$, then there exist $c, C \in (0, \infty)$ such that $\pr{S^{(k)} \geq \lambda k \log k} \leq Ce^{-c \lambda}$ for all $k, \lambda \geq 1$.
\end{enumerate}
\end{lem}
\begin{proof}
%
\begin{enumerate}[(i)]
\item The proof is essentially the same as the argument for a similar result on \cite[p. 25]{RSLTCurKort}. Let $S_n = \sum_{i=1}^n X_i$, $H_0=0$, and $H_k = \inf\{ n \geq 0: S_n \geq k\}$ be the hitting time of $k$ for the random walk $(S_n)_{n \geq 1}$. Then, if $S^{(k)} \geq k \lambda$ it must be the case that $H_{k \lambda} \leq T^{(k)}$. Therefore, for any $A\geq 1$ we can write
\begin{align*}
\pr{S^{(k)} \geq k \lambda} &\leq \pr{H_{k \lambda} \leq T^{(k)}} \\
&\leq \pr{H_{2kA} \leq T^{(k)}, H_{4kA} \leq T^{(k)}, H_{6kA} \leq T^{(k)}, \ldots , H_{2 \lfloor \frac{1}{2A} \lambda \rfloor kA} \leq T^{(k)}} \\
&\leq \prod_{i=1}^{\lfloor \frac{1}{2A} \lambda \rfloor} \prcond{H_{2ikA} \leq T^{(k)}}{{H_{2(i-1)kA} \leq T^{(k)}}}{}.
\end{align*}
Since there are no jumps exceeding $k$ before time $T^{(k)}$, it follows that $S_{H_{2iAk}} \leq (2iA+1)k$ on the event $\{H_{2ikA} \leq T^{(k)}\}$ for all $i \geq 1$. Moreover, since $T^{(k)}$ has a geometric distribution, it therefore follows from the memoryless property that for all $i \geq 1$:
\begin{align*}
\lefteqn{\prcond{H_{2iAk} \leq T^{(k)}}{{H_{2(i-1)Ak} \leq T^{(k)}}}{}} \qquad & \\ &\leq \prcond{H_{2iAk} \leq T^{(k)}}{{H_{2(i-1)Ak} \leq T^{(k)}}, S_{H_{2(i-1)Ak}} \leq (2(i-1)A+1)k}{} \\
&\leq \pr{H_{Ak} \leq T^{(k)}}.
\end{align*}
Therefore, the exponential decay will follow once we can show that $\pr{H_{Ak} \leq T^{(k)}}$ can be bounded below $1$ uniformly in $k$. To show this, we use Markov's inequality and Wald's equality to write:
\begin{align*}
\pr{H_{Ak} \leq T^{(k)}} = \pr{S^{(k)} \geq Ak} \leq \E{S^{(k)}}A^{-1}k^{-1} \leq A^{-1}c_1\econd{X_1}{X_1 < k}{} k^{\beta} k^{-1} \leq A^{-1}c_2
\end{align*}
for all $k \geq 1$, where $c_1$ and $c_2$ are explicit constants depending on the tails of $X_1$. We can therefore choose $A>1$ so that the right hand side above is strictly less than $1$.
\item If $\beta = 1$, we can use the same proof but we write
\begin{align*}
\pr{S^{(k)} \geq \lambda k \log k} \leq \prod_{i=1}^{\lfloor \frac{1}{2A} \lambda \rfloor} \prcond{H_{2iAk \log k} \leq T^{(k)}}{{H_{2(i-1)Ak \log k} \leq T^{(k)}}}{} \leq \pr{H_{Ak\log k} \leq T^{(k)}}^{\lfloor \frac{1}{2A} \lambda \rfloor},
\end{align*}
and use that there exists $c<\infty$ such that $\econd{X_1}{X_1 < k}{} \leq 2c \log k$ to bound the final probability using  Markov's inequality and Wald's equality as in $(i)$.

\end{enumerate}
\end{proof}

We will also need the following lemma. The result should be standard; however we couldn't find a specific proof in the literature, so have provided one for completeness.

\begin{lem}\label{lem:stable sum tail prob app}
Let $(X_i)_{i \geq 1}$ be i.i.d. non-negative and such that $\pr{X_i \geq x} \sim cx^{-\beta}$ as $x \to \infty$ for some $\beta \in (0, \infty)$ and $c \in (0, \infty)$. For $n \geq 1$, let $S_n = \sum_{i=1}^n X_i$. 
\begin{enumerate}[(i)]
\item If $\beta < 1$, then there exists a constant $c' < \infty$ such that for each $n \geq 2$ and all $\lambda \geq 1$,
\[
\pr{S_n \geq n^{\frac{1}{\beta}} \lambda} \leq c'\lambda^{-\beta}.
\]
\item If $\beta = 1$, then there exists a constant $c' < \infty$ such that for each $n \geq 2$ and all $\lambda \geq 1$,
\[
\pr{S_n \geq \lambda n \log n} \leq c'\lambda^{-\beta}.
\]
\item If $\beta > 1$, then for any $\epsilon > 0$ there exists a constant $c_{\epsilon} < \infty$ such that for each $n \geq 2$ and all $\lambda \geq 1$, 
\[
\pr{S_n \geq \lambda n} \leq c_{\epsilon}\lambda^{-(\beta - \epsilon)}.
\]
This bound still holds if the $(X_i)_{i \geq 1}$ are not independent.
\end{enumerate}
\end{lem}
\begin{proof}
We first give the proof in case $(i)$. The proof is no doubt standard and this particular formulation follows a similar strategy to the analysis of \cite[Chapter 3, p.160]{ProbThEx}. Decompose $S_n$ as the sum $\hat{S}_n + \tilde{S}_n$, where
\[
\hat{S}_n = \sum_{i=1}^n X_i \mathbbm{1}\{X_i \geq n^{\frac{1}{\beta}}\}, \hspace{1cm} \tilde{S}_n = \sum_{i=1}^n X_i \mathbbm{1}\{X_i < n^{\frac{1}{\beta}}\}.
\]
We condition on the number of terms of $\hat{S}_n$, which is stochastically dominated by a \textsf{Binomial}($n, 2cn^{-1}$) random variable for large enough $n$.

Given that $N_n := |\{i \leq n: X_i \geq n^{\frac{1}{\beta}}\}| = m$, $\tilde{S}_n$ can be dealt with using Lemma \ref{lem:truncated stable sum app}: since we have $m+1$ copies of the sum considered there, it is necessary for one such copy to be at least $\frac{\lambda n^{\frac{1}{\beta}}}{2(m+1)}$ in order that $\tilde{S}_n \geq \frac{\lambda n^{\frac{1}{\beta}}}{2}$. Similarly, to control $\hat{S}_n$, note that it is necessary that at least one term of $\hat{S}_n$ is at least $\frac{\lambda n^{\frac{1}{\beta}}}{2m}$ in order that $\hat{S}_n \geq \frac{\lambda n^{\frac{1}{\beta}}}{2}$. Moreover, $\prcond{X_1 \geq \frac{\lambda n^{\frac{1}{\beta}}}{2m}}{X_1 \geq n^{\frac{1}{\beta}}}{} \leq 4^{\beta} m^{\beta}\lambda^{-\beta}$ for all sufficiently large $\lambda$. We can therefore write for all sufficiently large $n$ (using a union bound and Lemma \ref{lem:truncated stable sum app}), that for some constants $c', C, C', C'' \in (0, \infty)$:
\begin{align*}
\pr{S_n \geq \lambda n^{\frac{1}{\beta}}} &\leq \pr{\hat{S}_n \geq \frac{\lambda n^{\frac{1}{\beta}}}{2}} + \pr{\tilde{S}_n \geq \frac{\lambda n^{\frac{1}{\beta}}}{2}} \\
&\leq \sum_{m=0}^{\sqrt{\lambda}} \binom{n}{m} (2cn^{-1})^m (1-2cn^{-1})^{n-m} \\
&\hspace{1cm} \times \left[ \prcond{\hat{S}_n \geq \frac{\lambda n^{\frac{1}{\beta}}}{2}}{N_n=m}{} + \prcond{\tilde{S}_n \geq \frac{\lambda n^{\frac{1}{\beta}}}{2}}{N_n=m}{} \right] \\
&\qquad + \pr{\mathsf{Binomial}(n, 2cn^{-1}) \geq \sqrt{\lambda}} \\
&\leq \sum_{m=0}^{\sqrt{\lambda}} \binom{n}{m} (2cn^{-1})^m (1-2cn^{-1})^{n-m} \left[ m^{\beta + 1} 4^{\beta} \lambda^{-\beta} + Cme^{-c'\frac{\lambda }{2(m+1)}} \right] + e^{-c\sqrt{\lambda}} \\
&\leq \sum_{m=0}^{\sqrt{\lambda}} \frac{1}{m !} e^{-2c} \left(\frac{2c}{1-2cn^{-1}}\right)^{m} \left[ C'm^{\beta + 1} \lambda^{-\beta} \right] + e^{-c\sqrt{\lambda}}\\
&\leq \lambda^{-\beta} \sum_{m=0}^{\infty} \frac{C' m^{\beta + 1} }{m !} \left(\frac{2c}{1-2cn^{-1}}\right)^{m} + e^{-c\sqrt{\lambda}} \leq C''\lambda^{-\beta}.
\end{align*}
The proof is the same in case $(ii)$ and we just import an extra log term from the application of Lemma \ref{lem:truncated stable sum app}.

In case $(iii)$, we can choose $\epsilon >0$ small enough that $\beta - \epsilon > 1$ and apply Markov's inequality then H\"older's inequality with $p = \beta - \epsilon$ to get that
\begin{align*}
\pr{\sum_{i=1}^n X^{(i)} \geq \lambda n} \leq \E{\left( \sum_{i=1}^n X^{(i)}\right)^p} n^{-p} \lambda^{-p} \leq n^p \E{\left(X^{(1)}\right)^p} n^{-p} \lambda^{-p} \leq c_{\epsilon}\lambda^{-(\beta - \epsilon)}.
\end{align*}
\end{proof}

The next lemma deals with the case where the random variables $(X_i)_{i \geq 1}$ also have a logarithmic term.

\begin{lem}\label{lem:stable sum tail prob log app}
Let $(X_i)_{i \geq 1}$ be i.i.d. non-negative and such that $\pr{X_i \geq x \log x} \sim cx^{-\beta}$ as $x \to \infty$ for some $\beta \in (0,1)$ and $c \in (0, \infty)$. For $n \geq 1$, let $S_n = \sum_{i=1}^n X_i$. Then for any $\epsilon>0$ there exists a constant $c_{\epsilon} < \infty$ such that for each $n \geq 2$ and all $\lambda \geq 1$,
\[
\pr{S_n \geq \lambda n^{\frac{1}{\beta}} \log n} \leq c_{\epsilon}\lambda^{-(\beta - \epsilon)}.
\]
\end{lem}
\begin{proof}
Fix $n \geq 1$ and for each $i \leq n$, set $Y_{i, \lambda} = \frac{X_i}{\log (\lambda^{\beta} n)} \mathbbm{1}\{X_i < (\lambda^{\beta} n)^{\frac{1}{\beta}} \log ((\lambda^{\beta} n)^{\frac{1}{\beta}})\}$. Note that $(Y_{i, \lambda})_{i \geq n}$ is an i.i.d. sequence and $Y_{i, \lambda} \leq \frac{1}{\beta}(\lambda^{\beta} n)^{\frac{1}{\beta}}$ deterministically. Moreover, for all sufficiently large $x$ we have for each $i \leq n$ that
\begin{align*}
\pr{Y_{i, \lambda} \geq x} = \pr{Y_{i, \lambda} \geq x}\mathbbm{1}\left\{x \leq \frac{1}{\beta}(\lambda^{\beta} n)^{\frac{1}{\beta}}\right\} &= \pr{\frac{X_i}{\log (\lambda^{\beta} n)} \geq x}\mathbbm{1}\left\{x \leq \frac{1}{\beta}(\lambda^{\beta} n)^{\frac{1}{\beta}}\right\} \\
&\leq \pr{\frac{X_i}{\log ((\beta x)^{\beta})} \geq x} \leq \pr{X_i \geq \beta x \log (\beta x)} \leq 2c\beta^{-\beta}x^{-\beta}.
\end{align*}
Note that this final bound does not depend on $\lambda$. For all sufficiently large $\lambda$, we therefore have by a union bound that
\begin{align*}
\pr{S_n \geq n^{\frac{1}{\beta}} (\log n) \lambda} &\leq \pr{\sup_{i \leq n} X_i \geq (\lambda^{\beta} n)^{\frac{1}{\beta}} \log ((\lambda^{\beta} n)^{\frac{1}{\beta}})} + \pr{S_n \geq n^{\frac{1}{\beta}} (\log n) \lambda, \ \sup_{i \leq n} X_i < (\lambda^{\beta} n)^{\frac{1}{\beta}} \log ((\lambda^{\beta} n)^{\frac{1}{\beta}})} \\
&\leq n \pr{ X_1 \geq (\lambda^{\beta} n)^{\frac{1}{\beta}} \log ((\lambda^{\beta} n)^{\frac{1}{\beta}})} + \pr{\sum_{i=1}^{n}Y_{i, \lambda}  \geq \frac{n^{\frac{1}{\beta}} \lambda}{\log (\lambda^{\beta})}} \\
&\leq 2c\lambda^{-\beta} + \pr{\sum_{i=1}^{n}Y_{i, \lambda}  \geq n^{\frac{1}{\beta}} \lambda}.
\end{align*}
The result therefore follows from applying the result of Lemma \ref{lem:stable sum tail prob app}$(i)$ to $\sum_{i=1}^{n}Y_{i, \lambda}$.
\end{proof}

We also have a bound for the lower tails.
\begin{lem}\label{lem:stable sum tail prob lower tail app}
Let $(X_i)_{i \geq 1}$ be i.i.d. and non-negative and suppose there exist $c, c''>0$, $C<\infty$ and $\beta \in (0, \infty)$ such that $\pr{X_i \geq x} \leq Cx^{-\beta}$ for all $x \geq 1$, and $\pr{X_i \geq x} \geq cx^{-\beta}$ for all $1 \leq x \leq c''n^{\frac{1}{\beta}}$. For $n \geq 1$, let $S_n = \sum_{i=1}^n X_i$.
\begin{enumerate}[(i)]
\item If $\beta < 1$, then there exists a constant $c' >0$ such that for each $n \geq 2$, $1 \leq \lambda \leq n^{\frac{1}{\beta}}$,
\[
\pr{S_n \leq n^{\frac{1}{\beta}} \lambda^{-1}} \leq e^{-c'\lambda^{\beta}}.
\]
\item If $\beta = 1$, then there exists a constant $c' >0$ such that for each $n \geq 2$, $1 \leq \lambda \leq n \log n$,
\[
\pr{S_n \leq \lambda^{-1} n \log n} \leq e^{-c'\lambda (\log \lambda)^{-1}}.
\]
\item If $\beta > 1$, then there exists a constant $c'>0$ such that for each $n \geq 2$, $1 \leq \lambda \leq n$,
\[
\pr{S_n \leq \lambda^{-1} n} \leq e^{-c'\lambda}.
\]
\end{enumerate} 
\begin{proof}
\begin{enumerate}[(i)]
\item For all sufficiently large $\lambda$ we can write for any $n \geq 1$:
\begin{align*}
\pr{S_n \leq n^{\frac{1}{\beta}} \lambda^{-1}} \leq \pr{\nexists i \leq n: X_i \geq n^{\frac{1}{\beta}} \lambda^{-1}} \leq (1-\frac{c}{2}n^{-1}\lambda^{\beta})^n \leq e^{-\frac{c}{2} \lambda^{\beta}}.
\end{align*}
It extends to all $1 \leq \lambda \leq n^{\frac{1}{\beta}}$ by modifying the constant slightly.
\item First note that, by assumption, we can couple the sequence $(X_i)_{i \geq 1}$ with another non-negative i.i.d. sequence $(\tilde{X}_i)_{i \geq 1}$ satisfying $\pr{X_i \geq x} \geq cx^{-\beta}$ \textit{for all} $x \geq 1$ and such that $X_i = \tilde{X}_i$ whenever $\tilde{X}_i \leq c''n^{\frac{1}{\beta}}$. Let $\tilde{S}_n = \sum_{i=1}^n \tilde{X}_i$.

First note that by \cite[Equation (1.3)]{BergerCauchy}, Markov's inequality and a union bound, there exist $A>0, A', A''<\infty$ such that, for all sufficiently large $n$,
\begin{align*}
\pr{\tilde{S}_n \leq 2An \log n} &\leq \frac{1}{4} \\
\pr{\#\{i \leq n: X_i \geq c''n^{\frac{1}{\beta}}\} \geq A'} &\leq \frac{1}{4} \\
\pr{\sup_{i \leq n} X_i \geq A'' n^{\frac{1}{\beta}}} &\leq \frac{1}{4}.
\end{align*}
Moreover, on the complement of these three events, we have that (provided $n$ is sufficiently large)
\[
S_n \geq \tilde{S}_n - A' \cdot A'' n^{\frac{1}{\beta}} \geq An\log n.
\]
Combining these we deduce that there exists $N<\infty$ such that $\pr{\tilde{S}_n \leq An \log n} \leq \frac{3}{4}$ for all $n\geq N$. By modifying $A$ slightly, we can also ensure this holds for all $n \geq 2$. When $\lambda \leq \frac{n \log n}{4A}$ (which ensures that $A^{-1}n\lambda^{-1} \log \lambda \geq 2$ for all $n$ exceeding some $n_A$), we can therefore write
\begin{align*}
\pr{S_n \leq \lambda^{-1}n \log n} \leq \pr{S_{A^{-1}n\lambda^{-1} \log \lambda} \leq \lambda^{-1}n \log n}^{A\lambda (\log \lambda)^{-1}} \leq \left(\frac{3}{4}\right)^{A\lambda (\log \lambda)^{-1}}.
\end{align*}
This proves the claimed result for $1 \leq \lambda \leq \frac{n \log n}{4A}$. We can then extend to $\lambda \leq n \log n$ in place of $\lambda \leq \frac{n \log n}{4A}$ by replacing $c'$ in the main statement with $\frac{c'}{4A}$.
\item Similarly to $(ii)$ but this time by the law of large numbers, we can similarly choose $A>0$ such that for all $1 \leq \lambda \leq \frac{n}{A}$,
\begin{align*}
\pr{S_n \leq \lambda^{-1}n} \leq \pr{S_{A^{-1}n\lambda^{-1}} \leq \lambda^{-1}n}^{A\lambda} \leq c^{A\lambda},
\end{align*}
and extend to all $\lambda \leq n$ as in $(ii)$.
\end{enumerate}
\end{proof}
\end{lem}

\begin{lem}\label{lem:Tauberian at 0 app}
Let $X$ be a non-negative random variable, and suppose that $\pr{X \geq x} \sim cx^{-\beta}$ as $x \to \infty$ for some $\beta \in [0,1)$ and $c \in (0, \infty)$. Then there exists a constant $c' \in (0, \infty)$ such that $1-\E{e^{-\theta X}} \sim c' \theta^{\beta}$ as $\theta \downarrow 0$.
\end{lem}
\begin{proof}
This is a standard Tauberian theorem, for example see \cite[Section IV, Theorem 8.2]{KorevaarTaub}.
\end{proof}

%
%

\end{appendix}

\bibliographystyle{plain}
\bibliography{biblio}

\end{document}